\documentclass[11pt]{article} %{{{
\usepackage{latexsym, amsfonts, amssymb, amsmath, amsthm, mathtools}
\usepackage{graphicx}
\usepackage[colorlinks=true,
            linkcolor=blue,
            filecolor=magenta,
            urlcolor=blue,
            citecolor=magenta,
            pdftitle={Global existence and asymptotic behavior in chemotaxis models with signal-dependent sensitivity and logistic-type source},
            pdfauthor={Le Chen, Ian Ruau, Wenxian Shen},
            pdfkeywords={global existence, boundedness, persistence, chemotaxis, signal-dependent sensitivity, logistic source, stability},
            pdfpagemode=FullScreen,
            breaklinks=true,
           ]{hyperref}
\usepackage[left=1in, right=1in, top=1in, bottom=1in]{geometry}
\usepackage[usenames,dvipsnames]{color}
\usepackage{ulem}

\usepackage{pgfplots}
\usepackage{tikz}
\usetikzlibrary{shapes.geometric, positioning, arrows, intersections, patterns, angles, quotes}

\newcommand{\Norm}[1]{\left\| #1 \right\|}
\newcommand{\norm}[1]{||#1||}

\usepackage{amsrefs}

\BibSpec{article}{%
    +{}  {\PrintAuthors}                {author}
    +{,} { \textit}                     {title}
    +{.} { }                            {part}
    +{:} { \textit}                     {subtitle}
    +{,} { \PrintContributions}         {contribution}
    +{.} { \PrintPartials}              {partial}
    +{,} { }                            {journal}
    +{}  { \textbf}                     {volume}
    +{}  { \PrintDatePV}                {date}
    +{,} { \issuetext}                  {number}
    +{,} { \eprintpages}                {pages}
    +{,} { }                            {status}
    +{,} { \url}                        {url}
    +{,} { \PrintDOI}                   {doi}
    +{,} { available at \eprint}        {eprint}
    +{}  { \parenthesize}               {language}
    +{}  { \PrintTranslation}           {translation}
    +{;} { \PrintReprint}               {reprint}
    +{.} { }                            {note}
    +{.} {}                             {transition}
    +{}  {\SentenceSpace \PrintReviews} {review}
}

\newtheorem{theorem}{Theorem}[section] % Theorem is numbered by section

\newtheorem{corollary}[theorem]{Corollary}

\newtheorem{definition}{Definition}[section]

\newtheorem{lemma}{Lemma}[section]

\newtheorem{proposition}{Proposition}[section]

\theoremstyle{remark}
\newtheorem{remark}{Remark}[section]
\theoremstyle{remark}

\numberwithin{equation}{section}
\numberwithin{figure}{section}
\numberwithin{table}{section}

\def\p{\partial}

\def\R{\mathbb R}

\begin{document}

\title{Chemotaxis models with signal-dependent sensitivity and a logistic-type
source, I: Boundedness and global existence}

\author{%
\begin{tabular}{ccc}
  Le Chen                  & Ian Ruau                 & Wenxian Shen \\
  \url{lzc0090@auburn.edu} & \url{iir0001@auburn.edu} & \url{wenxish@auburn.edu}
\end{tabular} \bigskip \\
Department of Mathematics and Statistics\\
Auburn University, Auburn, AL 36849, USA}
\date{}
%\date{\today}
\maketitle

\begin{abstract}
  This series of papers is concerned with the boundedness, global existence, and
  asymptotic behavior of positive classical solutions to the following
  chemotaxis model with signal-dependent sensitivity and a logistic-type source:
  \begin{equation}\label{E:main-abstract-eq}\tag{CM}
  \begin{dcases}
    u_t=\Delta u- \chi_0\nabla \cdot\left(\frac{u^m}{(1+v)^\beta}\nabla v\right)+au-b u^{1+\alpha}, & x \in \Omega, \\
    0=\Delta v-\mu v+\nu u^\gamma,                                                                  & x \in \Omega, \\
    \frac{\partial u}{\partial n}=\frac{\partial v}{\partial n}=0,                                  & x \in \partial\Omega.
  \end{dcases}
  \end{equation}
  Here, $\Omega$ is a bounded smooth domain in $\mathbb{R}^N$. The function $u$
  denotes the population density of the biological species, and $v$ denotes the
  concentration of the chemical agent. The parameters $\alpha, \gamma, m, \mu,
  \nu$ are positive, $\chi_0$ is a real number, and $a,b, \beta$ are
  nonnegative. In the present part, we investigate the boundedness and global
  existence of positive classical solutions of \eqref{E:main-abstract-eq} from
  three different perspectives: the role played by the negative chemotaxis
  sensitivity coefficient $\chi_0$, the role of the nonlinear cross diffusion
  rate $\frac{u^m}{(1+v)^\beta}$, and the role of the logistic-type source
  $u(a-b u^\alpha)$. We prove that all positive classical solutions
  of~\eqref{E:main-abstract-eq} stay bounded provided that the chemotaxis
  sensitivity $\chi_0$ is negative, or the nonlinear cross diffusion rate
  $\frac{u^m}{(1+v)^\beta}$ is relatively weak, or the logistic source
  $u(a-bu^\alpha)$ is relatively strong. We also prove that when $m\ge 1$, the
  boundedness of a positive classical solution implies its global existence. In
  part II, we will study the asymptotic behavior of globally defined positive
  classical solutions, including uniform persistence of globally defined
  positive solutions and stability of positive constant equilibria.

  Note that the decay of the sensitivity function $\chi(v) \coloneqq
  \frac{\chi_0}{(1+v)^\beta}$ ($\beta>0$) as $v \to \infty$ has a damping effect
  on the ability of the biological species to move towards the signal when its
  concentration is high. However, the signal-dependent sensitivity complicates
  the analysis of boundedness and global existence
  for~\eqref{E:main-abstract-eq}. This paper develops novel ideas and techniques
  for the study of boundedness and global existence in
  \eqref{E:main-abstract-eq}, some of which can also be applied to other
  chemotaxis models. The results established in this paper provide partial
  evidence that the signal-dependent sensitivity $\chi(v) =
  \frac{\chi_0}{(1+v)^\beta}$ with large $\beta$ has less of an effect on the
  evolution of the population density of the biological species. For example, we
	  prove that if $\beta > \max\{1,\frac{1}{2}+ \frac{\chi_0}{4} \max\{2,\gamma
	  N\}\}$ and $m = 1$, then any positive classical solution
	  of~\eqref{E:main-abstract-eq} exists globally. Several existing results for
	  some special cases of \eqref{E:main-abstract-eq} are recovered.

\end{abstract}

\medskip
\medskip


\medskip


{\hypersetup{linkcolor=black}
\tableofcontents
}

\section{Introduction}

\subsection{Overview}

In the current series of papers, we study the boundedness, global existence,
and asymptotic behavior of classical solutions to chemotaxis models of the
following form on a bounded, connected, and smooth domain $\Omega \subset
\mathbb{R}^N$:
\begin{align}\label{E:main-eq0}
\begin{dcases}
  u_t = \Delta u - \nabla \cdot (u^m \chi(v) \nabla v) + au - b u^{1+\alpha}, & x \in \Omega,         \\
  \tau v_t = \Delta v - \mu v + \nu u^\gamma,                                 & x \in \Omega,         \\
  \frac{\partial u}{\partial n} = \frac{\partial v}{\partial n} = 0,          & x \in \partial\Omega.
\end{dcases}
\end{align}
This is a parabolic-elliptic system when $\tau = 0$, and a parabolic-parabolic
system when $\tau > 0$. Here, the function $\chi(v)$, known in the literature
as the \textit{chemotactic sensitivity function}, is given by
\begin{equation}\label{E:chi-function-eq}
  \chi(v) = \frac{\chi_0}{(c + v)^\beta}.
\end{equation}
In equations~\eqref{E:main-eq0} and~\eqref{E:chi-function-eq}, $m, \alpha,
\gamma, \mu$, and $\nu$ are positive constants, $\chi_0$ is a real number, and
$a$, $b$, $c$, $\beta$, and $\tau$ are nonnegative constants.

System~\eqref{E:main-eq0} describes the evolution of a biological species
influenced by a chemical substance or signal produced by the species itself at
the rate $\nu u^\gamma$. In this context, $u(t,x)$ represents the population
density of the biological species, and $v(t,x)$ represents the density of the
chemical substance or signal. When $\beta = 0$, $\chi(v)$ simplifies to the
constant chemotactic sensitivity function $\chi(v) \equiv \chi_0$. When $c = 0$
and $\beta = 1$, $\chi(v) = \chi_0/v$, which reflects a singular behavior near
$v = 0$, and an inhibition of chemotactic migration at high signal
concentrations. This type of sensitivity was first proposed
in~\cite{keller.segel:71:traveling} based on the \textit{Weber--Fechner law} of
stimulus perception. When $c>0$ and $\beta>0$, the sensitivity function
$\chi(v)$ takes into account the partial loss of the signal substance through
its binding (see~\cite{viglialoro.woolley:18:boundedness, baghaei:24:global,
black.lankeit.ea:20:stabilization, winkler:10:absence}). The coefficient
$\chi_0>0$ indicates that the signal is a \textit{chemoattractant} signal, while
$\chi_0<0$ indicates that the signal is \textit{chemorepellent}. The reaction
term $a u - b u^{1+\alpha}$ governs the local dynamics of the biological species
and is known as a \textit{logistic source} when both $a$ and $b > 0$.
System~\eqref{E:main-eq0} with $a = b = 0$ is referred to as a \textit{minimal
chemotaxis model}. The parameter $\tau\ (\ge 0)$ is associated with the
diffusion rate of the chemical substance. When $\tau = 0$, it corresponds to the
case in which the chemical substance diffuses very rapidly.
System~\eqref{E:main-eq0} with $\tau = 0$ (resp. $\tau > 0$) is referred to as a
\textit{parabolic-elliptic} (resp. \textit{parabolic-parabolic}) chemotaxis
system. When $\beta > 0$, System~\eqref{E:main-eq0} is known as a chemotaxis
system with signal-dependent sensitivity. In the literature, a chemotaxis system
is also referred to as a \textit{Keller-Segel system}, named after the
pioneering works of E. F. Keller and L. A.
Segel~\cite{keller.segel:70:initiation, keller.segel:71:model,
keller.segel:71:traveling}.

Central problems in studying~\eqref{E:main-eq0} include determining whether
solutions with given initial conditions exist globally and understanding the
asymptotic behavior of globally defined bounded solutions of~\eqref{E:main-eq0}.
Numerous studies have addressed these issues for various special cases. For
example, the works~\cite{herrero.medina.ea:97:finite-time,
herrero.velazquez:96:singularity, nagai:01:blowup, nagai.senba:98:global,
tello.winkler:07:chemotaxis} studied~\eqref{E:main-eq0} with $\chi(v)~\equiv
\chi_0$ and $m = \alpha = \gamma = 1$, which is given by:
\begin{equation}\label{E:special-eq1}
\begin{dcases}
  u_t = \Delta u - \chi_0 \nabla \cdot (u \nabla v) + au - b u^2,    & x \in \Omega, \\
  \tau v_t = \Delta v - \mu v + \nu u,                               & x \in \Omega, \\
  \frac{\partial u}{\partial n} = \frac{\partial v}{\partial n} = 0, & x \in \partial \Omega.
\end{dcases}
\end{equation}
In this case, if the system is minimal, i.e., $a = b = 0$, it is known that when
$N \ge 2$, finite-time blow-up of positive solutions occurs under certain
conditions on the mass and the moment of the initial data (see
\cite{herrero.medina.ea:97:finite-time, herrero.velazquez:96:singularity,
nagai:01:blowup, nagai.senba:98:global}, etc.). If $a, b > 0$, it is known that
finite-time blow-up does not occur if $N \le 2$, or if $N > 3$ and $b$ is large
relative to $\chi_0$ (see~\cite{issa.shen:17:dynamics, issa.shen:20:pointwise,
tello.winkler:07:chemotaxis, winkler:10:boundedness}, etc.).  Hence, the
finite-time blow-up phenomenon in~\eqref{E:special-eq1} is suppressed to some
extent by the logistic source. However, it remains an open question whether
solutions of~\eqref{E:special-eq1} are always global.

There are several studies of~\eqref{E:main-eq0} in the case that $\chi(v) \equiv
\chi_0$, $\beta = 0$, $\tau = 0$, and $m$, $\alpha$, $\gamma > 0$, which is
given by:
\begin{equation}\label{E:special-eq3}
  \begin{dcases}
    u_t = \Delta u - \chi_0 \nabla \cdot \left(u^m \nabla v\right) + au - b u^{1+\alpha}, & x \in \Omega, \\
    0 = \Delta v - v + u^\gamma,                                                          & x \in \Omega, \\
    \frac{\partial u}{\partial n} = \frac{\partial v}{\partial n} = 0,                    & x \in \partial \Omega
  \end{dcases}
\end{equation}
(see \cite{galakhov.salieva.ea:16:on, hong.tian.ea:20:attraction-repulsion,
hu.tao:17:boundedness, zheng:24:on, xiang:19:dynamics}, etc.). For example, the
work~\cite{galakhov.salieva.ea:16:on} studied~\eqref{E:special-eq3} with
$\chi_0>0$, $m, \gamma \ge 1$, $\alpha > 0$, $a = b > 0$, and $\mu = \nu = 1$.
It is proved in~\cite{galakhov.salieva.ea:16:on} that if
\begin{subequations}\label{E:Cond-GalSal}
  \begin{equation}\label{E:Cond-GalSal-1}
    \alpha > m + \gamma - 1 \quad \text{or}
  \end{equation}
  \begin{equation}\label{E:Cond-GalSal-2}
    \alpha = m + \gamma - 1 \quad \text{and} \quad b > \frac{N\alpha - 2}{2(m - 1) + N\alpha} \chi_0,
  \end{equation}
\end{subequations}
then nonnegative classical solutions exist globally. Moreover, if
\begin{equation}\label{E:Cond-GalSal-3}
  \alpha \ge m + \gamma - 1 \quad \text{and} \quad b > 2\chi_0,
\end{equation}
then every globally defined positive classical solution converges to the unique
positive equilibrium solution $(u,v) \equiv (1,1)$. Note that the condition
\eqref{E:Cond-GalSal-2} with $m = \alpha = \gamma = 1$ becomes
\[
  b>\frac{N-2}{N}\chi_0,
\]
which is obtained in~\cite{tello.winkler:07:chemotaxis}.

The works~\cite{biler:99:global, black:20:global, ding.wang.ea:19:global,
fujie:15:boundedness, fujie.senba:16:global, fujie.winkler.ea:14:blow-up,
kurt.shen:20:finite-time, kurt.shen:23:chemotaxis, kurt.shen.ea:24:stability,
nagai.senba:98:global, winkler:11:global, zhao.zheng:18:global} investigated
system~\eqref{E:main-eq0} with singular sensitivity $\chi(v) = \chi_0/v$ and $m
= \alpha = \gamma = 1$, which is given by:
\begin{equation}\label{E:special-eq2}
  \begin{dcases}
    u_t = \Delta u - \chi_0 \nabla \cdot \left(\frac{u}{v} \nabla v\right) + au - b u^2, & x \in \Omega, \\
    \tau v_t = \Delta v - \mu v + \nu u,                                                 & x \in \Omega, \\
    \frac{\partial u}{\partial n} = \frac{\partial v}{\partial n} = 0,                   & x \in \partial \Omega.
  \end{dcases}
\end{equation}
It is known that when $N = 2$, finite-time blow-up does not occur
for~\eqref{E:special-eq2} (see~\cite{fujie.senba:16:global,
fujie.winkler.ea:14:blow-up}). Note that when $a = b = 0$ and $N
\ge 3$, finite-time blow-up may occur for~\eqref{E:special-eq2} (see
\cite{nagai.senba:98:global}). When $a, b > 0$ and $N \ge 3$, positive solutions
of the parabolic-elliptic system exist globally under the following biologically
meaningful conditions: $a$ is large relative to $\chi_0$ and the initial
distribution of $u$ is not small (see~\cite{kurt.shen:20:finite-time}). In
general, it remains an open question whether finite-time blow-up does not occur
when $a, b > 0$ and $N \ge 3$. See also
\cite{le.kurt:25:global, kurt:25:boundedness, zhang.mu.ea:25:global,
zhao:23:boundedness, zhao:24:global, zhao.xiao:24:global} for the study
of~\eqref{E:main-eq0} in the case when $c=0$, $\tau = 0$, $m=1$, and $\alpha$,
$\beta$, $\gamma > 0$.

Among others, the works~\cite{black.lankeit.ea:20:stabilization,
mizukami.yokota:17:unified, viglialoro.woolley:18:boundedness,
winkler:10:absence} studied system~\eqref{E:main-eq0} with $\chi(v) =
\frac{\chi_0}{(1+v)^\beta}$ ($\beta>0$) and $m = \alpha = \gamma = 1$, which is
given by
\begin{equation}\label{E:special-eq4}
  \begin{dcases}
    u_t=\Delta u- \chi_0\nabla \cdot\left(\frac{u}{(1+v)^\beta}\nabla v\right)+au-b u^{2}, & x \in \Omega, \\
    \tau v_t=\Delta v-\mu v+\nu u,                                                         & x \in \Omega, \\
    \frac{\partial u}{\partial n}=\frac{\partial v}{\partial n}=0,                         & x \in \partial\Omega.
  \end{dcases}
\end{equation}
Note that the sensitivity function $\chi(v) = \frac{\chi_0}{(1+v)^\beta}$ with
$\beta>0$ stays bounded for all $v\ge 0$ and that the decay of $\chi(v)$ as $v
\to \infty$ has a damping effect on the ability of the biological species to
move towards the signal when its concentration is high. Intuitively, one expects
that positive classical solutions of \eqref{E:special-eq4} exist globally under
some conditions which are weaker than those for the global existence of positive
classical solutions of \eqref{E:special-eq1} (resp. \eqref{E:special-eq2}).
Moreover, one might expect that such sensitivity could prevent finite-time
blow-up in~\eqref{E:main-eq0}. The works
\cite{black.lankeit.ea:20:stabilization, mizukami.yokota:17:unified,
winkler:10:absence} show that positive classical solutions of~\eqref{E:special-eq4}
exist globally in the case that $\tau = 1$, $\beta > 1$, and $a = b = 0$, which
implies that the nonlinear sensitivity $\chi(v) = \frac{\chi_0}{(1+v)^\beta}$
with $\beta>1$ prevents finite-time blow-up in system~\eqref{E:special-eq4} when
it is of the parabolic-parabolic type. Fewer studies have focused on the global
existence of classical solutions of system~\eqref{E:special-eq4} when it is of
the parabolic-elliptic type.

The objective of this series of papers is to investigate the boundedness, global
existence, and asymptotic behavior of positive classical solutions
of~\eqref{E:main-eq0} with $c = 1$ (as given in~\eqref{E:chi-function-eq}), and
$\tau = 0$ (as specified in~\eqref{E:main-eq0}), that is,
\begin{equation}\label{E:main-PE}
\begin{dcases}
  u_t=\Delta u- \chi_0\nabla \cdot\left(\frac{u^m}{(1+v)^\beta}\nabla v\right)+au-b u^{1+\alpha}, & x \in \Omega,         \\
  0=\Delta v-\mu v+\nu u^\gamma,                                                                  & x \in \Omega,         \\
  \frac{\partial u}{\partial n}=\frac{\partial v}{\partial n}=0,                                  & x \in \partial\Omega,
\end{dcases}
\end{equation}
 where the parameters of our system~\eqref{E:main-PE} satisfy the following
 conditions:
\begin{align}\label{E:parameters}
  \alpha, \gamma, m, \mu, \nu > 0 , \quad
  a, b, \beta \ge 0, \quad\text{and} \quad
  \chi_0 \in \R.
\end{align}
In this part, we focus on the boundedness and global existence of positive
classical solutions of~\eqref{E:main-PE}. We investigate the boundedness and
global existence of positive classical solutions of~\eqref{E:main-PE} from three
different angles: one from the angle of the role played by negative chemotaxis
sensitivity coefficient $\chi_0$, one from the angle of the role played by the
nonlinear cross diffusion rate $\frac{u^m}{(1+v)^\beta}$, and one from the angle
of the role played by the logistic-type source $u(a-b u^\alpha)$. In part II, we
will study the asymptotic behavior of globally defined positive classical
solutions of~\eqref{E:main-PE}.

\smallskip

This paper has several features not present in the existing literature. For
example, most studies in the literature on the global existence of classical
solutions of~\eqref{E:main-PE} (even with $\beta = 0$) focus on the case where
$m \ge 1$ and/or $\gamma \ge 1$. In this paper, we allow $m < 1$ and $\gamma <
1$. Note that $m < 1$ may induce some singularity at $u = 0$.  As mentioned
above, the decay of the sensitivity function $\chi(v) \coloneqq
\frac{\chi_0}{(1+v)^\beta}$ ($\beta>0$) as $v \to \infty$ has a damping effect
on the ability of the biological species to move towards the signal when its
concentration is high. However, the signal-dependent sensitivity complicates the
analysis of global existence and asymptotic dynamics for~\eqref{E:main-PE}. We
develop novel ideas and techniques for the study of boundedness and global
existence in~\eqref{E:main-PE}, some of which can also be applied to other
chemotaxis models (see subsection~\ref{SS:remarks} for some difficulties that
arise in the study of~\eqref{E:main-PE} and for some novel ideas and techniques
developed in this paper).

\subsection{Definitions and notation}\label{SS:Definition}

In this subsection, we present some notation to be used throughout the paper,
introduce the definition of classical solutions of~\eqref{E:main-PE} and present
a proposition on the local existence of classical solutions.

\smallskip

For a given function $u \in C(\overline\Omega)$, we may use $\|u\|_\infty$ for
$\|u\|_{C(\overline\Omega)} = \sup_{x\in\Omega}|u(x)|$.  For any $p \ge 1$ and
$u\in L^p(\Omega)$, we use $\Norm{u}_p$ to denote the $L^p$-norm of $u$, i.e.,
$\Norm{u}_p \coloneqq \left(\int_\Omega |u(x)|^p dx\right)^{1/p}$. For given
$\theta_1, \theta_2 \in (0,1)$ and an interval $I\subset\mathbb{R}$, denote
the \textit{H\"older space}
\begin{equation*}
  C^{\theta_1,\theta_2}\left(I\times\overline\Omega\right)
  \coloneqq \left\{u\in C(I\times\overline\Omega)\, \big|\, \Norm{u}_{C^{\theta_1,\theta_2}(I\times\overline\Omega)}<\infty \right\},
\end{equation*}
where the \textit{H\"older norm} is defined by
\begin{equation*}
  \Norm{u}_{C^{\theta_1,\theta_2}\left(I\times\overline\Omega\right)}
  \coloneqq \mathop{\sup_{(t,x)\in I\times\overline \Omega}}|u(t,x)|+ \mathop{\sup_{(t_1,x_1),(t_2,x_2)\in I\times\overline \Omega}}_{(t_1,x_1) \not = (t_2,x_2)}
  \frac{|u(t_1,x_1)-u(t_2,x_2)|}{|t_1-t_2|^{\theta_1}+|x_1-x_2|^{\theta_2}}.
\end{equation*}
We use the standard notation $C^{k,\ell}\left(I \times \overline{\Omega}
\right)$, where $k, \ell \in \{0,1,2,\dots\}$, to denote the space of functions
possessing $k$ continuous derivatives in time and $\ell$ continuous derivatives
in space.

Throughout the paper, we assume that $\Omega \subset \mathbb{R}^N$ is a bounded
smooth domain. The number $C$ denotes a generic constant which is independent of
the solutions of~\eqref{E:main-PE}, but may depend on the parameters under
consideration and may be different at different places.

\begin{definition}\label{D:classical-sol}
  \begin{itemize}

\item[(1)] A function $(u(t,x),v(t,x))$ is called a {\rm classical solution} of
  the parabolic-elliptic system~\eqref{E:main-PE} on the time interval $(0,T)$
  if
  \begin{gather*}
    u(\cdot,\cdot) \in C^{1,2}\left( (0,T) \times \Omega\right)\cap C^{0,1}\left( (0,T) \times \overline{\Omega}\right), \\
    v(\cdot,\cdot) \in C^{0,2}\left( (0,T) \times \Omega\right)\cap C^{0,1}\left( (0,T) \times \overline{\Omega}\right),
  \end{gather*}
  and $(u(t,x),v(t,x))$ satisfies~\eqref{E:main-PE} for all $t\in (0,T)$.

\item[(2)] A given function $(u(t,x), v(t,x))$ is said to be a {\rm global
  classical solution} of~\eqref{E:main-PE} if it is a classical solution
  of~\eqref{E:main-PE} on $(0, \infty)$. A global classical solution is said to
  be {\rm positive} (resp. {\rm bounded}) if $\inf_{x\in\Omega}u(t,x) > 0$ for
  any $t \in (0, \infty)$ (resp.~${\limsup_{t\to\infty}
  \|u(t,\cdot)\|_\infty<\infty}$).

\end{itemize}
\end{definition}

By standard arguments, it can be proved that for any $u_0$
satisfying~\eqref{E:initial-cond-PE} below, there exists~${T_{\max}(u_0) \in (0,
\infty]}$ such that~\eqref{E:main-PE} has a unique classical solution, denoted
by $\left(u(t,x;u_0), v(t,x;u_0)\right)$, with the initial condition $u(0,x;u_0)
= u_0(x)$ on the interval $(0, T_{\max}(u_0))$. To be more precise, we have the
following proposition:

\begin{proposition}[Local existence]\label{P:local-existence}
  % Consider the parabolic-elliptic system~\eqref{E:main-PE}.
  For any given $u_0$ satisfying
  \begin{equation}\label{E:initial-cond-PE}
    u_0\in C(\overline{\Omega}) \quad \text{and}\quad  \inf_{x\in \Omega}u_0(x) > 0,
  \end{equation}
  there is $T_{\max}(u_0)\in (0,\infty]$ such that the parabolic-elliptic
  system~\eqref{E:main-PE} admits a unique classical solution $(u(t,x;u_0),
  v(t,x;u_0))$ on $(0, T_{\max}(u_0))$ satisfying that
  \begin{align}\label{E:local-PE-eq0}
    \lim_{t\to 0_+}\Norm{u(t,\cdot;u_0)-u_0(\cdot)}_{\infty} = 0 \quad \text{and} \quad
    u(t,x) > 0, \quad \text{for all $(t,x)\in [0, T_{\max}(u_0))\times\overline\Omega$,}
  \end{align}
  and
  \begin{equation}\label{E:local-PE-eq1}
    u(\cdot,\cdot;u_0) \in C^{1,2} \left( (0, T_{\max}(u_0))\times \overline{\Omega}\right), \quad
    v(\cdot,\cdot;u_0) \in C^{0,2} \left( (0, T_{\max}(u_0))\times \overline{\Omega}\right).
  \end{equation}
  Moreover, if $T_{\max}(u_0) < \infty$, then either
  \begin{equation}\label{E:local-Alternative}
    \limsup_{t \nearrow T_{\max}(u_0)} \sup_{x\in \Omega} u(t, x;u_0) = \infty \quad {\rm or}\quad
    \liminf_{t \nearrow T_{\max}(u_0)}\inf_{x\in\Omega} u(t,x;u_0) = 0.
  \end{equation}
  If $T_{\max}(u_0)<\infty$ and $m\ge 1$, then
  \begin{equation}\label{E:local-infty}
    \limsup_{t \nearrow T_{\max}(u_0)}  \sup_{x\in \Omega} u(t, x;u_0) = \infty.
  \end{equation}
\end{proposition}

We will provide an outline of the proof of Proposition~\ref{P:local-existence}
in Section~\ref{SS:Local-Exist}. \medskip

Note that by regularity for
elliptic equations (see Theorem~12.1 of~\cite{amann:84:existence}), for any $1 <
p < \infty$, there is $C_{N,p} > 0$ such that
\begin{equation}\label{E:main-inequality-eq-4-0}
  \int_{\Omega}|D^2 v|^p\le C_{N,p}\int_\Omega \Big(|\Delta v|^p+v^p\Big)\quad \text{for all $v\in W^{2,p}(\Omega)$}.
\end{equation}
Define
\begin{equation}\label{E:C-star}
  C_{N,p}^* \coloneqq \inf\bigg\{C_{N,p}>0\,\big|\, \text{$C_{N, p}$ satisfies \eqref{E:main-inequality-eq-4-0}}\bigg\}.
\end{equation}
Then~\eqref{E:main-inequality-eq-4-0} holds with $C_{p, N}$ being replaced by
$C_{N, p}^*$. Moreover, by the Riesz-Thorin Interpolation Theorem (see Theorem
6.27 in~\cite{folland:99:real}), $C_{N, p}^*$ is bounded for $p$ in any compact
interval of $(1, \infty)$. We will need the following constants:
\begin{equation}\label{E:M-star}
  M^*(N,p,\mu,\nu) \coloneqq \nu^{p}\left[
    \frac{8^{p}}{p}\,C_{N,p}^*\!\left(2^{p}+\frac{1}{\mu^{p}}\right)
  + \frac{2^{2p}}{(p-1)\,p^{p}} \right],
\end{equation}
\begin{equation}\label{E:K}
  K(N, \alpha, \gamma, \mu, \nu) \coloneqq \liminf_{q\to q_*,\, q>q_*}
  \left[M^*\!\left(N,\, \frac{q+\alpha}{\gamma},\, \mu,\, \nu\right)\right]^{\!\frac{\gamma}{q+\alpha}},
  \quad \text{with $q_* \coloneqq \max\left\{1,\frac{N\alpha}{2}\right\}$,}
\end{equation}
% \hfil\vfil\eject
and
\begin{align}\label{E:Psi-Theta}
  \Psi_\beta \coloneqq \left(\frac{\beta}{1+\beta}\right)^{1+\beta}
  \quad \text{and} \quad
  \Theta_\beta \coloneqq \beta^\beta\, (1+\beta)^{-(1+\beta)}.
\end{align}
Observe that (see, also, Fig.~\ref{fig:Psi-Theta}):
\[
  \Psi_0=0,\quad  \lim_{\beta\to\infty} \Psi_\beta=e^{-1},
  \quad \text{and} \quad
  \lim_{\beta \to 0_+} \Theta_\beta = 1,\quad \lim_{\beta\to\infty} \Theta_\beta=0.
\]

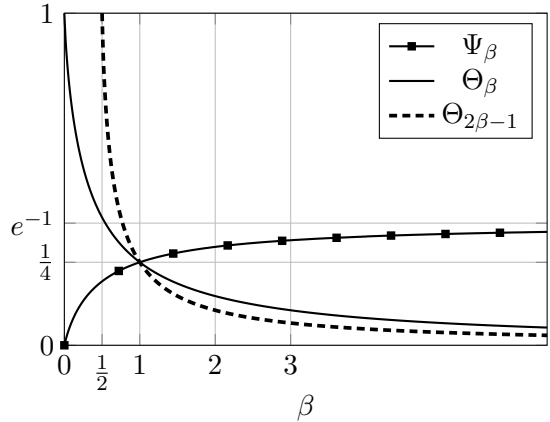
\begin{figure}[htpb]
  \centering
  \begin{tikzpicture}[baseline=(current bounding box.center)]
  \begin{axis}[
    width=8cm, height=6cm,
    xlabel={$\beta$},
    domain=0:6.4, samples=400,
    grid=both, legend pos=north east,
    enlargelimits=false, clip=false,
    ytick={0, 1/4, exp(-1), 1},
    yticklabels={$0$, $\tfrac{1}{4}$, $e^{-1}$, $1$},
    xtick={0, 1/2, 1, 2, 3},
    xticklabels={$0$, $\tfrac{1}{2}$, $1$, $2$, $3$},
    every axis plot/.append style={mark=none, smooth, black, thick}
  ]

  % 1) (b/(1+b))^(1+b)
  \addplot[mark=square*, mark size=1.2pt, mark repeat = 45] {(x/(1+x))^(1+x)};
  \addlegendentry{$\Psi_\beta$}

  % 2) b^b/(1+b)^(1+b)
  \addplot[solid] { (x^x)/((1+x)^(1+x)) };
  \addlegendentry{$\Theta_\beta$}

  % 3) x^x/(1+x)^{1+x} /. x -> 2b - 1
  %    = (2b-1)^{2b-1} / (2b)^{2b}, valid for b>1/2
  \addplot[densely dashed, domain=0.5:6.4, line width=1.4pt] { ((2*x - 1)^(2*x - 1)) / ((2*x)^(2*x)) };
  \addlegendentry{$\Theta_{2\beta-1}$}

  \end{axis}
  \end{tikzpicture}

  \caption{The constants $\Psi_\beta$, $\Theta_\beta$, and $\Theta_{2\beta-1}$
  viewed as functions of $\beta$.}

  \label{fig:Psi-Theta}
\end{figure}

Throughout this paper, we assume that $u_0$ satisfies~\eqref{E:initial-cond-PE}.
We may drop $t,x,u_0$ in $u(t,x;u_0)$ and $v(t,x; u_0)$, and drop $u_0$ in
$T_{\max}(u_0)$ if no confusion occurs. \medskip

\subsection{Statements of the main results}\label{SS:Main}

In this subsection, we state our main results. In the following,
$(u(t,x;u_0),v(t,x;u_0))$ always denotes the solutions of the parabolic-elliptic
system~\eqref{E:main-PE} starting with initial function $u_0$ that
satisfies~\eqref{E:initial-cond-PE}.

Our first theorem addresses the boundedness and global existence of solutions
of~\eqref{E:main-PE} with $\chi_0 \le 0$, which will be proved in
Section~\ref{S:negative-sensitivity}.

\begin{theorem}[Boundedness and global existence with negative chemotaxis
  sensitivity]\label{T:negative-sensitivity}

  Assume $\chi_0\le 0$.
  \begin{itemize}

    \item[(1)] If $a, b > 0$, then
      \begin{align}\label{E:bounded-max-eq1}
        \Norm{u(t,\cdot;u_0)}_\infty
        \le \max\left\{\Norm{u_0}_\infty,\: \left(\frac{a}{b}\right)^{1/\alpha} \right\}
        \qquad \text{for all $t \in (0,T_{\max}(u_0))$.}
      \end{align}
      Moreover, if $m \ge 1$, then $T_{\max}(u_0) = \infty$.

    \item[(2)] If $a = b = 0$, then
      \begin{equation}\label{E:bounded-max-eq0}
        \limsup_{t\to T_{\max}-} \Norm{u(t,\cdot;u_0)}_\infty
        \le \Norm{u_0}_\infty \qquad \text{for all $t \in (0,T_{\max}(u_0))$.}
      \end{equation}
      Moreover, if $m \ge 1$, then $T_{\max}(u_0) = \infty$.
  \end{itemize}
\end{theorem}

Our second theorem is on the boundedness and global existence of solutions of
\eqref{E:main-PE} with weak nonlinear cross diffusion in the sense that $0 < m
\le 1$ and $\beta \ge 1$, which will be proved in
Section~\ref{S:weak-cross-diffusion}.

\begin{theorem}[Boundedness and global existence with weak nonlinear cross diffusion]
  \label{T:weak-cross-diffusion}

  Assume that $a, b \ge 0$ and $\beta \ge 1$.
  \begin{itemize}

    \item[(1)] If $0<m<1$, then
      \begin{align}\label{E:bounded-finite}
        \limsup_{t\to T_{\max}-} \Norm{u(t,\cdot;u_0)}_\infty<\infty.
      \end{align}

    \item[(2)] If $m=1$ and $\chi_0 < \dfrac{2(2\beta-1)}{\max\{2, \gamma N\}}$,
      then~\eqref{E:bounded-finite} holds and $T_{\max} (u_0) = \infty$.

  \end{itemize}
\end{theorem}

Our third theorem is on the boundedness and global existence of solutions
of~\eqref{E:main-PE} with relatively strong logistic source, which will be
proved in Section~\ref{S:strong-logistic-source}.

\begin{theorem}[Boundedness and global existence with relatively strong logistic source]
\label{T:strong-logistic-source}

  Assume that $a, b > 0$.
  \begin{itemize}

    \item[(1)] (Boundedness) Assume that $m > 0$.
      Property~\eqref{E:bounded-finite} holds given one of the following
      conditions:

      \begin{itemize}

        \item[(i)] $\beta \ge 0$, $\alpha > m + \gamma - 1$;

        \item[(ii)] $\beta\ge 1/2$, and $\alpha > 2m + \gamma - 2$;

        \item[(iii)] $\beta\ge 0$, $\alpha = m + \gamma - 1$, and
          \begin{align}\label{E:cond-chi-eq1}
            \chi_0 < \frac{\big((N\alpha-2)_+ + 2m\big) b}{(N\alpha-2)_+ \left(\nu + {\Psi_\beta}\, K \right) },
          \end{align}
          where the right-hand side is understood as $+\infty$ when
          $(N\alpha-2)_+ = 0$, $K = K(N,\alpha,\gamma,\mu,\nu)$ and
          $\Psi_\beta$ are defined in~\eqref{E:K} and~\eqref{E:Psi-Theta},
          respectively;

        \item[(iv)] $\beta\ge 1/2$, $\alpha = 2m + \gamma - 2$, and
          \begin{equation}\label{E:cond-chi-eq2}
            \chi_0 < \left( \frac{8\,b}{(N\alpha-2)_+   \Theta_{2\beta-1}\, K}\right)^{1/2},
          \end{equation}
          where the right-hand side is understood as $+\infty$ when
          $(N\alpha-2)_+ = 0$, $K = K(N,\alpha,\gamma,\mu,\nu)$ and
          $\Theta_\beta$ are defined in~\eqref{E:K} and~\eqref{E:Psi-Theta},
          respectively;

      \end{itemize}

    \item[(2)] (Global existence). If $m \ge 1$, then $T_{\max}(u_0) = \infty$
      in any of the above cases.

  \end{itemize}

\end{theorem}

Note that the regions satisfying the conditions (i)--(iv) may overlap.  For
example, if $m\ge 1$, then $\alpha>2m+\gamma-2$ in (ii) implies the condition
(i).  We refer the reader to Fig.~\ref{fig:Parameter_Regions} for the regions
described by (i)--(iv), and to Remark~\ref{R:global-existence3} for more
explanations of these regions.

\medskip

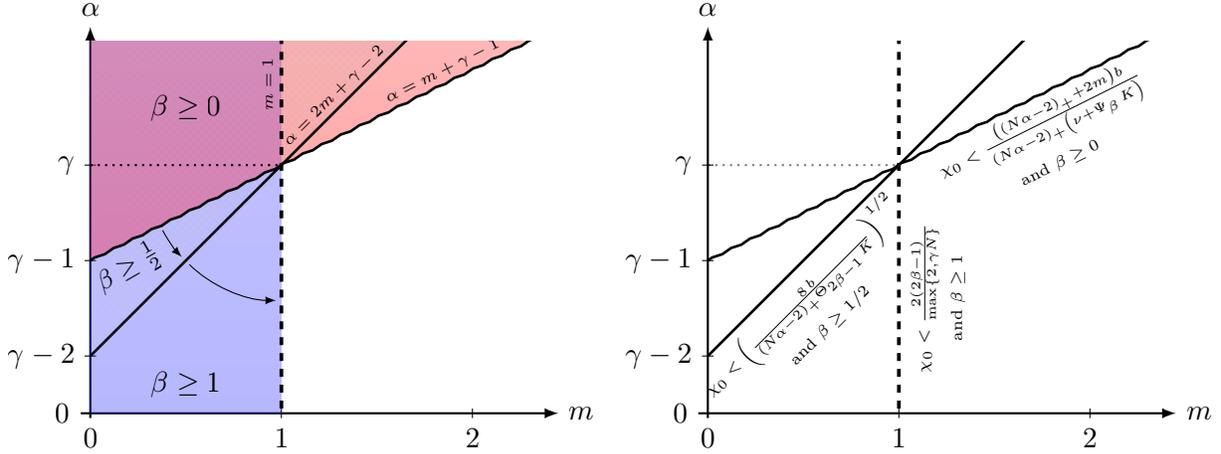
\begin{figure}[htpb]
  \centering
  \begin{tikzpicture}[x=4em, y=2em, scale = 1.65, baseline=(current bounding box.center), blend group=soft light]
    \tikzset{>=latex}

    \def\shift{0.4}

    \node at (0.5, 3.2) {$\beta \ge 0$};
    \node[font=\small, rotate = 30] at (0.2, 1.52) {$\beta \ge \frac{1}{2}$};
    \node at (0.5, 0.3) {$\beta \ge 1$};

    \begin{scope}
      \clip (0,0) rectangle (2.7,3.9); % <-- visible region

      \draw[line width=1.4pt, dashed]
        (1,0.05) --++ (0,4)
        node [pos=0.85, above, sloped, font=\tiny, yshift = -0.1em] {$m=1$};

      \draw [line width=1.0pt, solid]
        (0,1-\shift) --++ (2,4)
        % node [pos=0.20, below, sloped, yshift = -2.5em] {(vi)}
        node [pos=0.66, above, sloped, font=\tiny, yshift = -0.2em] {$\alpha = 2m + \gamma - 2$};

      \draw [dotted, thick, opacity=1] (0, 3-\shift) --++ (1,0);

      \draw [
        decorate,
        decoration={zigzag, segment length=8pt, amplitude=0.3pt},
        line width=1.0pt
        ]
        (0,2-\shift) --++ (2.4,2.4)
        node [pos = 0.78, above, sloped, font=\tiny, yshift = -0.2em] {$\alpha = m + \gamma - 1$};

      % === angle arcs with arrows ===
      % Intersection point (all three lines meet here)
      \coordinate (O)  at (1,3-\shift);

      % Points along each line, starting from O
      \coordinate (Ltwo) at (0,1-\shift);  % on first slanted line
      \coordinate (Lone) at (0,2-\shift);  % on second slanted line
      \coordinate (Vup)  at (1,1);         % on vertical line above O

      % First arc: between first and second slanted lines
      \pic [
        draw,
        ->,
        angle radius=1.8cm,
        shorten >= 2pt,
        shorten <= 2pt
      ] {angle = Lone--O--Ltwo};

      % Second arc: between second slanted and vertical line
      \pic [
        draw,
        ->,
        angle radius=1.8cm,
        shorten >= 2pt,
        shorten <= 2pt
      ] {angle = Ltwo--O--Vup};

      \shadedraw[
        pattern=crosshatch,
        draw=none,
        opacity=0.1
        ] (0,1-\shift) --++ (2,4) --++ (-2,0) -- cycle;

      \shadedraw[
        bottom color=red!90,  % darker at the bottom
        top color=red!70,      % lighter at the top (white)
        draw=none,
        opacity=0.4
        ] (0,2-\shift) --++ (2.4,2.4) --++ (-2.4,0) -- cycle;

      \shadedraw[
        bottom color=blue!70,  % darker at the bottom
        top color=blue!60,      % lighter at the top (white)
        draw=none,
        opacity=0.40
        ] (0,0) --++ (1,0) --++ (0,4) --++ (-1,0) -- cycle;

    \end{scope}

    \draw (+0.03,1-\shift) --++ (-0.06,0) node [left] {$\gamma - 2$};
    \draw (+0.03,2-\shift) --++ (-0.06,0) node [left] {$\gamma - 1$};
    \draw (+0.03,3-\shift) --++ (-0.06,0) node [left] {$\gamma$};
    \draw[thick, ->] (-0.05,0) node [left] {$0$} --++ (2.5,0) node [right] {$m$};
    \draw[thick, ->] (0,-0.05) node [below] {$0$} --++ (0,4.1) node [above] {$\alpha$};
    \draw (1,+0.05) --++ (0,-0.1) node [below] {$1$};
    \draw (2,+0.05) --++ (0,-0.1) node [below] {$2$};

  \end{tikzpicture}
  \begin{tikzpicture}[x=4em, y=2em, scale = 1.65, baseline=(current bounding box.center)]
    \tikzset{>=latex}

    \def\shift{0.4}

    \begin{scope}
      \clip (0,0) rectangle (2.7,3.9); % <-- visible region

      \draw[line width=1.4pt, dashed]
        (1,0.05) --++ (0,4)
        % node [pos=0.07, below, sloped, yshift = -0.5em] {(vii)}
        % node [pos=0.85, above, sloped, font=\tiny, yshift = -0.3em] {$(m=1)$}
	        node [pos=0.29, below, sloped, font=\tiny, yshift = -0.0em] {$\chi_0 < \frac{2(2\beta-1)}{\max\{2,\gamma N\}}$}
	        node [pos=0.29, below, sloped, font=\tiny, yshift = -1.4em] {and $\beta \ge 1$};

      \draw [line width=1.0pt, solid]
        (0,1-\shift) --++ (2,4)
        % node [pos=0.20, below, sloped, yshift = -2.5em] {(vi)}
        % node [pos=0.66, above, sloped, font=\tiny, yshift = -0.3em] {($\alpha = 2m + \gamma - 2$)}
        node [pos=0.20, below, sloped, font=\tiny, yshift = -1.6em] {and $\beta \ge 1/2$}
        node [pos=0.20, below, sloped, font=\tiny, yshift = 0.2em] {$\chi_0 < \left( \frac{8\,b}{(N\alpha-2)_+   \Theta_{2\beta-1}\, K}\right)^{1/2}$};

      % \node [rotate=30] at (0.15, 1.45) {(iii)};
      % \node [rotate=35, font=\tiny] at (0.45, 1.80) {$\beta \ge 1/2$};

      \draw [dotted, thick, opacity=0.5] (0, 3-\shift) --++ (1,0);

      % \fill [pattern=crosshatch, draw=none, opacity=0.2] (0,1-\shift) --++ (2,4) --++ (-2,0) -- cycle;

      % \shadedraw[
      %   bottom color=gray!60,  % darker at the bottom
      %   top color=gray!0,      % lighter at the top (white)
      %   draw=none,
      %   opacity=0.75
      %   ] (0,2-\shift) --++ (2.4,2.4) --++ (-2.4,0) -- cycle;
      % % \node at (0.3, 3.5) {(ii)};

      \draw [
        decorate,
        decoration={zigzag, segment length=8pt, amplitude=0.3pt},
        line width=1.0pt
        ]
        (0,2-\shift) --++ (2.4,2.4)
        % node [pos = 0.51, below, sloped] {(v)}
        node [pos = 0.70, below, sloped, font=\tiny, yshift = -1.80em] {and $\beta \ge 0$}
        node [pos = 0.70, below, sloped, font=\tiny, yshift = 0.20em] {$\chi_0 < \frac{\big((N\alpha-2)_+ + 2m\big) b}{(N\alpha-2)_+ \left(\nu + {\Psi_\beta}\, K \right)}$};

    \end{scope}

    % \filldraw [fill=black] (1,3-\shift) circle [x radius=0.025, y radius=0.05];

    \draw (+0.03,1-\shift) --++ (-0.06,0) node [left] {$\gamma - 2$};
    \draw (+0.03,2-\shift) --++ (-0.06,0) node [left] {$\gamma - 1$};
    \draw (+0.03,3-\shift) --++ (-0.06,0) node [left] {$\gamma$};
    \draw[thick, ->] (-0.05,0) node [left] {$0$} --++ (2.5,0) node [right] {$m$};
    \draw[thick, ->] (0,-0.05) node [below] {$0$} --++ (0,4.1) node [above] {$\alpha$};
    \draw (1,+0.05) --++ (0,-0.1) node [below] {$1$};
    \draw (2,+0.05) --++ (0,-0.1) node [below] {$2$};
    % \draw (3,+0.1) --++ (0,-0.2) node [below] {$3$};

  \end{tikzpicture}

  \caption{Overlapping parameter regimes guaranteeing global classical solutions
    of the parabolic--elliptic system~\eqref{E:main-PE} in the
    $(m,\alpha)$-plane (see Theorems~\ref{T:weak-cross-diffusion}
    and~\ref{T:strong-logistic-source}). Left: open regions determined by the
    three lines $\alpha = m + \gamma - 1$, $\alpha = 2m + \gamma - 2$, and $m =
    1$; the red wedge strictly above $\alpha = m + \gamma - 1$ requires $\beta
    \ge 0$, the crosshatched region strictly above $\alpha = 2m + \gamma - 2$
    requires $\beta \ge \tfrac{1}{2}$, and the blue strip with $0 < m < 1$
    requires $\beta \ge 1$. As we enlarge the admissible region to include
    smaller values of $\alpha$ and $m$, the lower bound on $\beta$ is
    strengthened from $\beta \ge 0$ to $\beta \ge \tfrac{1}{2}$ and then to
    $\beta \ge 1$, which shows that one needs larger $\beta$ to help stabilize
    the system. Right: the same three boundary lines, on which we have the
  corresponding critical cases; along each line global existence holds provided
$\chi_0$ satisfies the indicated smallness condition. }

  \label{fig:Parameter_Regions}
\end{figure}

\subsection{Remarks}\label{SS:remarks}

In this subsection, we provide some remarks on our main results, some
difficulties that arise, and the techniques and ideas utilized in the proofs
of those results.

\begin{remark}[Local existence]\label{R:local-existence}
  This remark is about the local existence of classical solutions
  of~\eqref{E:main-PE}.
  \begin{itemize}

    \item[(1)] As mentioned before, most studies in the literature on the
      existence of classical solutions of~\eqref{E:main-PE} (even with $\beta =
      0$) focus on the case where $m \ge 1$ and/or $\gamma \ge 1$
      (see~\cite{galakhov.salieva.ea:16:on, hu.tao:17:boundedness, zheng:24:on,
      xiang:19:dynamics, zabarankin:16:analytical}). In such cases, if a
      positive classical solution does not exist globally, it necessarily blows
      up in finite time.

    \item[(2)] Proposition~\ref{P:local-existence} includes the existence of
      classical solutions of~\eqref{E:main-PE} with $m < 1$. When $m < 1$,
      $\nabla u^m(t,x)$ may fail to exist at points $(t,x)$ where $u(t,x) = 0$
      for a classical solution $(u(t,x),v(t,x))$ of~\eqref{E:main-PE}. This is
      the reason that $u_0$ in Proposition~\ref{P:local-existence} is assumed to
      satisfy $\inf_{x\in\Omega} u_0(x) > 0$.
      Proposition~\ref{P:local-existence} asserts that if a positive classical
      solution of~\eqref{E:main-PE} does not exist globally, then either the
      solution blows up in finite time, or its infimum approaches zero at a
      finite time. We note that it remains an open problem whether the latter
      scenario can occur when $m < 1$.

  \end{itemize}
\end{remark}

\begin{remark}[Boundedness and global existence]\label{R:global-existence0}
This remark is about one principal idea used in the literature to prove
boundedness and global existence of classical solutions of chemotaxis models.

\begin{itemize}
  \item[(1)] In many studies, the $L^\infty$-boundedness of positive classical
    solutions of a chemotaxis model is derived from their $L^p$-boundedness for
    some $p \gg 1$. For the system~\eqref{E:main-PE}, the boundedness of the
    $L^p$-norm of a positive classical solution for sufficiently large $p$ also
    implies the boundedness of its $L^\infty$-norm (see
    Proposition~\ref{P:Lp->GlobalExist}). However, it is not easy to prove the
    $L^p$-boundedness of positive solutions of~\eqref{E:main-PE}, due to the
    nonlinear sensitivity and the arbitrariness of $\gamma$. More precisely,
    suppose that $(u,v) = \left(u(t,x),v(t,x)\right)$ is a positive classical
    solution of~\eqref{E:main-PE} on some interval $(0,T)$. Then for any $p > 1$
    and $t \in (0, T)$, we have
      \begin{align}\label{E:u-lp-eq1}
      \frac{1}{p} \cdot \frac{d}{dt} \int_{\Omega} u^p
       & = -(p-1)\int_{\Omega} u^{p-2}|\nabla u|^2 + (p-1) \chi_0 \int_{\Omega} \frac{u^{p+m -2}}{(1+v)^\beta}\nabla u \cdot \nabla v \nonumber \\
       & \quad +\int_{\Omega} a u^p -  \int_{\Omega} bu^{p+\alpha}.
     \end{align}
   Multiplying the second equation of~\eqref{E:main-PE} by
   $\frac{u^{p+m-1}}{(1+v)^\beta}$ and integrating it over $\Omega$ yield
   \begin{align}\label{E:u-lp-eq2}
     (p-1)\chi_0\int_\Omega \frac{u^{p+m-2}}{(1+v)^\beta}\nabla u\cdot\nabla v
      & = -\frac{(p-1)\chi_0\mu}{p+m-1} \int_\Omega \frac{u^{p+m-1}}{(1+v)^\beta} v
              + \frac{(p-1)\chi_0\beta}{p+m-1}\int_\Omega \frac{u^{p+m-1}}{(1+v)^{\beta+1}}|\nabla v|^2\nonumber \\
      & \quad +\frac{(p-1)\chi_0\nu}{p+m-1}\int_\Omega \frac{u^{p+m-1}}{(1+v)^\beta} u^\gamma\quad \text{for all $t\in \left(0,T\right)$.}
   \end{align}
   By~\eqref{E:u-lp-eq1} and~\eqref{E:u-lp-eq2}, we have
   \begin{align}\label{E:u-lp-eq3}
     \frac{1}{p}\frac{d}{dt}\int_\Omega u^p
      & =-(p-1)\int_\Omega u^{p-2}|\nabla u|^2 -\frac{(p-1)\chi_0\mu}{p+m-1} \int_\Omega \frac{u^{p+m-1}}{(1+v)^\beta} v \nonumber \\
      & \quad + \frac{(p-1)\chi_0\beta}{p+m-1}\int_\Omega \frac{u^{p+m-1}}{(1+v)^{\beta+1}}|\nabla v|^2
              + \frac{(p-1)\chi_0\nu}{p+m-1}\int_\Omega \frac{u^{p+m-1}}{(1+v)^\beta} u^\gamma \nonumber                           \\
      & \quad + a\int_\Omega u^p-b\int_\Omega u^{p+\alpha}\quad \text{for all $t\in (0,T)$.}
   \end{align}
   One may then use \eqref{E:u-lp-eq1} or \eqref{E:u-lp-eq3} to control the
   terms involving $\nabla v$ and/or $v$ using the good terms, in particular,
   the first and last terms on the right hand side in~\eqref{E:u-lp-eq1}
   or~\eqref{E:u-lp-eq3}. However, several difficulties arise when doing so (see
   (2) and (3) for the examples).

  \item[(2)] When $\chi_0<0$, to estimate $\int_\Omega u^p$
    using~\eqref{E:u-lp-eq3}, it turns out that we need to estimate $\int_\Omega
    v$ (see Remark~\ref{R:chi-negative}). But, by the second equation
    in~\eqref{E:main-PE}, $\mu \int_\Omega v = \nu \int_\Omega u^\gamma$. Hence,
    when $\gamma > 1$, it is not easy to see that $\int_\Omega v$ is bounded.

  \item[(3)] In the case $\chi_0>0$, to estimate $\int_\Omega u^p$
    using~\eqref{E:u-lp-eq3}, we need to control the term involving $\nabla v$
    in~\eqref{E:u-lp-eq1} or \eqref{E:u-lp-eq3}. For this, it turns out that we
    need to establish the following important estimate
   \begin{equation}\label{E:new-new-eq}
     \int_{\Omega} \frac{|\nabla v(t,x)|^{2q}}{\left(1+v(t,x)\right)^{(1+\beta) q}}
     \le \Theta_\beta^q M^* \int_{\Omega} u^{\gamma q}(t,x),
   \end{equation}
   where $M^* > 0$ is some constant that does not depend on $(u, v)$ and
   $\Theta_\beta$ is given in~\eqref{E:Psi-Theta} (see Proposition~\ref{P:Main1}
   below). Such an estimate can be proved by similar arguments as the ones used
   in~\cite[Proposition 1.3]{kurt.shen:20:finite-time}.  In this paper, we
   provide a new and simple proof of this important estimate, together with its
   extension from $\beta = 0$ to general $\beta \ge 0$.

\end{itemize}
\end{remark}

\begin{remark}[Boundedness and global existence with negative sensitivity]\label{R:global-existence1}

  This remark is about the boundedness and global existence of classical
  solutions of~\eqref{E:main-PE} stated in Theorem~\ref{T:negative-sensitivity}
  and the techniques and ideas in its proof.

\begin{itemize}
  \item[(1)] Theorem~\ref{T:negative-sensitivity} is proved for the first time.
    Among others, it indicates that negative chemotaxis does not induce
    finite-time blow-ups, and when $m\ge 1$, it also guarantees the global
    existence of positive classical solutions.

  \item[(2)] As it is mentioned in Remark~\ref{R:global-existence0}, when
    $\chi_0<0$, it is difficult to estimate $\int_\Omega u^p$
    using~\eqref{E:u-lp-eq3}. To overcome this difficulty, we provide a direct
    proof of the boundedness of $\Norm{u(t,\cdot;u_0)}_\infty$. As a byproduct,
    we also obtain an explicit upper bound for $\Norm{u(t,\cdot;u_0)}_\infty$.
    The proof is new and can also be applied to other chemotaxis models.

\end{itemize}

\end{remark}

\begin{remark}[Boundedness and global existence with weak nonlinear cross diffusion rate]\label{R:global-existence2}

  This remark is about the boundedness and global existence of classical
  solutions of~\eqref{E:main-PE} stated in Theorem~\ref{T:weak-cross-diffusion}
  and the techniques and ideas in its proof.

  \begin{itemize}

    \item[(1)] Note that when $u \ge 1$, $\frac{u^m}{(1+v)^\beta}$ increases as
      $m$ increases and decreases as $\beta$ increases. We therefore refer
      $\frac{u^m}{(1+v)^\beta}$ as weak nonlinear cross diffusion rate when $0 <
      m \le 1$ and $\beta\ge 1$.

    \item[(2)] In Theorem~\ref{T:weak-cross-diffusion}, $a$ and $b$ can be zero.
      Hence it applies to the minimal model. Its proof involves nontrivial
      applications of Ehrling's Lemma, Gagliardo-Nirenberg inequality, and the
      inequality~\eqref{E:new-new-eq} established in Proposition~\ref{P:Main1}.

    \item[(3)] Theorem~\ref{T:weak-cross-diffusion}(1) indicates that the regular
      diffusion dominates the dynamics when $0<m<1$ and $\beta \ge 1$, and
      Theorem~\ref{T:weak-cross-diffusion}(2) indicates that when $m = 1$ and
      $\beta \ge 1$, finite-time blow-up does not occur when $\chi_0$ is not
      large.

	    \item[(4)] When $m=1$, the condition $\chi_0 < \frac{2(2\beta-1)}{\max\{2,
	      \gamma N\}}$ indicates that large $\beta$ has less of an effect on the
	      evolution of the population density of the biological species.
	      Equivalently, for $m=1$ and $\beta \ge 1$ this smallness condition on
	      $\chi_0$ can be written as the large-$\beta$ requirement
	      \[
	        \beta > \max\left\{1,\; \frac{1}{2} + \frac{\chi_0}{4}\, \max\{2,\gamma N\}\right\},
	      \]
	      which is exactly the condition quoted in the abstract.

  \end{itemize}

\end{remark}

\begin{remark}[Boundedness and global existence with relatively strong logistic source]\label{R:global-existence3}

  This remark is about the boundedness and global existence of classical
  solutions of~\eqref{E:main-PE} stated in
  Theorem~\ref{T:strong-logistic-source} and the techniques and ideas in its
  proof. 

  \begin{itemize}

    \item[(1)] Theorem \ref{T:strong-logistic-source} reflects the influence of
      logistic source on the boundedness and global existence of classical
      solutions of~\eqref{E:main-PE}. The conditions $\alpha > m + \gamma - 1$
      in (i) and $\alpha > 2m + \gamma - 2$ in (ii) indicate that the logistic
      source is relatively strong compared with the nonlinear cross diffusion
      function $\frac{u^m}{(1+v)^\beta}$ and the production function $u^\gamma$.
      (iii) and (iv) are borderline cases of (i) and (ii), respectively.

    \item[(2)] The conditions in (i)--(iv) overlap as illustrated in
      Fig.~\ref{fig:Parameter_Regions}. For example, if $\alpha = m + \gamma -
      1$, $m \in (0,1)$, and $\beta \ge \tfrac{1}{2}$, then $\alpha > 2m +
      \gamma - 2$, so condition~(ii) applies and hence~\eqref{E:bounded-finite}
      holds without assuming~\eqref{E:cond-chi-eq1}. Observe
      that~\eqref{E:cond-chi-eq1} always holds for $N \le 2/\alpha$ because in
      this case $(N\alpha-2)_+ = 0$ in~\eqref{E:cond-chi-eq1}.

    \item[(3)] As mentioned above, the work~\cite{galakhov.salieva.ea:16:on}
      studied the dynamics of system~\eqref{E:main-PE} with $m, \gamma \ge 1$,
      $a = b$, and $\beta = 0$, and identified the parameter regions for the
      global existence of all nonnegative classical solutions. The reader is
      also referred to~\cite{tello.winkler:07:chemotaxis} for the study
      of~\eqref{E:main-PE} with $\beta = 0$ and $m = \alpha = \gamma = 1$.
      Although it is expected that the signal-dependent sensitivity
      $(1+v)^{-\beta}$ ($\beta>0$) would not affect the global existence of the
      solutions, proving this is nontrivial.
      Theorem~\ref{T:strong-logistic-source} shows that this holds partially.
      For example, the condition~(i) in
      Theorem~\ref{T:strong-logistic-source}(1), namely, $\alpha > m + \gamma -
      1$, is independent of $\beta$ and coincides with Eq.~(1.7)
      of~\cite{galakhov.salieva.ea:16:on}, which demonstrates that~Theorem 1.1
      of~\cite{galakhov.salieva.ea:16:on} under the condition Eq.~(1.7)
      of~\cite{galakhov.salieva.ea:16:on} for~\eqref{E:main-PE} with $\beta = 0$
      remains valid for any $\beta > 0$. The function $\Theta_\beta$ in
      condition \eqref{E:cond-chi-eq2} is a monotone decreasing function of
      $\beta$ that approaches zero as $\beta \to \infty$ (see
      Fig.~\ref{fig:Psi-Theta}). Hence a larger value of $\beta$ makes
      condition~\eqref{E:cond-chi-eq2} easier to satisfy, or equivalently,
      increases the threshold for admissible $\chi_0$.

    \item[(4)] Theorem~\ref{T:strong-logistic-source} (2) indicates that, when
      $m \ge 1$, the boundedness of $u(t,x;u_0)$ on $[0, T_{\max}(u_0))$
      implies $T_{\max}(u_0) = \infty$. It remains open whether the boundedness
      of $u(t,x;u_0)$ on $[0, T_{\max}(u_0))$ implies $T_{\max}(u_0) = \infty$
      when $m < 1$.

  \end{itemize}
\end{remark}

\begin{remark}[Three $\chi_0$-smallness conditions when
  $m=1$, $\alpha = \gamma$, and $\beta \ge 1$]\label{R:chi0-beta-comparison}

  In the slice $m=1$ and $\alpha=\gamma$, the two borderline relations
  $\alpha=m+\gamma-1$ and $\alpha=2m+\gamma-2$ coincide; see
  Fig.~\ref{fig:Parameter_Regions}. Hence, when $a,b>0$ and $\beta\ge 1$, the
  boundedness conclusion may be obtained either from
  Theorem~\ref{T:strong-logistic-source}(1) in cases~(iii)--(iv) or from
  Theorem~\ref{T:weak-cross-diffusion}(2). To emphasize the role of the
  signal-dependent sensitivity $(1+v)^{-\beta}$, we compare the corresponding
  $\chi_0$-smallness requirements. Set $K=K(N,\gamma,\gamma,\mu,\nu)$ as
  in~\eqref{E:K}. When $N\gamma>2$, the three conditions take the form:
  \[
    \chi_0<\chi_{*,1}(\beta),\quad
    \chi_0<\chi_{*,2}(\beta),\quad
    \chi_0<\chi_{*,\mathrm{w}}(\beta),
  \]
  Here $\Psi_\beta$ and $\Theta_\beta$ are defined in~\eqref{E:Psi-Theta}.
  These thresholds are given by:
  \begin{subequations}\label{E:critical}
    \begin{align}
      \chi_0
      & < \chi_{*,1}(\beta)
      \coloneqq \frac{N\gamma\, b}{(N\gamma-2)\left(\nu + \Psi_\beta K\right)},
      \label{E:critical-1}\\
      \chi_0
      & < \chi_{*,2}(\beta)
      \coloneqq \left(\frac{8b}{(N\gamma-2)\Theta_{2\beta-1}K}\right)^{1/2},
      \label{E:critical-2}\\
      \chi_0
      & < \chi_{*,\mathrm{w}}(\beta)
      \coloneqq \frac{2(2\beta-1)}{\max\{2,\gamma N\}}.
      \label{E:critical-w}
    \end{align}
  \end{subequations}
  When $N\gamma\le 2$, we have $(N\gamma-2)_+=0$ and the right-hand sides of
  \eqref{E:cond-chi-eq1} and~\eqref{E:cond-chi-eq2} are interpreted as
  $+\infty$. In particular, the strong-logistic borderline conditions impose no
  restriction on $\chi_0$. In the region where all three criteria are
  available, the least restrictive sufficient condition on $\chi_0$ is
  therefore
  \[
    \chi_0<\max\left\{\chi_{*,1}(\beta),\chi_{*,2}(\beta),\chi_{*,\mathrm{w}}(\beta)\right\}.
  \]

  \begin{enumerate}

    \item (Role of $\beta$.) The $\beta$-dependence of the three thresholds
      in~\eqref{E:critical} is qualitatively different. The functions
      $\beta\mapsto\chi_{*,2}(\beta)$ and
      $\beta\mapsto\chi_{*,\mathrm{w}}(\beta)$ are increasing, reflecting the
      stabilizing effect of larger $\beta$. In contrast,
      $\beta\mapsto\chi_{*,1}(\beta)$ is decreasing because
      $\Psi_\beta=\left(\frac{\beta}{1+\beta}\right)^{1+\beta}$ is increasing in
      $\beta$ and $\Psi_\beta\to e^{-1}$ as $\beta\to\infty$. In particular,
      using $\Theta_{\eta}=\eta^\eta(1+\eta)^{-(1+\eta)}\sim e^{-1}\eta^{-1}$ as
      $\eta\to\infty$, we obtain for $N\gamma>2$ that, as $\beta\to\infty$,
      \begin{align*}
        \chi_{*,1}(\beta)
        & \to \frac{N\gamma\, b}{(N\gamma-2)\left(\nu + e^{-1}K\right)},
        \\
        \chi_{*,2}(\beta)
        & \sim \left(\frac{8e\,b}{(N\gamma-2)K}\right)^{1/2}\,(2\beta-1)^{1/2},
        \\
        \chi_{*,\mathrm{w}}(\beta)
        & \sim \frac{4}{\max\{2,\gamma N\}}\,\beta.
      \end{align*}
      Thus, for fixed $b$ and $N\gamma>2$, the admissible size of $\chi_0$ grows
      at most like $O(\beta^{1/2})$ under \eqref{E:cond-chi-eq2} but like
      $O(\beta)$ under Theorem~\ref{T:weak-cross-diffusion}(2). In particular,
      for sufficiently large $\beta$, $\chi_{*,\mathrm{w}}(\beta)$ provides the
      least restrictive condition on $\chi_0$ among the three.

    \item (Tradeoff with the damping strength $b$.) For fixed $\beta$, the two
      strong-logistic thresholds satisfy $\chi_{*,1}(\beta)=O(b)$ and
      $\chi_{*,2}(\beta)=O(b^{1/2})$, whereas $\chi_{*,\mathrm{w}}(\beta)$ does
      not involve $b$ at all. Hence, for moderate $\beta$ and large $b$, either
      \eqref{E:cond-chi-eq1} or \eqref{E:cond-chi-eq2} can allow substantially
      larger $\chi_0$ than Theorem~\ref{T:weak-cross-diffusion}(2).

    \item (Role of the dimension $N$.) The dependence on $N$ enters explicitly
      through $N\gamma-2$ and $\max\{2,\gamma N\}$ and implicitly through
      $K=K(N,\gamma,\gamma,\mu,\nu)$. If $N\gamma\le 2$, then
      \eqref{E:cond-chi-eq1}--\eqref{E:cond-chi-eq2} are unconditional in
      $\chi_0$, whereas \eqref{E:critical-w} yields $\chi_0<2\beta-1$. If
      $N\gamma>2$, then $\max\{2,\gamma N\}=\gamma N$ and
      $\chi_{*,\mathrm{w}}(\beta)=\frac{2(2\beta-1)}{\gamma N}$ decays like
      $1/N$ for fixed $\beta$. In comparison, ignoring the $N$-dependence of
      $K$, the explicit prefactors in $\chi_{*,1}(\beta)$ and
      $\chi_{*,2}(\beta)$ scale only like $N\gamma/(N\gamma-2)$ and
      $(N\gamma-2)^{-1/2}$, respectively. This is consistent with the heuristic
      that aggregation becomes harder to control in higher dimensions,
      especially in view of the clear $1/N$ decay of the weak-cross-diffusion
      threshold $\chi_{*,\mathrm{w}}(\beta)$.

  \end{enumerate}
\end{remark}

\begin{remark}[Optimality via blow-up in related models]
  We briefly collect blow-up results in related Keller--Segel models which can
  be viewed as evidence for near-optimality of the logistic thresholds in
  Theorem~\ref{T:strong-logistic-source}. To make the comparison transparent,
  consider the constant-sensitivity and linear-production slice $\beta = 0$ and
  $m = \gamma = 1$. Then~\eqref{E:main-PE} becomes a parabolic--elliptic
  Keller--Segel model with logistic-type term $au-bu^{1+\alpha}$ coupled to an
  elliptic signal equation with linear decay term $-\mu v$, namely $-\Delta
  v+\mu v=\nu u$.   In the Possion-type case $\mu=0$ (where Neumann solvability
  requires subtracting the spatial mean of $u$, i.e., substracting
  $\nu\int_\Omega u(t,x)dx$ on the right-hand side of $-\Delta v+\mu v=\nu u$),
  Theorem~1.1 of~\cite{fuest:21:approaching} constructs radially symmetric
  initial data in balls which lead to finite-time blow-up for all
  $\kappa=1+\alpha<2$ in dimensions $N\ge 4$, and even at the borderline
  exponent $\kappa=2$ in dimensions $N\ge 5$ provided the coefficient of
  $u^\kappa$ is sufficiently small. Moreover, even in the case $0=\Delta v-v+u$
  with linear signal decay, \cite[Theorem~1.1]{winkler:18:finite-time} shows
  that finite-time blow-up can occur in dimensions $N\in\{3,4\}$ for suitable
  radial data if $\kappa$ is only slightly above $1$ (in particular for
  $\kappa<7/6$), demonstrating that superlinear logistic death does not, in
  general, prevent chemotactic collapse. Finally, in the absence of logistic
  damping, finite-time blow-up in parabolic--elliptic chemotaxis systems is
  classical in dimensions $N\ge 2$ for suitable choices of sensitivity and data,
  see, e.g., \cite{nagai:95:blow-up, nagai.senba:98:global}. These results do
  not directly apply to the full structure of~\eqref{E:main-PE} with
  signal-dependent sensitivity $(1+v)^{-\beta}$ and general exponents
  $(m,\gamma)$, but they support the qualitative message that, without
  additional smallness assumptions on the taxis strength (or without
  sufficiently strong damping), one cannot expect unconditional boundedness for
  small $\alpha$, and that the borderline $\alpha=1$ is a natural threshold in
  closely related parabolic--elliptic settings.
\end{remark}

\bigskip

The rest of the paper is organized as follows. In Section~\ref{S:Prelim}, we
present some preliminary material to be used in later sections. We study the
boundedness and global existence of system~\eqref{E:main-PE} and prove
Theorems~\ref{T:negative-sensitivity}, \ref{T:weak-cross-diffusion},
and~\ref{T:strong-logistic-source} in Sections~\ref{S:negative-sensitivity},
\ref{S:weak-cross-diffusion}, and~\ref{S:strong-logistic-source}, respectively.
Throughout this paper, we make a deliberate effort to keep many constants as
explicit as possible whenever it is feasible, as this may help clarify the
structure of the argument and enable readers to follow the proof procedures more
easily.

\section{Preliminaries}\label{S:Prelim}

In this section, we state some preliminary materials to be used in later
sections.

\subsection{Analytic semigroups generated by $A = -\Delta +\mu I$}\label{SS:Semigroup}

In this subsection, we recall some basic properties of the analytic semigroup
generated by $-\Delta +\mu I$. Let $1\le p < \infty$ and $A_p = -\Delta +\mu I$
with domain $\mathcal{D}(A_p)$ given by
\begin{align*}
  \mathcal{D}(A_p)=\left\{u\in W^{2,p}(\Omega)\,:\, \text{$\frac{\partial u}{\partial n}=0$ on $\partial\Omega$} \right\}.
\end{align*}
The operator $A_p$ generates an analytic semigroup, denoted by $T_p(t) = e^{-A_p
t}$, on $L^p(\Omega)$. Let $X_p^\sigma$ be the fractional power space associated
with $A$, i.e., $X_p^\sigma = \mathcal{D}(A_p^\sigma)$ for $\sigma>0$. Let
$A_\infty = -\Delta + \mu I$ whose domain $\mathcal{D}(A_\infty)$ is given by
\begin{align*}
  \mathcal{D}(A_\infty)=\left\{u\in C^2(\overline{\Omega})\,:\, \text{$\frac{\partial u}{\partial n}=0$ on $\partial\Omega$} \right\}.
\end{align*}
The operator $A_\infty$ generates an analytic semigroup, denoted by $T_\infty(t)
= e^{-A_\infty t}$, on $C(\overline{\Omega})$. The reader is referred
to~\cite{henry:81:geometric} for further details on analytic semigroups and
fractional power spaces.

Throughout the rest of the paper, we suppress the subscription $p$ in $A_p$ and
$T_p(t)$ $(1 \le p \le \infty$) when context allows. It should be clear from the
context whether $e^{-At}$ represents the analytic semigroup in $L^p(\Omega)$
generated by $A_p$ or the analytic semigroup in $C(\overline{\Omega})$ generated
by $A_\infty$.

\begin{lemma}[{\cite[Theorem 1.4.3]{henry:81:geometric} or~\cite[Theorem
  6.13]{pazy:83:semigroups}}]\label{L:FracSemigroup} For all $\sigma \ge 0$, $p
  \in [1, \infty]$, and $\delta \in (0, \mu)$, there is a constant $C_{\sigma,
  p, \delta} > 0$ such that
  \begin{equation}\label{E:FracSemigroup-1}
    \Norm{{A_p^\sigma e^{-A_pt}} u}_p \le C_{\sigma,p,\delta} t^{-\sigma} e^{-\delta t} \Norm{u}_p, \qquad \text{for all $t>0$.}
  \end{equation}
  For each $\sigma \in (0, 1]$, there is a constant $C_\sigma > 0$ such that
  for any $u \in X^\sigma_p$,
  \begin{equation}\label{E:FracSemigroup-2}
    \Norm{\big(e^{-A_p t}-I\big)u}_p \le C_\sigma t^\sigma \Norm{A^\sigma u}_p \qquad \text{for all $t>0$.}
  \end{equation}
\end{lemma}

\begin{lemma}[{\cite[Theorem 1.6.1]{henry:81:geometric}}]\label{L:Embedding}
  Suppose that $\sigma \in [0,1]$ and $1 \le p < \infty$. Then the following
  Sobolev-type embeddings hold:
  \begin{subequations}\label{E:Embedding}
  \begin{eqnarray}\label{E:Embedding-1}
    X_p^\sigma \hookrightarrow W^{k,q}(\Omega)            & \text{when} & k-N/q<2\sigma-N/p, \: q\ge p,\: k\in \{0,1\}, \\
    X_p^\sigma\hookrightarrow C^\theta(\overline{\Omega}) & \text{when} & 0 \le \theta <2\sigma - N/p. \label{E:Embedding-2}
  \end{eqnarray}
  \end{subequations}
\end{lemma}

\begin{lemma}\label{L:T-Nabla}
  Let $\lambda_1$ be the first nonzero eigenvalue of $-\Delta$ in $\Omega$ under
  Neumann boundary conditions. There exists some constant $C > 0$ which depends
  only on $\Omega$, such that for all $p \in (1, \infty)$ and $t > 0$, it holds
  that
  \begin{equation}\label{E:T-Nabla}
    \Norm{e^{-tA_p}\nabla \cdot \phi}_p
    \le C \left(1+t^{-1/2}\right) e^{-(\lambda_1+\mu) t} \Norm{\phi}_p
  \end{equation}
  where $\phi \in \big(C_0^\infty({\Omega})\big)^N$. Consequently, for all
  $t>0$, the operator $e^{-tA_p}\nabla\cdot$ possesses a uniquely determined
  extension to an operator from $L^p(\Omega)$ into $L^p(\Omega)$, with norm
  controlled according to~\eqref{E:T-Nabla}.
\end{lemma}

\begin{proof}
  This lemma is a special case of part (iv)
  of~\cite[Lemma~1.3]{winkler:10:aggregation}.
\end{proof}

\begin{lemma}\label{L:new-lm}
  For any $\sigma > 0$ and $p \in (1, \infty)$, there is $C > 0$ such that for
  any $u \in (L^p(\Omega))^N$, it holds that
  \begin{equation}\label{E:new-lm}
    \norm{A_p^\sigma e^{-A_p t} \nabla \cdot u}_p \le C t^{-\sigma} (1+t^{-1/2}) e^{-(\lambda_1 +\mu) t/2} \Norm{u}_p
    \qquad \text{for all $t>0$.}
  \end{equation}
\end{lemma}

\begin{proof}
  Note that
  \begin{align*}
    \|A_p^\sigma e^{-A_p t}\nabla \cdot u\|_p
    & =\|A_p^\sigma e^{-A_pt/2} e^{-A_p t/2}\nabla \cdot u\|_p             \\
    & \le C t^{-\sigma} \|e^{-A_p t/2}\nabla \cdot u\|_p \quad              & \text{by Lemma~\ref{L:FracSemigroup}} \\
    & \le C t^{-\sigma} (1+t^{-1/2}) e^{-(\lambda_1+\mu) t/2} \|u\|_p\quad  & \text{by Lemma~\ref{L:T-Nabla}}
  \end{align*}
  for all $t>0$. The lemma is thus proved.
\end{proof}

\subsection{Local existence of classical solutions}\label{SS:Local-Exist}

In this subsection, we outline the main steps in the proof of
Proposition~\ref{P:local-existence}.

\begin{proof}[Proof of Proposition~\ref{P:local-existence}]
  The proof follows standard semigroup theory and a fixed-point argument (see,
  e.g., \cite[Theorem 3.1]{horstmann.winkler:05:boundedness}, \cite[Theorem
  2.1]{tello.winkler:07:chemotaxis}, \cite[Theorems 1.2 and
  1.3]{salako.shen:17:global}). For the reader's convenience and for future
  reference, we present an outline of the proof in the following steps. \medskip

  \noindent\textbf{Step 1.~} Suppose that $u_0$
  satisfies~\eqref{E:initial-cond-PE}. For given $T>0$, set
  \begin{align*}
    X_T \coloneqq C\left([0,T],C(\overline{\Omega})\right) \quad \text{with} \quad
    \Norm{u}_{X_T} \coloneqq \max_{t\in[0,T]} \Norm{u(t)}_{\infty}.
  \end{align*}
  Recall that the semigroup operator $T(t)$ is defined in
  Section~\ref{SS:Local-Exist}. For any given $u \in X_T$, define
  \begin{align*}
    G(u)(t)
    & \coloneqq T(t)u_0 - \int_0^t T(t-s)\nabla\cdot(\chi(v(s)) u^m(s)\nabla v(s))ds + \int_0^t T(t-s)u(s)(a+\mu-bu^\alpha(s))ds \\
    & = G_0(u_0) - G_1(u)(t) + G_2(u)(t),
  \end{align*}
  where we use the notation $\chi(v) \coloneqq \chi_0(1+v)^{-\beta}$ and
  \begin{equation}\label{E:Resolvent}
    v(t) \coloneqq (\mu I-\Delta)^{-1} (\nu u^\gamma(t))
    = \nu \int_0^\infty e^{-As} u^{\gamma}(t)ds.
  \end{equation}
  In this step, we provide some estimates for $G_1(u)$, $G_2(u)$, and $v$, which
  will be used in later steps as well as in later sections.

  Let $u \in X_T$ and set $R \coloneqq \Norm{u}_{X_T}$. First,
  % by
  % Lemma~\ref{L:FracSemigroup}, we obtain that
  % \begin{align}\label{E:bound-v}
  %   \Norm{v(s)}_\infty
  %   = \Norm{(\mu I-\Delta)^{-1} (\nu u^\gamma(s))}_\infty
  %   \le C \Norm{u^\gamma(s)}_\infty
  %   \le C R^\gamma, \quad \text{for all $s\in [0,T].$}
  % \end{align}
  by choosing $1/2 < \sigma_0 < 1$ and $p_0>\frac{N}{2\sigma_0 -1}$, we see that
  Lemma~\ref{L:Embedding} implies
  \begin{align}\label{E:bound-v+grad-v}
    \Norm{v(s,\cdot)}_\infty+ \Norm{\nabla v(s,\cdot)}_\infty
    & \le C \Norm{A_{p_0}^{\sigma_0}  v(s,\cdot)}_{p_0}\nonumber                                \\
    & \le C \Norm{u^\gamma(s,\cdot)}_{p_0} \int_0^\infty t^{-\sigma_0}e^{-\delta t}dt \nonumber \\
    & \le C R^\gamma, \quad \forall\, t\in [0, T].
  \end{align}

  Next, by Lemmas~\ref{L:FracSemigroup} and~\ref{L:new-lm}, for any $p > 1$, $0
  < \sigma < 1/2$, and $0 < \delta < \mu/2$, we see that for all $t \in [0,T]$,
  \begin{equation}\label{E:G-1-eq1}
    \Norm{A_p^\sigma G_1(u)(t)}_p
    \le C \int_0^ t (t-s)^{-\sigma} \left(1+(t-s)^{-1/2}\right)e^{-\delta(t-s)} \Norm{u^m(s,\cdot)}_p \Norm{\nabla v(s,\cdot)}_\infty ds,
  \end{equation}
  and
  \begin{equation}\label{E:G-2-eq1}
    \Norm{A_p^\sigma G_2(u)(t)}_p
    \le C \int_0^t (t-s)^{-\sigma} e^{-\delta (t-s)} \left(\Norm{u(s,\cdot)}_p + \Norm{u^{1+\alpha}(s,\cdot)}_p\right) ds.
  \end{equation}

  Now, for $p > 1$, $0 < \sigma < \tilde{\sigma} < 1/2$, $0 < \delta < \mu/2$,
  and $t_1, t_2 \in [0,T]$ with $t_1 < t_2$, we have
  \begin{align}\label{E:G-1-eq2}
    & \Norm{A_p^\sigma\Big(G_1(u)(t_2)-G_1(u)(t_1)\Big)}_p\nonumber                                                                                                                                                               \\
    & \le C \Norm{\int_0^ {t_1} A^\sigma \Big(T(t_2-t_1)-I\Big) T(t_1-s) \nabla \cdot(\chi(v(s))u^m(s)\nabla v(s)) ds}_p\nonumber                                                                                                 \\
    & \quad + C \Norm{\int_{t_1}^{t_2} A^\sigma T(t_2-s)  \nabla \cdot(\chi(v(s))u^m(s)\nabla v(s)) ds}_p   \nonumber                                                                                                             \\
    & \le C (t_2-t_1)^{\tilde \sigma-\sigma}\int_0^ {t_1} \Norm{A^{\tilde \sigma} T(t_1-s) \nabla \cdot(\chi(v(s))u^m(s)\nabla v(s))}_p ds  \quad \text{(by \eqref{E:FracSemigroup-2})}\nonumber \\
    & \quad + C \int_{t_1}^{t_2} \Norm{A^\sigma T(t_2-s) \nabla \cdot(\chi(v(s))u^m(s)\nabla v(s))}_p ds  \nonumber                                                                                                               \\
    & \le C (t_2-t_1)^{\tilde \sigma-\sigma} \int_0^{t_1} (t_1-s)^{-\tilde \sigma}  \left(1+(t_1-s)^{-1/2}\right)e^{-\delta(t_1-s)} \Norm{u^m(s,\cdot)}_p \Norm{\nabla v(s,\cdot)}_\infty ds\nonumber                          \\
    & \quad + C \int_{t_1}^{t_2}(t_2-s)^{-\sigma}  \left(1+(t_2-s)^{-1/2}\right)e^{-\delta(t_2-s)}\Norm{u^m(s,\cdot)}_p \Norm{\nabla v(s,\cdot)}_\infty ds  \quad \text{(by \eqref{E:G-1-eq1})}.
     \end{align}
  Similarly, we have
  \begin{align}\label{E:G-2-eq2}
    \MoveEqLeft \Norm{A_p^\sigma\Big(G_2(u)(t_2)-G_2(u)(t_1)\Big)}_p\nonumber\\
    & \le C \left(t_2-t_1\right)^{\tilde \sigma-\sigma } \int_0^{t_1} \Norm{A^{\tilde \sigma} T(t_1-s) u(s)\big(a+\mu-b u^\alpha(s)\big)}_p ds\nonumber                               \\
    & \quad + C \int_{t_1}^{t_2} \Norm{A^\sigma T(t_2-s) u(s)\big(a+\mu - b u^\alpha(s)\big)}_p ds \nonumber                                                                          \\
    & \le C(t_2-t_1)^{\tilde \sigma-\sigma} \int_0^{t_1}   e^{-\delta (t_1-s)} (t_1-s)^{-\tilde \sigma} \left(\Norm{u(s,\cdot)}_p + \Norm{u^{1+\alpha}(s,\cdot)}_p\right) ds\nonumber \\
    & \quad + C\int_{t_1}^{t_2}  e^{-\delta (t_2-s)} (t_2-s)^{-\sigma} \left(\Norm{u(s,\cdot)}_p + \Norm{u^{1+\alpha}(s,\cdot)}_p\right) ds.
  \end{align}
  \medskip

  \noindent\textbf{Step 2.} In this step, we demonstrate that $G(u)$ is well
  defined in $X_T$, that is, $G(u) \in X_T$. \medskip

  We first show that
  \begin{equation}\label{E:local-pf-eq1}
    G(u)(s)\in C(\overline\Omega) \quad \text{for all $s\in [0,T]$} \quad \text{and} \quad
    \sup_{s\in [0,T]} \Norm{G(u)(s)}_\infty<\infty.
  \end{equation}
  To this end, we choose $\sigma > 0$ and $p > 1$ such that
  \begin{align}\label{E:sigma-p}
    0 < \sigma < 1/4 \quad \text{and} \quad  2\sigma-N/p > 0.
  \end{align}
  Then, we can apply the embedding in~\eqref{E:Embedding-2} with $\theta =0$ to
  see that for any $t\in[0,T]$,
  \begin{align}\label{E:le-s1-gineq}
    \Norm{G(u)(t)}_{\infty}
    \le & \Norm{T(t)u_0}_{\infty}+ C\|A_p^\sigma G_1(u)(t)\|_p  +\|A_p^\sigma G_2(u)(t)\|_p .
  \end{align}
  By~\eqref{E:bound-v+grad-v} and~\eqref{E:G-1-eq1}, we get
  \begin{align}\label{E:le-s1-eq1}
    \Norm{A_p^\sigma G_1(u)(t)}_p
    \le C R^{m+\gamma} \left(T^{1-\sigma} +T^{1/2-\sigma} \right).
  \end{align}
  Applying~\eqref{E:G-2-eq1} yields
  \begin{align}\label{E:le-s1-eq2}
    \Norm{A_p^\sigma G_2(u)(t)}_p \le CRT^{1-\sigma}+ CR^{1+\alpha}T^{1-\sigma}.
  \end{align}
  Note that $T(t)u_0\in C(\overline \Omega)$ and $G_0 = \Norm{T(t)u_0}_\infty
  \le C\Norm{u_0}_\infty$ for all $t\in [0,T]$. Then by~\eqref{E:le-s1-gineq},
  \eqref{E:le-s1-eq1}, and~\eqref{E:le-s1-eq2}, we
  obtain~\eqref{E:local-pf-eq1}.

  Next, we show that for any $t_1, t_2 \in [0,T]$,
  \begin{equation}\label{E:local-pf-eq2}
    \lim_{ t_2-t_1\to 0} \Norm{G(u)(t_2)-G(u)(t_1)}_\infty = 0.
  \end{equation}
  To this end, let $t_1, t_2 \in [0,T]$ with $t_2 > t_1$. Choose $\sigma$ and
  $p$ satisfying \eqref{E:sigma-p}. Then by~\eqref{E:bound-v+grad-v},
  \eqref{E:G-1-eq2}, and~\eqref{E:G-2-eq2} with $\tilde{\sigma}=2\sigma$, we
  obtain
  \begin{align*}
     \MoveEqLeft \Norm{G(u)(t_2)-G(u)(t_1)}_\infty                                                                                                               \\
     & \le  \Norm{(T(t_2)-T(t_1))u_0}_\infty + \Norm{A_p^\sigma\Big(G_1(u)(t_2)-G_1(u)(t_1)\Big)}_p  + \Norm{A_p^\sigma\Big(G_2(u)(t_2)-G_2(u)(t_1)\Big)}_p      \\
     & \le  \Norm{(T(t_2)-T(t_1))u_0}_\infty + C R^{m+\gamma} (t_2-t_1)^\sigma + C R^{m+\gamma}\left( (t_2-t_1)^{1-\sigma}+(t_2-t_1)^{\frac{1}{2}-\sigma}\right) \\
     & \quad +C (R+R^{1+\alpha}) \left( (t_2-t_1)^\sigma+(t_2-t_1)^{1-\sigma}\right).
  \end{align*}
  By the continuity of $T(t)u_0$ in $t\ge0$ and the boundedness of $T(t)$ in $t
  \ge 0$, we have that
  \begin{align*}
    \Norm{T(t_1) (T(t_2-t_1)-I) u_0}_\infty\to 0 \quad {\rm as}\,\,\, t_2-t_1\to 0.
  \end{align*}
  The limit in \eqref{E:local-pf-eq2} then follows. From~\eqref{E:local-pf-eq1}
  and~\eqref{E:local-pf-eq2}, we can conclude that $G(u) \in X_T$. \medskip

  \noindent\textbf{Step~3.} For given $u_0$
  satisfying~\eqref{E:initial-cond-PE}, fix two numbers $r$ and $R$ such that
  \begin{align*}
    0 < r < \inf_{x\in\Omega}u_0(x)\le \sup_{x\in\Omega} u_0(x) < R.
  \end{align*}
  In this step, we claim that there is $T \coloneqq T(r,R) > 0$ such that $G$
  maps
  \begin{align*}
    S_{r,R,T} \coloneqq \left\{u\in X_T\, :\, r\le u(t)\le R \quad \forall\, t\in[0,T]\right\}
  \end{align*}
  into itself, and is a contraction. \medskip

  First, for any $u \in S_{r,R,T}$, by the arguments in Step~2, for any
  $t\in[0,T]$, we have
  \begin{align*}
    \Norm{G(u)(t)}_\infty \le & \Norm{T(t) u_0}_\infty+ C R^{m+\gamma} \left(T^{1-\sigma}+T^{1/2-\sigma} \right)+ CT^{1-\sigma} (R+ R^{1+\alpha})
  \end{align*}
  and
  \begin{align*}
    \Norm{G(u)(t)}_\infty \ge & \Norm{T(t) u_0}_\infty- C R^{m+\gamma}
    \left(T^{1-\sigma}+T^{1/2-\sigma} \right)- CT^{1-\sigma} (R+ R^{1+\alpha}).
  \end{align*}
  This together with $\Norm{T(t)u_0}_\infty \to \Norm{u_0}_\infty \: (>r)$ as $t
  \to 0_+$ implies that there is $T_1>0$ such that for all $T \in (0, T_1]$ and
  $u \in S_{r,R,T}$,
  \begin{align*}
    r\le G(u)(t)\le R,\quad \text{for all $t\in [0,T]$.}
  \end{align*}
  Therefore, $G$ maps $S_{r,R,T}$ into itself. \medskip

  Second, let $\sigma$ and $p$ be chosen as in~\eqref{E:sigma-p}. Arguing
  similarly to Step~2, we apply~\eqref{E:Embedding-2} to deduce that, for any
  $u_1, u_2 \in S_{r,R,T}$ and $t \in [0,T]$, it holds that
  \begin{align}\label{E:le-s2-mainineq}
    \Norm{G(u_1)(t)-G(u_2)(t)}_\infty
    \le       & C \int_0^t \Norm{A^\sigma T(t-s)\nabla\cdot\left[\left(\chi(v_1(s))-\chi(v_2(s))\right)u_1^m(s)\nabla v_1(s)\right]}_pds\nonumber     \\
              & + C \int_0^t \Norm{A^\sigma T(t-s)\nabla\cdot\left[\chi(v_2(s))\left(u_1^m(s)-u_2^m(s)\right)\nabla v_1(s)\right]}_pds\nonumber       \\
              & + C \int_0^t \Norm{A^\sigma T(t-s)\nabla\cdot\left[\chi(v_2(s))u_2^m(s)\left(\nabla v_1(s)-\nabla v_2(s)\right) \right]}_pds\nonumber \\
              & + C (a+\mu) \int_0^t \Norm{A^\sigma T(t-s)\left[u_1(s)-u_2(s)\right]}_pds\nonumber                                                    \\
              & + C \int_0^t \Norm{A^\sigma T(t-s)\left[u_1^{1+\alpha}(s)-u_2^{1+\alpha}(s)\right]}_pds\nonumber                                      \\
    \eqqcolon & H_{1,1} + H_{1,2} + H_{1,3} + H_2+ H_3.
  \end{align}
  To estimate the terms in~\eqref{E:le-s2-mainineq}, we note first that $v_i$
  solves $(\mu I-\Delta)v_i=\nu u_i^\gamma$ with Neumann boundary conditions.
  For $H_{1,1}$, since $\chi(\cdot)$ function is a smooth function with bounded
  derivative, by the Mean Value Theorem,
  \begin{align*}
    \Norm{\chi(v_1)-\chi(v_2)}_{X_T}
    \le C \Norm{v_1-v_2}_{X_T}
     =  C \nu \Norm{(\mu I-\Delta)^{-1}(u_1^\gamma-u_2^\gamma)}_{X_T}
    \le C \Norm{u_1^\gamma-u_2^\gamma}_{X_T}.
  \end{align*}
  For $H_{1,3}$, by the elliptic regularity for $(\mu I-\Delta)^{-1}$,
  $\Norm{\nabla v_1-\nabla v_2}_{X_T} \le C
  \Norm{u_1^\gamma-u_2^\gamma}_{X_T}$. Applying the Mean
  Value Theorem to the power function $g(\xi)=\xi^\gamma$ on the interval
  $[r,R]$, we obtain
  \[
    \Norm{u_1^\gamma-u_2^\gamma}_{X_T}
    \le C \left( \max_{\xi \in [r,R]} \xi^{\gamma-1} \right) \Norm{u_1-u_2}_{X_T}
    < \infty.
  \]
  Similar applications of the Mean Value Theorem to $f(\xi) = \xi^m$ and $h(\xi)
  = \xi^{1+\alpha}$ yield the Lipschitz constants $\max_{\xi\in[r,R]} \xi^{m-1}$
  and $\max_{\xi\in[r,R]} \xi^\alpha$, respectively. Combining these bounds with
  the uniform estimates $\Norm{u_i}_\infty \le R$, $\Norm{\nabla v_i}_\infty \le
  C R^\gamma$, and the estimates for the operator $A^\sigma T(t-s)\nabla$
  in~\eqref{E:new-lm} then leads to the inequalities in~\eqref{E:le-s2-H1}
  below:
  \begin{subequations}\label{E:le-s2-H1}
    \begin{align}
      H_{1,1} & \le C \left( \max_{\xi \in [r,R]} \xi^{\gamma-1}\right) R^{m+\gamma} \left[T^{1-\sigma}+T^{\frac{1}{2}-\sigma}\right]\Norm{u_1-u_2}_{X_T}, \\
      H_{1,2} & \le C \left( \max_{\xi \in [r,R]} \xi^{m-1}\right) R^\gamma \left[T^{1-\sigma}+T^{\frac{1}{2}-\sigma}\right]\Norm{u_1-u_2}_{X_T},          \\
      H_{1,3} & \le C \left( \max_{\xi \in [r,R]} \xi^{\gamma-1}\right) R^m \left[T^{1-\sigma}+T^{\frac{1}{2}-\sigma}\right]\Norm{u_1-u_2}_{X_T}.
    \end{align}
  \end{subequations}
  As for $H_2$ and $H_3$, we instead apply~\eqref{E:FracSemigroup-1} and a
  similar argument as the one used in Step 2 to obtain
  \begin{equation}\label{E:le-s2-H23}
    H_2 \le C T^{1-\sigma}\Norm{u_1-u_2}_{X_T},  \quad \text{and} \quad
    H_3 \le C \left(\max_{\xi \in [r,R]}\xi^{\alpha}\right) T^{1-\sigma}\Norm{u_1-u_2}_{X_T}.
  \end{equation}
  Plugging the estimates in~\eqref{E:le-s2-H1} and~\eqref{E:le-s2-H23} back
  into~\eqref{E:le-s2-mainineq}, there are $0<\kappa<1$ and $T_2 > 0$ such that
  for any $0 < T \le T_2$ and $u_1, u_2 \in S_{r,R,T}$,
  \begin{align*}
    \Norm{G(u_1) - G(u_2)}_{X_T} \le \kappa \Norm{u_1 - u_2}_{X_T}.
  \end{align*}
  Therefore, $G$ is a contraction. The claim then follows with $T = \min\{T_1,
  T_2\}$. \medskip

  As a direct consequence of the contraction mapping principle, there is a
  unique $u \in S_{r,R,T}$ such that $G(u) = u$. In other words, there is a
  unique solution to the system
  \begin{equation}\label{E:local-pf-eq3}
    \begin{dcases}
      u(t,\cdot;u_0) = T(t)u_0-\int_0^t T(t-s)\nabla\cdot(\chi(v(s,\cdot;u_0)) u^m(s,\cdot;u_0)\nabla v(s,\cdot;u_0))ds \cr
      \quad\qquad\qquad  + \int_0^t T(t-s)u(s,\cdot;u_0)(a+\mu-bu^\alpha(s,\cdot;u_0))ds, \cr
      v(t,\cdot;u_0) = (\mu I-\Delta)^{-1}(\nu u^\gamma(t,\cdot;u_0)),
    \end{dcases}
  \end{equation}
  for $t \in (0,T]$. Note that $v$ is determined via $u$. \medskip

  \noindent\textbf{Step~4.~} Repeating the above two steps, we can extend the
  solution $(u(t,x;u_0), v(t,x;u_0))$ of~\eqref{E:local-pf-eq3} to a maximal
  interval $(0, T_{\max}(u_0))$. Along the way, the following property is
  preserved:
  \begin{equation}\label{E:local-pf-eq4}
    u(t,x;u_0)>0\quad \text{for all $t\in [0,T_{\max}(u_0))$ and $x\in\overline\Omega$}.
  \end{equation}
  Moreover, if $T_{\max}(u_0) < \infty$, then one of the cases
  in~\eqref{E:local-Alternative} must occur. \medskip

  \noindent\textbf{Step~5.~} Following the arguments of the proof of Theorem~1.1
  in~\cite{issa.shen:17:dynamics}, we can further show that $(u(t,x;u_0),
  v(t,x;u_0))$, as obtained in the previous steps, is indeed a classical
  solution of~\eqref{E:main-PE}. Expanding the divergence term
  in~\eqref{E:main-PE}, we see that $u = u(t,x;u_0)$ satisfies
  \begin{equation}\label{E:local-pf-eq7}
    \begin{dcases}
      \p_t u =\Delta u-m\chi_0 \frac{u^{m-1}}{(1+v)^\beta}\nabla u\cdot\nabla v +\chi_0\beta \frac{u^m}{(1+v)^{\beta+1}}|\nabla v|^2\cr
      \qquad \quad -\chi_0\frac{u^m}{(1+v)^\beta}(\mu v-\nu u^\gamma)+u(a-bu^\alpha ),\quad & x\in\Omega, \cr
      \frac{\partial u}{\partial n}=0,\quad                                                 & x\in\partial\Omega.
    \end{dcases}
  \end{equation}
  Now suppose that $m \ge 1$ and $T_{\max}(u_0)<\infty$, we will
  establish~\eqref{E:local-infty}. Otherwise, if
  \[
    \limsup_{t\to T_{\max}(u_0)-}\sup_{x\in\Omega}u(t,x;u_0) < \infty,
  \]
  then from~\eqref{E:local-Alternative}, we have that
  \begin{align}\label{E:local-contra}
    \liminf_{t \to T_{\max}(u_0)-}\inf_{x\in\Omega}u(t,x;u_0) = 0.
  \end{align}
  From~\eqref{E:bound-v+grad-v}, there is $M > 0$ such that
  \begin{align*}
    \max\left\{\Norm{v}_\infty, \Norm{\nabla v}_\infty \right\} \le M \quad {\rm on}\quad [0,T_{\max}(u_0)).
  \end{align*}
  Hence, we have that the solution of the following ODE is a sub-solution
  of~\eqref{E:local-pf-eq7}:
  \begin{align*}
    \begin{dcases}
      \frac{d\underline u}{dt}=-|\chi_0|\beta M^2 \underline u^m-|\chi_0|\mu M
      \underline u^m -|\chi_0|\nu \underline u^{m+\gamma}+a\underline u-b \underline
      u^{1+\alpha},\cr
      \underline u(0) = \inf_{x\in\Omega} u_0(x) > 0.
    \end{dcases}
  \end{align*}
  Noting that $\underline u(0)\le u_0(x)$, by comparison principle for parabolic
  equations, we see that $u(t,x;u_0)\ge \underline u(t)$ for all $(t,x) \in [0,
  T_{\max}(u_0)) \times \Omega$. The condition $m\ge 1$ implies that $
  \displaystyle \inf_{t\in [0,T_{\max}(u_0))}\underline u(t) > 0$. Hence, $
  \displaystyle \liminf_{t\to T_{\max}(u_0)-} \inf_{x\in\Omega}u(t,x;u_0) > 0$,
  which contradicts~\eqref{E:local-contra}. Therefore, \eqref{E:local-infty}
  holds under the conditions that $T_{\max}(u_0) < \infty$ and $m \ge 1$.
  \medskip

  \noindent\textbf{Step~6.~} In this final step, we claim that the classical
  solutions of~\eqref{E:main-PE} satisfying the conclusions in
  Proposition~\ref{P:local-existence} are unique. The claim follows from the
  same arguments as the proof of Theorem 1.1 in~\cite{issa.shen:17:dynamics}. We
  leave the details for the interested reader. \medskip

  With the above six steps, we complete the whole proof of
  Proposition~\ref{P:local-existence}.
\end{proof}

\subsection{Entropy and bootstrap tools}\label{SS:OtherLemmas}

We collect two tools that we invoke repeatedly across the proof sections below.
The first is an elementary entropy-type bound for the sensitivity coefficient
that provides a uniform estimate in terms of the constant $\Psi_\beta$. The
second packages the bootstrapping mechanism from a single $L^{p_0}$ bound and a
generic differential inequality into a reusable abstract proposition, together
with its chemotaxis corollary applied to our specific setting.

\begin{lemma}\label{L:v-entropy}
  For all $\beta > 0$ and $v > 0$, the following inequality holds with equality
  achieved at $v = 1/\beta$:
  \begin{equation}\label{E:v-entropy}
    \frac{\beta}{(1+v)^{1+\beta}}
    \le \frac{\Psi_\beta}{v},
    \qquad
    \text{with $\Psi_\beta:=\Big(\tfrac{\beta}{1+\beta}\Big)^{1+\beta}$.}
  \end{equation}
  Moreover, the constant $\Psi_\beta$ (see Fig.~\ref{fig:Psi-Theta} for a plot
  of this constant) as a function of $\beta$ is strictly increasing with
  \[
    \lim_{\beta \downarrow 0} \Psi_\beta = 0 \quad \text{and} \quad
    \lim_{\beta \uparrow \infty} \Psi_\beta = \frac{1}{e}.
  \]
\end{lemma}
\begin{proof}
  The claimed inequality is equivalent to $\frac{\beta\,v}{(1+v)^{1+\beta}} \le
  \Psi_\beta$. Define $h(v) \coloneqq \frac{\beta\,v}{(1+v)^{1+\beta}}$ for $v >
  0$. A direct computation gives
  \[
    \frac{h'(v)}{h(v)}
    =\frac{1}{v}-\frac{1+\beta}{1+v}
    =\frac{1-\beta v}{v(1+v)}.
  \]
  Hence, $h'(v) = 0$ if and only if $v = 1/\beta$; moreover $h'(v) > 0$ for $0 <
  v < 1/\beta$ and $h'(v) < 0$ for $v > 1/\beta$. Thus, $v = 1/\beta$ is the
  unique global maximizer of $h$ on $(0,\infty)$, and
  \[
    \sup_{v>0} h(v)=h(1/\beta)
    = \frac{\beta\,(1/\beta)}{\bigl(1+1/\beta\bigr)^{1+\beta}}
    = \Bigl(\frac{\beta}{1+\beta}\Bigr)^{1+\beta}
    = \Psi_\beta .
  \]
  Therefore, $h(v)\le \Psi_\beta$ for all $v > 0$. Finally, let $g(\beta)
  \coloneqq \ln \Psi_\beta = (1+\beta)\ln\left(\frac{\beta}{1+\beta}\right)$.
  Then we see that $g'(\beta) = \frac{1}{\beta} +
  \ln\left(\frac{\beta}{1+\beta}\right)$ and $g''(\beta) =
  -\frac{1}{\beta^2(1+\beta)} < 0$. Therefore, $g$ is strictly concave. Since
  $g'(\beta) \to 0$ as $\beta \to \infty$, we see that $g'(\beta) < 0$ for all
  $\beta > 0$. Hence, $g(\beta)$ is strictly increasing. The two limits of
  $\Psi_\beta$ can be verified by direct computation. This completes the proof
  of Lemma~\ref{L:v-entropy}.
\end{proof}

\begin{lemma}[Abstract $L^p$ bootstrap]\label{L:bootstrap}
  Let $u \ge 0$ be defined on $(0, T) \times \Omega$ with $T \in (0,\infty]$ and
  fix $\rho > 0$. Assume that there exists $p_0 > \max\{1, \rho N/2\}$ such that
  $\displaystyle \sup_{t \in (0, T)} \int_\Omega u^{p_0} < \infty$. Suppose
  moreover that $u$ is regular enough so that for each $p \ge p_0$ the
  quantities in the inequality below with $w \coloneqq u^{p/2}$ are finite.
  Assume further that for every $p \ge p_0$ there exist constants $A = A(p,\rho)
  > 0$, $B = B(p, \rho) > 0$, $K = K(p, \rho) > 0$, and $L = L(p, \rho) \in \R$
  such that for all $t \in (0, T)$,
  \begin{equation}\label{E:bootstrap}
    \frac{1}{p} \cdot \frac{d}{dt} \int_\Omega w^2
    + A \int_\Omega |\nabla w|^2 + B \int_\Omega w^2
    \le K \int_\Omega w^{2 + \frac{2\rho}{p}} + L.
  \end{equation}
  Then for every $p > 1$, we have that $\displaystyle \sup_{t \in (0, T)}
  \int_\Omega u^p < \infty$.
\end{lemma}

\begin{proof}
  We argue by iteration. Throughout the proof, $C$ denotes a positive constant
  independent of $t \in (0,T)$ whose value may change from line to line.

  Fix $p \in (p_0, 2p_0)$ and let $w = u^{p/2}$. By the
  hypothesis~\eqref{E:bootstrap} with this $p$,
  \begin{align}
    \label{E_:bootstrap-1}
    \frac{1}{p} \cdot \frac{d}{dt} \int_{\Omega} w^2
    \le - A \int_{\Omega}|\nabla w|^2 - B \int_\Omega w^2
       + K  \int_{\Omega} w^{2+\frac{2\rho}{p}} + L.
  \end{align}
  Applying the Gagliardo--Nirenberg inequality (Theorem~1
  in~\cite{nirenberg:66:extended}) with $p_1 = 2 + \frac{2\rho}{p}$ and $q_1 =
  \frac{2p_0}{p} \in (1,2)$ gives
  \begin{equation}\label{E_:bootstrap-2}
    \Norm{w}_{p_1} \le C \Norm{\nabla w}_2^\theta \Norm{w}_{q_1}^{1-\theta}+C \Norm{w}_{q_1},
  \end{equation}
  where $\theta$ satisfies
  \begin{align*}
    \frac{1}{p_1} = \Big(\frac{1}{2}-\frac{1}{N}\Big)\theta+\frac{1-\theta}{q_1} 
    \quad \Longleftrightarrow \quad 
    \theta = \frac{p(p+\rho-p_0)N}{(p+\rho)(pN+2p_0-p_0 N)}.
  \end{align*}
  Note that $\Norm{w}_{q_1}^{q_1} = \Norm{u}_{p_0}^{p_0}$, which is bounded by
  assumption. By~\eqref{E_:bootstrap-2},
  \begin{align*}
    \int_\Omega w^{p_1}
    = \int_\Omega w^{2+\frac{2 \rho}{p}}
    \le C \left(\int_\Omega |\nabla w|^2\right)^{\frac{\theta p_1}{2}} + C.
  \end{align*}
  Moreover,
  \begin{align*}
    \frac{\theta p_1}{2}
    = \theta \cdot \frac{p+\rho}{p}
    = \frac{N(p-p_0)+\rho N}{N(p-p_0)+2p_0} < 1
    \quad \text{since $p_0 > \frac{\rho N}{2}$.}
  \end{align*}
  By Young's inequality,
  \begin{align}\label{E_:bootstrap-3}
    K \int_\Omega w^{2+\frac{2 \rho}{p}}
    \le A \int_\Omega |\nabla w|^2 + C.
  \end{align}
  Combining~\eqref{E_:bootstrap-1} and~\eqref{E_:bootstrap-3} yields
  \begin{align*}
    \frac{1}{p} \cdot \frac{d}{dt} \int_{\Omega} w^2
    \le - B \int_\Omega w^2 + C.
  \end{align*}
  An application of Gr\"onwall's inequality then shows that $\sup_{t \in (0,T)}
  \int_\Omega u^p < \infty$ for $p \in (p_0, 2p_0)$.

  Repeating the argument with $p_0$ replaced by $\frac{3 p_0}{2}$ extends the
  bound to $p \in (p_0, 3p_0)$. Iterating this step shows boundedness for all $p
  > p_0$. Finally, boundedness for $p \in (1, p_0]$ follows by H\"older's
  inequality and the boundedness of $\int_\Omega u^{p_0}$. This completes the
  proof of Lemma~\ref{L:bootstrap}.
\end{proof}

The following corollary will be applied to the proof of
Theorem~\ref{T:weak-cross-diffusion}(2) with $\rho = \gamma$ and
Theorem~\ref{T:strong-logistic-source}(1) (iii)--(iv) with $\rho = \alpha$ and
$T=T_{\max}(u_0)$.

\begin{corollary}[Bootstrap for the chemotaxis cross term]\label{C:bootstrap}
  Let $(u,v)$ be a positive classical solution of~\eqref{E:main-PE} on $(0,T)$
  with $T\in (0,\infty]$ and parameters given in~\eqref{E:parameters}. Assume
  that there is $\rho > 0$ satisfying that for any $\epsilon > 0$ and $p>1$,
  there is $C_{\epsilon,p}$ such that 
  \begin{equation}\label{E:C-bootstrap}
    \chi_0\int_\Omega u^{p+m-2}\frac{\nabla u\cdot\nabla v}{(1+v)^{\beta}}\le
    \epsilon \int_\Omega u^{p-2} |\nabla u|^2 + C_{\epsilon,p}  \int_\Omega u^{p+\rho}\quad \forall\, t \in  (0,T).
  \end{equation}
  Assume also that there is $p_0>\max\{1, \rho N/2\}$ such that $\displaystyle
  \sup_{t\in \left(0, T\right)}\int_\Omega u^{p_0} \le C$. Then for any $p>1$,
  \begin{equation}\label{E:C-boundedness}
   \limsup_{t\to T-}\int_\Omega u^p<\infty.
  \end{equation}
\end{corollary}

\begin{proof}
  Fix $p>1$ and set $w \coloneqq u^{p/2}$. Let $C_{1/2}$ be the constant
  in~\eqref{E:C-bootstrap} corresponding to $\epsilon = 1/2$ (for this $p$).
  By~\eqref{E:u-lp-eq1},
  \begin{align*}
    \frac{1}{p} \cdot \frac{d}{dt} \int_\Omega w^2
      & = -(p-1)\int_\Omega u^{p-2}|\nabla u|^2
          + (p-1)\chi_0\int_\Omega \frac{u^{p+m-2}}{(1+v)^\beta} \nabla u \cdot \nabla v
          + a\int_\Omega u^p - b\int_\Omega u^{p+\alpha} \\
      &\le-(p-1)\int_\Omega u^{p-2}|\nabla u|^2
          + (p-1)\chi_0\int_\Omega \frac{u^{p+m-2}}{(1+v)^\beta} \nabla u \cdot \nabla v
          + (a+1)\int_\Omega w^2 - \int_\Omega w^2.
  \end{align*}
  Using~\eqref{E:C-bootstrap} with $\epsilon = 1/2$ yields
  \begin{align*}
    (p-1)\chi_0\int_\Omega \frac{u^{p+m-2}}{(1+v)^\beta}\nabla u\cdot\nabla v
    \le \frac{p-1}{2}\int_\Omega u^{p-2}|\nabla u|^2
        + (p-1)C_{\frac{1}{2},p}\int_\Omega u^{p+\rho}.
  \end{align*}
  Combining the previous two inequalities yields
  \begin{align*}
    \frac{1}{p} \cdot \frac{d}{dt} \int_\Omega w^2
      &\le -\frac{p-1}{2}\int_\Omega u^{p-2}|\nabla u|^2
          + (p-1)C_{\frac{1}{2},p}\int_\Omega u^{p+\rho}
          + (a+1)\int_\Omega w^2 - \int_\Omega w^2.
  \end{align*}
  Since $\nabla w = \frac{p}{2}u^{\frac{p}{2}-1}\nabla u$, we have
  \begin{equation}\label{E:Norm-Gradient-w}
    \Norm{\nabla w}_2^2 = \frac{p^2}{4} \int_\Omega u^{p-2}|\nabla u|^2.
  \end{equation}
  Moreover, we have $\int_\Omega u^{p+\rho} = \int_\Omega
  w^{2+\frac{2\rho}{p}}$. Hence,
  \begin{align*}
    \frac{1}{p} \cdot \frac{d}{dt} \int_\Omega w^2
    + \frac{2(p-1)}{p^2}\int_\Omega |\nabla w|^2 + \int_\Omega w^2
    \le (p-1)C_{\frac{1}{2},p}\int_\Omega w^{2+\frac{2\rho}{p}} + (a+1)\int_\Omega w^2.
  \end{align*}
  Since $\rho>0$, we have $2+\frac{2\rho}{p}>2$, and therefore $w^2 \le 1 +
  w^{2+\frac{2\rho}{p}}$. It follows that
  \[
    (a+1)\int_\Omega w^2
    \le (a+1)|\Omega| + (a+1)\int_\Omega w^{2+\frac{2\rho}{p}}.
  \]
  Combining the previous two inequalities, we obtain that for all
  $t\in(0,T_{\max})$,
  \begin{align*}
    \frac{1}{p} \cdot \frac{d}{dt} \int_\Omega w^2
    + A\int_\Omega |\nabla w|^2 + B\int_\Omega w^2
    \le K\int_\Omega w^{2+\frac{2\rho}{p}} + L,
  \end{align*}
  where
  \[
    A=\frac{2(p-1)}{p^2},
    \qquad
    B=1,
    \qquad
    K=(p-1)C_{\frac{1}{2},p}+(a+1),
    \qquad
    L=(a+1)|\Omega|.
  \]
  This is exactly~\eqref{E:bootstrap} in Lemma~\ref{L:bootstrap}. Together with
  the condition $\sup_{t\in \left(0, T_{\max}\right)} \int_\Omega u^{p_0} \le
  C$, we can apply Lemma~\ref{L:bootstrap} to yield~\eqref{E:C-boundedness}.
  This completes the proof of Corollary~\ref{C:bootstrap}.
\end{proof}

\subsection{Fundamental estimates for $v$ and $\nabla v$} \label{SS:Fundamental-v}

In this subsection, we present several fundamental estimates for $v$ and $\nabla
v$ in terms of $u$ to be used in the proofs of the main theorems. Throughout
this subsection, we assume that $u_0$ satisfies~\eqref{E:initial-cond-PE}, and
$(u(t,x;u_0), v(t,x;u_0))$ is the unique classical solution of~\eqref{E:main-PE}
on the maximal interval $(0, T_{\max}(u_0))$ with initial condition $u(0,x;u_0)
= u_0(x)$. If no confusion occurs, we may drop $u_0$ in $(u(t,x;u_0),
v(t,x;u_0))$ and $T_{\max}(u_0)$.

\begin{proposition}\label{P:Main2}
  Suppose that $u_0$ satisfies~\eqref{E:initial-cond-PE}. For all $p \in [1,
  \infty]$ and $t\in (0,T_{\max}(u_0))$, it holds
  \begin{equation}\label{E:resolvent-Lp}
    \Norm{v(t, \cdot)}_p \le \frac{\nu}{\mu}\: \Norm{u^\gamma(t, \cdot)}_p.
  \end{equation}
  In particular, if $1\le p<\infty$, then
  \[
    \int_\Omega v^p(t, \cdot)\,dx \le \left(\frac{\nu}{\mu}\right)^{p}
    \int_\Omega u^{\gamma p}(t, \cdot) dx .
  \]
\end{proposition}
\begin{proof}
  Recall the expression of $v$ in terms of $u$ given in~\eqref{E:Resolvent}.
  For $s > 0$ and $x, y \in \Omega$, let $K(s, x, y)$ denote the Neumann heat
  kernel. It is well known that $K(s, x, y)$ is nonnegative and $\int_\Omega
  K(s, x, y) dy = 1$. Fix $s > 0$. The Neumann heat semigroup $e^{s\Delta}$ is
  a contraction on both $L^1(\Omega)$ and $L^\infty(\Omega)$ spaces. Indeed,
  from~\eqref{E:Resolvent}, we have
  \begin{gather*}
    \Norm{e^{s\Delta}f}_1
    = \int_\Omega \left| \int_\Omega K(s,x,y) f(y) dy \right| dx
    \le \int_\Omega \int_\Omega K(s,x,y) |f(y)| dy dx
    = \int_\Omega |f(y)| dy = \Norm{f}_1 \quad \text{and}\\
    \Norm{e^{s\Delta}f}_\infty
    = \sup_{x\in\Omega} \left| \int_\Omega K(s,x,y) f(y) dy \right|
    \le \sup_{x\in\Omega} \int_\Omega K(s,x,y) |f(y)| dy
    \le \Norm{f}_\infty.
  \end{gather*}
  Therefore, by the Riesz-Thorin interpolation theorem (see, e.g., Theorem 6.27
  in~\cite{folland:99:real}), the Neumann heat semigroup $e^{s\Delta}$ is a
  contraction on $L^p(\Omega)$ spaces as well for all $p \in [1, \infty]$.
  Hence,
  \[
    \Norm{v}_p
    \le \nu \int_0^\infty e^{-\mu s} \Norm{e^{s\Delta}(u^\gamma)}_p ds
    \le \nu \int_0^\infty e^{-\mu s} ds \Norm{u^\gamma}_p
    = \frac{\nu}{\mu} \Norm{u^\gamma}_p.
  \]
  This proves~\eqref{E:resolvent-Lp}. When $p$ is finite, we can raise both
  sides to the $p$-th power to yield the integral form.
\end{proof}

% \textcolor{magenta}{
% \begin{lemma}\label{L:v-entropy}
%   For all $\beta > 0$ and $v > 0$, the following inequality holds:
%   $\displaystyle \frac{\beta}{(1+v)^{1+\beta}} \le \frac{1}{v}$.
% \end{lemma}
% \begin{proof}
%   The inequality is equivalent to $(1+v)^{1+\beta} \ge \beta v$. Define $f(v)
%   \coloneqq (1+\beta)\ln(1+v) - \ln(\beta v)$ for $v > 0$. Then
%   \[
%     f'(v) = \frac{1+\beta}{1+v} - \frac{1}{v}
%           = \frac{\beta v - 1}{v(1+v)}.
%   \]
%   Hence, $f'(v) = 0$ when $v = 1/\beta$. Moreover, $f'(v) < 0$ for $v <
%   1/\beta$ and $f'(v) > 0$ for $v > 1/\beta$, showing that $v = 1/\beta$ gives
%   the global minimum of $f$. Therefore, by evaluating at the critical point,
%   \[
%     \inf_{v > 0} f(v) \ge
%     f\!\left(\frac{1}{\beta}\right)
%       = (1+\beta)\ln\left(1+\frac{1}{\beta}\right) - \ln 1
%       = (1+\beta)\ln\left(1+\frac{1}{\beta}\right) > 1,
%   \]
%   which is equivalent to $(1+v)^{1+\beta} \ge \beta v$. This proves the desired
%   inequality.
% \end{proof}
% }

\begin{proposition}\label{P:Main1}
  For all $p > 1$, $\beta \ge 0$, and $t\in (0,T_{\max}(u_0))$, it holds that
  \begin{align}\label{E:Main-P0}
    \int_{\Omega}\frac{ |\nabla v|^{2p}}{ v^{p}}
    \le M^* \int_\Omega u^{\gamma p}(t,x;u_0),
  \end{align}
  and
  \begin{equation}\label{E:Main-P1}
    \int_{\Omega} \frac{|\nabla v(t,x;u_0)|^{2p}}{(1+v(t,x;u_0))^{{(1+\beta)}p}}
    \le {\Theta_\beta^p}\, M^* \int_{\Omega} u^{\gamma p}(t,x;u_0),
  \end{equation}
  where the constant $M^* = M^*(N,p,\mu,\nu)$ is given in~\eqref{E:M-star} and
  $\Theta_{\beta}$ is given in~\eqref{E:Psi-Theta} (see
  Fig.~\ref{fig:Psi-Theta} for a plot of $\Theta_\beta$).
\end{proposition}

\begin{proof}
  Fix an arbitrary $t \in (0,T_{\max}(u_0))$. We note that~\eqref{E:Main-P0}
  implies \eqref{E:Main-P1}, which follows from Lemma~\ref{L:v-entropy}:
  \begin{equation}\label{E:Main1-Step2}
    \int_{\Omega}\frac{|\nabla v|^{2p}}{(1+v)^{(1+\beta)p}}
    \le \Psi_\beta^p\, \beta^{-p} \int_{\Omega}\frac{|\nabla v|^{2p}}{v^{p}}
    \le \Psi_\beta^p\, \beta^{-p}\,M^*\int_{\Omega} u^{\gamma p}
    = \Theta_\beta^p\, M^*\int_{\Omega} u^{\gamma p}.
  \end{equation}

  It remains to prove~\eqref{E:Main-P0}. The proof follows similar arguments
  to those in~\cite[Proposition~1.3]{kurt.shen:20:finite-time}. Nevertheless,
  here we provide a simpler proof. First, notice that
  \begin{align*}
    \int_{\Omega} \frac{|\nabla v|^{2p}}{v^p}
    & =  \int_{\Omega}  \frac{|\nabla v|^{2p-2}}{v^p} \nabla v\cdot \nabla v                                                                \\
    & = -\int_{\Omega} v \nabla \cdot \Big(|\nabla v|^{2p-2}v^{-p}\nabla v\Big)                                                             \\
    & = -\int_{\Omega} v\nabla \Big (\big(|\nabla v|^2\big)^{p-1} v^{-p}\Big)\cdot\nabla v- \int_{\Omega} v^{-p+1}|\nabla v|^{2p-2}\Delta v \\
    & = p\int_{\Omega}v^{-p}|\nabla v|^{2p} -(p-1)\int_{\Omega} v^{-p+1}|\nabla v|^{2p-4}\nabla (|\nabla v|^2)\cdot \nabla v-\int_{\Omega}v^{-p+1}|\nabla v|^{2p-2}\Delta v.
  \end{align*}
  This implies that
  \begin{align*}
    \int_{\Omega} \frac{|\nabla v|^{2p}}{v^p}
    & =\int_{\Omega} v^{-p+1}|\nabla v|^{2p-4} \nabla(|\nabla v|^2)\cdot \nabla v+\frac{1}{p-1}\int_{\Omega }v^{-p+1}|\nabla v|^{2p-2}\Delta v\nonumber \\
    & \le \int_{\Omega} v^{-p+1}|\nabla v|^{2p-4} \nabla(|\nabla v|^2)\cdot \nabla v +
    \frac{1}{2}\int_{\Omega}|\nabla v|^{2p} v^{-p} +
    \frac{2^{p-1}}{(p-1) p^p}\int_{\Omega }|\Delta v|^p,
  \end{align*}
  where we have applied Young's inequality in the last step. Hence,
  \begin{equation}\label{E:main-inequality-eq-2}
    \int_{\Omega} \frac{|\nabla v|^{2p}}{ v^{p}}
    \le 2 \int_{\Omega} v^{-p+1}|\nabla v|^{2p-4} \nabla(|\nabla v|^2)\cdot \nabla v
        + \frac{2^p}{(p-1) p^p}\int_{\Omega }|\Delta v|^p.
  \end{equation}

  Next, a direct calculation yields
  \begin{equation*}
    \nabla(|\nabla v|^2)\cdot \nabla v=2\nabla v \cdot\big(D^2 v \nabla v\big)\le 2|\nabla v|^2 |D^2 v|.
  \end{equation*}
  This together with \eqref{E:main-inequality-eq-2} and another application of
  Young's inequality implies that
  \begin{align*}
    \int_{\Omega}\frac{ |\nabla v|^{2p}}{ v^{p}}
    & \le 4\int_{\Omega} v^{-p+1}|\nabla v|^{2p-2}|D^2 v| + \frac{2^p}{(p-1)p^p}\int_{\Omega}|\Delta v|^p\nonumber \\
    & \le \frac{1}{2}\int_{\Omega} |\nabla v|^{2p} v^{-p}
          +\frac{2^{p-1} 4^p}{p}\int_\Omega|D^2 v|^p + \frac{2^p}{(p-1) p^p}\int_{\Omega}|\Delta v|^p,
  \end{align*}
  Therefore,
  \begin{equation}\label{E:main-inequality-eq-3}
    \int_{\Omega}\frac{ |\nabla v|^{2p}}{ v^{p}}\le \frac{8^p}{p}\int_{\Omega}|D^2 v|^p + \frac{2^p}{(p-1) p^p}\int_{\Omega} |\Delta v|^p.
  \end{equation}
  By~\eqref{E:main-inequality-eq-4-0} and~\eqref{E:C-star},
  \begin{equation}\label{E:main-inequality-eq-4}
    \int_{\Omega}|D^2 v|^p
    \le C_{N,p}^*\int_\Omega \Big(|\Delta v|^p+v^p\Big)\quad \text{for all $v\in W^{2,p}(\Omega)$}.
  \end{equation}
  Combining~\eqref{E:main-inequality-eq-3} and~\eqref{E:main-inequality-eq-4},
  we have
  \begin{align*}
    \int_{\Omega}\frac{ |\nabla v|^{2p}}{ v^{p}}
    & \le \left(\frac{8^p}{p}C_{N,p}^*+\frac{2^p}{(p-1)p^p}\right)\int_{\Omega}|\Delta v|^p + \frac{8^p}{p}C_{N,p}^*\int_\Omega v^p.
  \end{align*}
  Now, by the second equation of~\eqref{E:main-PE} and
  Proposition~\ref{P:Main2}, we have that
  \begin{align*}
    \int_\Omega |\Delta v|^p
    =   & \int_{\Omega} |\mu v-\nu u^{\gamma}|^p
    \le   2^{p-1}\int_\Omega(|\mu v|^p + |\nu u^\gamma|^p) \\
    \le & 2^{p-1}\int_\Omega \left[\left(\frac{\nu}{\mu}\right)^p \mu^p u^{\gamma p} + \nu^p u^{\gamma p}\right]
    =     (2\nu)^p \int_\Omega u^{\gamma p}.
  \end{align*}
  The inequality~\eqref{E:Main-P0} is then proved by plugging the above
  inequality into the previous one and applying Proposition~\ref{P:Main2} again
  for the last term. This completes the proof of Proposition~\ref{P:Main1}.
\end{proof}

\begin{proposition}\label{P:Main3}
  For given $u_0$ satisfying~\eqref{E:initial-cond-PE}, for any $p >
  \max\{1,\beta\}$ and $\epsilon > 0$, there is $C_\epsilon > 0$ such that for
  all $t \in (0, T_{\max}(u_0))$,
  \begin{align*}
    \int_\Omega \frac{v^{p+1}(t,\cdot;u_0)}{(1+v(t,\cdot;u_0) )^\beta}
    \le  \epsilon \int_\Omega\frac{u^{\gamma(p+1)}(t,\cdot;u_0)}{(1+v(t,\cdot;u_0))^\beta}
       + C_\epsilon \left(\int_\Omega\frac{ v(t,\cdot;u_0)}{(1+v(t,\cdot;u_0))^{\frac{\beta}{p+1}}}\right)^{p+1} .
  \end{align*}
\end{proposition}

\begin{proof}
  First, by the second equation in~\eqref{E:main-PE}, for any $p > 1$,
  \begin{align}\label{E:new-v-estimate-eq1}
    0 & =\int_\Omega \frac{v^p}{(1+v)^\beta}\Delta v-\mu \int_\Omega\frac{v^{p+1}}{(1+v)^\beta}+\nu\int_\Omega \frac{u^\gamma v^p}{(1+v)^\beta}\nonumber                                                                                 \\
      & =-p\int_\Omega \frac{v^{p-1}}{(1+v)^\beta}|\nabla v|^2+\beta\int_\Omega\frac{v^p}{(1+v)^{\beta+1}}|\nabla v|^2-\mu \int_\Omega\frac{v^{p+1}}{(1+v)^\beta}+\nu\int_\Omega \frac{u^\gamma v^p}{(1+v)^\beta}\nonumber \\
      & \le -(p-\beta)\int_\Omega \frac{v^{p-1}}{(1+v)^\beta}|\nabla v|^2-\mu \int_\Omega\frac{v^{p+1}}{(1+v)^\beta}+\nu\int_\Omega \frac{u^\gamma v^p}{(1+v)^\beta}
  \end{align}
  for all $t \in (0,T_{\max}(u_0))$.  Note that there is $C_1=C_1(\mu,\nu)>0$
  such that
  \begin{align*}
    \nu\int_\Omega \frac{u^\gamma v^p}{(1+v)^\beta}\le \mu \int_\Omega \frac{v^{p+1}}{(1+v)^\beta}+C_1\int_\Omega \frac{u^{\gamma(p+1)}}{(1+v)^\beta}.
  \end{align*}
  This together with~\eqref{E:new-v-estimate-eq1} implies that
  \begin{equation}\label{E:new-v-estimate-eq2}
    \int_\Omega \frac{v^{p-1}}{(1+v)^\beta} |\nabla v|^2
    \le \frac{C_1}{p-\beta} \int_\Omega \frac{u^{\gamma(p+1)}}{(1+v)^\beta}
  \end{equation}
  for all $p > \beta$ and $t \in (0,T_{\max}(u_0))$.

  Next, by Ehrling's lemma (see, e.g., Exercise~6.12, on p.~173, under the name
  ``J.-L. Lion's Lemma'' of~\cite{brezis:11:functional}), for any $\delta > 0$,
  there is $C_{1,\delta} > 0$ such that
  \begin{align}\label{E:new-v-estimate-eq3}
        \int_\Omega \frac{v^{p+1}}{(1+v)^\beta}
    =   \int_\Omega \left(\frac{v^{\frac{p+1}{2}}}{(1+v)^{\frac{\beta}{2}}}\right)^2
    \le \delta \int_\Omega\left|\nabla  \left(\frac{v^{\frac{p+1}{2}}}{(1+v)^{\frac{\beta}{2}}}\right)\right|^2
        + C_{1,\delta} \left(\int_\Omega \frac{v^{\frac{p+1}{2}}}{(1+v)^{\frac{\beta}{2}}}\right)^2.
  \end{align}
  Observe that
  \begin{align}\label{E:new-v-estimate-eq4}
    \int_\Omega\left|\nabla \left(\frac{v^{\frac{p+1}{2}}}{(1+v)^{\frac{\beta}{2}}}\right)\right|^2
    & = \int_\Omega \left| \frac{p+1}{2} \frac{v^{\frac{p-1}{2}}}{(1+v)^{\frac{\beta}{2}}}\nabla v -\frac{\beta}{2}\frac{v^{\frac{p+1}{2}}}{(1+v)^{\frac{\beta}{2}+1}}\nabla v\right|^2\nonumber \\
    & \le \frac{(p+1)^2}{2}\int_\Omega \frac{v^{p-1}}{(1+v)^\beta}|\nabla v|^2 +\frac{\beta^2}{2}\int_\Omega \frac{v^{p+1}}{(1+v)^{\beta+2}}|\nabla v|^2\nonumber                                \\
    & \le \left(\frac{(p+1)^2}{2}+\frac{\beta^2}{2}\right) \int_\Omega \frac{v^{p-1}}{(1+v)^\beta}|\nabla v|^2\nonumber                                                                          \\
    & \le \frac{C_1}{p - \beta} \left(\frac{(p+1)^2}{2}+\frac{\beta^2}{2}\right) \int_\Omega \frac{u^{\gamma(p+1)}}{(1+v)^\beta},
  \end{align}
  where the last inequality is due to~\eqref{E:new-v-estimate-eq2}. By
  H\"older's inequality and Young's inequality, for $\epsilon' = \frac{1}{2
  C_{1, \delta}}$, there is a constant $C_{2, \delta} > 0$ such that
  \begin{align}\label{E:new-v-estimate-eq5}
    \left(\int_\Omega \frac{v^{\frac{p+1}{2}}}{(1+v)^{\frac{\beta}{2}}}\right)^2
     =   & \left(\int_\Omega \frac{v^{\frac{p}{2}}}{(1+v)^{\frac{\beta}{2}\frac{p}{p+1}}} \times \frac {v^{\frac{1}{2}}}{(1+v)^{\frac{\beta}{2}\frac{1}{p+1}}}\right)^2\nonumber                 \\
     \le & \left(\int_\Omega \frac{v^p}{(1+v)^{\frac{\beta p}{p+1}}}\right)\left(\int_\Omega \frac{ v}{(1+v)^{\frac{\beta}{p+1}}}\right)\nonumber                                                \\
     \le & \frac{1}{2 C_{1, \delta}} \left(\int_\Omega \frac{v^p}{(1+v)^{\frac{\beta p}{p+1}}}\right)^{\frac{p+1}{p}} + C_{2, \delta}\left(\int_\Omega \frac{v}{(1+v)^{\frac{\beta}{p+1}}}\right)^{p+1}\nonumber \\
     \le & \frac{1}{2 C_{1, \delta}} \int_\Omega \frac{v^{p+1}}{(1+v)^\beta} + C_{2, \delta}\left(\int_\Omega \frac{v}{(1+v)^{\frac{\beta}{p+1}}}\right)^{p+1}.
  \end{align}
  Combining~\eqref{E:new-v-estimate-eq3}, \eqref{E:new-v-estimate-eq4},
  and~\eqref{E:new-v-estimate-eq5}, we complete the proof of
  Proposition~\ref{P:Main3} with
  \begin{align*}
    \epsilon   = 2 \delta \frac{C_1}{p - \beta} \left(\frac{(p+1)^2}{2}+\frac{\beta^2}{2}\right) \quad \text{and} \quad
    C_\epsilon = 2 C_{1, \delta} C_{2, \delta}, \quad \text{for all $\delta > 0$.}
  \end{align*}
  This completes the proof of Proposition~\ref{P:Main3}.
\end{proof}

\subsection{Fundamental estimates for $u$} \label{SS:Fundamental-u}

In this subsection, we present several fundamental estimates for $u$ to be used
in the proofs of the main theorems. Throughout this subsection, we also assume
that $u_0$ satisfies~\eqref{E:initial-cond-PE}, and $(u(t,x;u_0), v(t,x;u_0))$
is the unique classical solution of~\eqref{E:main-PE} on the maximal interval
$(0, T_{\max}(u_0))$ with initial condition $u(0,x;u_0) = u_0(x)$. Again, if no
confusion occurs, we may drop $u_0$ in $(u(t,x;u_0), v(t,x;u_0))$ and
$T_{\max}(u_0)$. The next proposition is a direct consequence of the comparison
principle, the proof of which is left for the interested reader.

\begin{proposition}\label{P:Main-0}
  For any $u_0$ that satisfies~\eqref{E:initial-cond-PE}, the following hold:
  \begin{itemize}
    \item[(1)] If $a=b=0$, then
      \begin{equation}\label{E:constant-mass}
        \int_{\Omega} u(t,x;u_0)dx =\int_{\Omega} u_0(x)dx  \quad \text{\rm for all}\,\,  t\in \left(0,T_{\max}(u_0)\right) .
      \end{equation}

  \item[(2)] If $a,b>0$, then
    \begin{equation*}
      \int_{\Omega} u(t,x;u_0)dx \le m^* \coloneqq {\max}\left\{\int_{\Omega} u_0(x)dx,\:  \left(\frac{a}{b}\right)^{1/\alpha}|\Omega| \right\} \quad \text{\rm for all}\,\,  t\in \left(0,T_{\max}(u_0)\right) ,
    \end{equation*}
    where $|\Omega|$ refers to the Lebesgue measure of $\Omega$.
  \end{itemize}
\end{proposition}

Before we proceed, we first note that, for all $t\in [0,T_{\max}(u_0))$, the
resolvent $(\mu I-\Delta)^{-1}$ can be written as in \eqref{E:Resolvent}, and
the mild form of $u$ reads
\begin{align}\label{E:u-mild}
    u(t,\cdot)
    & = e^{-At}u_0 - \chi_0 \int_0^t e^{-A(t-s)} \nabla\cdot\!\left(\frac{u^m(s,\cdot)}{(1+v(s,\cdot))^\beta}\nabla v(s,\cdot)\right)ds \nonumber\\
    & \quad + \int_0^t e^{-A(t-s)} u(s,\cdot)\big(\mu+a - b u^\alpha(s,\cdot)\big)\,ds.
  \end{align}

\begin{proposition}\label{P:Lp->GlobalExist}
  Let $u_0$ satisfy~\eqref{E:initial-cond-PE}, and let
  $T_{\max}(u_0)\in(0,\infty]$ be the maximal existence time given by
  Proposition~\ref{P:local-existence}. Suppose that for some {$p >
  \max\{N,mN,\gamma N\}$,}
  \begin{equation}\label{E:Lp->GlobalExist}
    \limsup_{t\to T_{\max}(u_0)-}\int_\Omega u^{p}(t,x;u_0)\,dx < \infty
    \quad \text{or equivalently} \quad
    \sup_{t\in [0,T_{\max}(u_0))} \Norm{u(t,\cdot)}_{p} < \infty.
  \end{equation}
  Then we have
  \begin{enumerate}

    \item there is some constant $M > 0$ depending on $p$, $\mu, \nu, \gamma$,
      and $\sup_{t \in [0, T_{\max}(u_0))} \Norm{u(t,\cdot)}_p$ such that
      \begin{equation}\label{E:Lp->Global-v}
        \sup_{t\in[0,T_{\max}(u_0))} \Norm{v(t,\cdot)}_\infty + \sup_{t\in[0,T_{\max}(u_0))} \Norm{\nabla v(t,\cdot)}_\infty \le M.
      \end{equation}

    \item $u$ is uniformly bounded in $L^\infty(\Omega)$ on
      $[0,T_{\max}(u_0))$, i.e.
      \begin{equation}\label{E:Lp->Global-u}
        \limsup_{t\to T_{\max}(u_0)-} \Norm{u(t,\cdot;u_0)}_\infty<\infty
        \quad \text{or equivalently} \quad
        \sup_{t\in[0,T_{\max}(u_0))} \Norm{u(t,\cdot)}_\infty < \infty.
      \end{equation}

  \end{enumerate}
\end{proposition}

\begin{proof}
  We first note that since $[0,T_{\max}(u_0)) \ni t \mapsto u(t,\cdot;u_0) \in
  C(\overline{\Omega})$ is continuous, the conditions
  in~\eqref{E:Lp->GlobalExist} are equivalent. The same is true for the
  equivalence of the two conditions in~\eqref{E:Lp->Global-u}. \medskip

  \noindent(1) Choose $\hat{p}>N$ such that $\gamma \hat{p} < p$. Then,
  by~\eqref{E:Lp->GlobalExist},
  \begin{align}\label{E_:ugammap}
    \sup_{t\in [0, T_{\max}(u_0))}\int_\Omega u^{\gamma \hat  p}(t,\cdot)<\infty.
  \end{align}
  By choosing $1/2 < \sigma < 1$ such that $2\sigma-N/\hat{p} > 1$, and by the
  arguments of~\eqref{E:bound-v+grad-v}, there is $M > 0$ such that
  \begin{equation*}
    \Norm{v(t,\cdot)}_\infty\le M,\quad \text{and} \quad
    \Norm{\nabla v(t,\cdot)}_\infty\le M, \qquad \forall\, t \in [0, T_{\max}(u_0)).
  \end{equation*}
  This proves~\eqref{E:Lp->Global-v}. \medskip

  \noindent(2) Fix an arbitrary $t\in [0,T_{\max})$. Recall the mild form of $u$
  in~\eqref{E:u-mild}. By the concavity of the function $f: w \mapsto w(\mu + a
  - b w^\alpha)$ for $w \ge 0$, the maximum of $f$ is attained by solving $f'(w)
  = 0$. The maximizer is given by $w_* = \left( \frac{\mu + a}{b(1 + \alpha)}
  \right)^{1/\alpha}$ with the corresponding maximum value
  \begin{align*}
    f(w_*) = \left( \frac{\mu + a}{b(1 + \alpha)} \right)^{1/\alpha} \times \frac{(\mu + a)\alpha}{1 + \alpha} \eqqcolon \tilde{M}.
  \end{align*}
  Hence, from~\eqref{E:u-mild} and the above inequality, we have
  \begin{align*}
    0 \le u(t,\cdot)\le {e^{-A t}u_0}-{\chi_0\int_0 ^t e^{-A(t-s)}\nabla \cdot\left(\frac{u^m(s,\cdot)}{(1+v(s,\cdot))^\beta}\nabla v(s,\cdot)\right)ds}+\int_0 ^t e^{-A(t-s)}\tilde{M} ds.
  \end{align*}
  It is not difficult to show that
  \begin{equation*}
    \Norm{e^{-At}u_0}_\infty \le \Norm{u_0}_\infty \quad \text{and} \quad
    \Norm{\int_0^t e^{-A (t-s)} \tilde{M} ds}_\infty \le \frac{\tilde{M}}{\mu} \qquad \text{for all $t\ge 0$.}
  \end{equation*}
  Choose $0 < \tilde{\sigma} < 1/2$ and $\tilde{p} > N$ such that $\tilde{p} >
  N/(2\tilde{\sigma})$ and $m \tilde{p} < p$. By~\eqref{E:Embedding-2},
  \eqref{E:G-1-eq1}, and~\eqref{E:Lp->Global-v}, we have
  \begin{align*}
     \MoveEqLeft \Norm{\int_0^t e^{-A(t-s)}\nabla \cdot\left(\frac{u^m(s,\cdot)}{(1+v(s,\cdot))^\beta}\nabla v(s,\cdot)\right) ds }_\infty                                    \\
     & \le C \int_0^t \left((t-s)^{-\tilde{\sigma}}+(t-s)^{-\frac{1}{2}-\tilde{\sigma}}\right) e^{-\delta(t-s)} \Norm{u^m(s,\cdot)}_{\tilde{p}}ds \nonumber \\
     & \le C \sup_{s\in [0, T_{\max}(u_0))} \Norm{u^m(s, \cdot)}_{\tilde{p}} < \infty,
  \end{align*}
  where the last inequality follows from~\eqref{E:Lp->GlobalExist}. This
  completes the proof of Proposition~\ref{P:Lp->GlobalExist}.
\end{proof}

We conclude this subsection with the following lemma about $\int_\Omega u^p$,
which will be used in the proofs of
Theorems~\ref{T:negative-sensitivity}--\ref{T:weak-cross-diffusion}. The proof
is standard and we omit it here.

\begin{lemma}\label{L:global-existence}
  Fix $p > 1$. Suppose there exist constants $C_1, C_2, C_3 \ge 0$, $C_4 > 0$,
  and $0 < \epsilon \le \alpha$ such that the following inequality holds:
  \begin{align*}
    \frac{d}{dt}\int_\Omega u^p\le C_1+ C_2\int_\Omega u^p+C_3 \int_\Omega u^{p+\alpha-\epsilon}-C_4\int_\Omega u^{p+\alpha}, \qquad \text{for all $t\in (0,T_{\max}(u_0))$.}
  \end{align*}
  Then $ \displaystyle \sup_{t\in [0,T_{\max}(u_0))} \int_\Omega u^p < \infty$.
\end{lemma}

\section{Boundedness and global existence with negative sensitivity}\label{S:negative-sensitivity}

In this section, we study the boundedness and global existence of classical
solutions of~\eqref{E:main-PE} with $\chi_0\le 0$ and prove
Theorem~\ref{T:negative-sensitivity}. Throughout this section, we assume that
$u_0$ satisfies~\eqref{E:initial-cond-PE}, and $(u, v)$ is the classical
solution of~\eqref{E:main-PE} with initial condition $u(0,x;u_0) = u_0(x)$.

We first introduce some notation. Let
\begin{equation}\label{E:super-u-and-v}
  \bar{u}(t) \coloneqq \max_{x\in\overline \Omega} u(t,x) \quad {\rm and}\quad
  \bar{v}(t) \coloneqq \max_{x\in\overline{\Omega}}v(t,x), \quad \text{for $t \in [0,T_{\max})$.}
\end{equation}
We will need the following fact
\begin{align}\label{E:MaxIneq}
  \mu \bar{v}(t) - \nu \bar{u}^\gamma(t) \leq 0,
  \quad \text{for all $t \in \left(0, T_{\max}\right)$.}
\end{align}
This follows from the observation that
\[
  v(t,\cdot) = (\mu I-\Delta)^{-1} \left(\nu u^\gamma(t,\cdot)\right)
  \le (\mu I-\Delta)^{-1} \left(\nu \bar{u}^\gamma(t)\right)
  = \frac{\nu}{\mu}\bar{u}^\gamma(t),
\]
where the inequality holds since $u^\gamma(t,\cdot) \leq \bar{u}^\gamma(t)$
pointwise. Therefore, we have $\mu v(t,x) - \nu \bar{u}^\gamma(t) \leq 0$ for
all $x \in \overline{\Omega}$. Taking the maximum with respect to $x$ on the
left-hand side yields~\eqref{E:MaxIneq}. \medskip

\subsection{A technical lemma}

In this subsection, we present a lemma to be used in the proof of
Theorem~\ref{T:negative-sensitivity}.

\begin{lemma}\label{L:nonincreasing}
  Assume that $\chi_0 \le 0$.
  \begin{itemize}
    \item[(1)] Assume that $a, b > 0$. If $\bar{u}(t_0) > u^* =
      \left(a/b\right)^{1/\alpha}$ for some $t_0 \in (0, T_{\max})$, then
      $\bar{u}(t)$ is nonincreasing on $(0,t_0]$.

    \item[(2)] Assume that $a = b = 0$. Then $\bar{u}(t)$ is non-increasing on
      $(0,T_{\max})$.

  \end{itemize}
\end{lemma}

\begin{proof}
  (1) Assume that $a, b>0$. It suffices to prove that for any given $\delta_0
  \in (0, t_0)$, $\bar{u}(t)$ is nonincreasing on $[\delta_0, t_0]$. Let us fix
  an arbitrary $\delta_0 \in (0, t_0)$. The proof proceeds in three steps.
  \medskip

  \noindent\textbf{Step 1:~} In this step, we first demonstrate that if there
  exists $t^* \in [\delta_0, t_0]$ such that $\bar{u}(t^*) \geq \bar{u}(t_0)$,
  then one can find $t < t^*$ with $t^* - t \ll 1$ such that $\bar{u}(t) \geq
  \bar{u}(t^*)$.

  Suppose that $t^* \in [\delta_0,t_0]$ and $\bar{u}(t^*) \ge \bar{u}(t_0)$. Let
  $x(t^*) \in \overline{\Omega}$ be such that $\bar{u}(t^*) = u(t^*,x(t^*))$.
  Notice that the function $f(s) \coloneqq s \left(a - b s^\alpha \right)$ is
  decreasing once $s \ge u^* = (a/b)^{1/\alpha}$ because $f'(s) = a -
  b(\alpha+1) s^\alpha \le -a \alpha < 0$ for $s \ge u^*$. Hence,
  \begin{equation}\label{E:Noninc-1}
    u(t^*,x(t^*))\big(a-b u^\alpha (t^*,x(t^*))\big)
    \le \bar{u}(t_0)\big(a-b \bar{u}^\alpha(t_0))
    < 0 .
  \end{equation}
  Notice that $v(t^*, x(t^*)) \le \bar{v}(t^*)$, $\bar{u}(t^*) = u(t^*,x(t^*))$,
  and $\chi_0 \le 0$, and hence,
  \begin{align}\label{E:Noninc-2}
    -\chi_0\frac{u^m(t^*,x(t^*))}{(1+v((t^*,x(t^*)))^\beta}\left(\mu v(t^*,x(t^*))-\nu u^\gamma (t^*,x(t^*))\right)   & \notag\\
    \le -\chi_0\frac{u^m(t^*,x(t^*))}{(1+v((t^*,x(t^*)))^\beta} \left(\mu \bar{v}(t^*)-\nu\bar{u}^\gamma (t^*)\right) & \le 0,
  \end{align}
  where the last inequality is due to~\eqref{E:MaxIneq}.
  Combining~\eqref{E:Noninc-1} and~\eqref{E:Noninc-2}, we have
  \begin{align*}
    -\chi_0 \frac{u^m(t^*,x(t^*))}{(1+v((t^*,x(t^*)))^\beta} \bigg[ \mu v\left(t^*,x(t^*)\right)-\nu u^\gamma \left(t^*,x(t^*)\right)\bigg] &\\
    +  u(t^*,x(t^*))\bigg[ a-b u^\alpha (t^*,x(t^*))\bigg] & \le \bar{u}(t_0)\left[ a-b \bar{u}^\alpha(t_0)\right ] < 0.
  \end{align*}
  Denote $B(y, \epsilon) \coloneqq \left\{ x \in \Omega \,:\, |x-y| < \epsilon
  \right\}$. Then, by the continuity of $\bar{u}(t)$ in $t$ on the closed
  interval $\left[\frac{\delta_0}{2}, t_0\right]$, for any given $0 < \epsilon
  \ll 1$, for all $t \in [t^*-\epsilon, t^*+\epsilon]$ and $x \in
  B(x(t^*),\epsilon)$, it holds that
  \begin{align}\label{E_:Step-1}
    -\chi_0\frac{u^m(t,x)}{(1+v(t,x))^\beta}\left(\mu v(t,x)-\nu u^\gamma (t,x)\right) +  u(t,x)\big(a-b u^\alpha (t,x)\big)
    \le  \frac{1}{2}\bar{u}(t_0)\big(a-b \bar{u}^\alpha(t_0))
    < 0.
  \end{align}
  Since $\chi_0 \le 0$, from the first equation in~\eqref{E:local-pf-eq7}, we
  have
  \begin{align*}
    \p_t u \le \Delta u-\chi_0 m\frac{u^{m-1}}{(1+v)^\beta}\nabla u\cdot\nabla v -\chi_0 \frac{u^m}{(1+v)^\beta} (\mu v-\nu u^\gamma)+u(a-bu^\alpha).
  \end{align*}
  Then, by~\eqref{E_:Step-1}, we see that
  \begin{equation}\label{E:global-stability-proof-PE3}
    \p_t u < \Delta u-\chi_0 m\frac{u^{m-1}}{(1+v)^\beta}\nabla u\cdot\nabla v,\quad \text{for all $(t,x) \in [t^*-\epsilon, t^*+\epsilon] \times B(x(t^*),\epsilon)$}.
  \end{equation}
  Applying the maximum principle for parabolic equations
  to~\eqref{E:global-stability-proof-PE3} (precisely, applying [Chapter 2,
  Theorem 1] in~\cite{friedman:64:partial}
  to~\eqref{E:global-stability-proof-PE3} in the case that $x(t^*) \in \Omega$
  and applying Theorem 2 in~\cite{friedman:58:remarks}
  to~\eqref{E:global-stability-proof-PE3} in the case that $x(t^*) \in \partial
  \Omega$)), we have that either\footnote{Here $t^*\le t_0<T_{\max}$ ensures
    $m\,u^{m-1}$ is bounded on the compact cylinder $[t^*-\epsilon,
    t^*+\epsilon]\times\overline\Omega$ for $0<\epsilon<T_{\max}-t^*$, so the
    drift term is bounded and the parabolic maximum principle applies; the sign
  $\chi_0 \le 0$ is used in~\eqref{E:global-stability-proof-PE3}.}
  \begin{equation}\label{E:proof-global-stable-eq1}
    \sup_{x\in B(x(t^*),\epsilon)} u(t_\epsilon^*,x)> \bar{u}(t^*)\ge \bar{u}(t_0) \quad \text{for some $t_\epsilon^* \in [t^*-\epsilon, t^*)$}
  \end{equation}
  or
  \begin{equation}\label{E:proof-global-stable-eq2}
    \sup_{(t,x) \in [t^*-\epsilon,t^*] \times B(x(t^*),\epsilon)}u(t,x) = \bar{u}(t^*).
  \end{equation}
  Therefore, we can conclude that for some $t_\epsilon^* \in [t^*-\epsilon,
  t^*)$, we have that
  \begin{align*}
    \bar{u}(t_\epsilon^*)\ge \sup_{x\in B(x(t^*),\epsilon)}
    u(t_\epsilon^*,x)\ge \bar{u}(t^*).
  \end{align*}
  \medskip

  \noindent\textbf{Step 2:~} In this step, we establish that
  \begin{align}\label{E_:Step2}
    \bar{u}(t) \ge \bar{u}(t_0) \quad \text{for all $t \in [\delta_0, t_0]$.}
  \end{align}
  Suppose, for contradiction, that there exists $\tilde t \in (\delta_0, t_0)$
  such that $\bar{u}(\tilde t) < \bar{u}(t_0)$. Let $(t', t'') \subset
  (\delta_0, t_0)$ denote the maximal open interval for which $\bar{u}(t) <
  \bar{u}(t_0)$ for all $t \in (t', t'')$. By construction, we have
  $\bar{u}(t'') = \bar{u}(t_0)$. By the result established in Step 1, for any
  $0 < \epsilon \ll 1$, there exists $t_\epsilon \in (t'' - \epsilon, t'')
  \subseteq (t', t'')$ such that $\bar{u}(t_\epsilon) \ge \bar{u}(t_0)$, which
  contradicts the definition of $(t', t'')$. Therefore, $\bar{u}(t) \ge
  \bar{u}(t_0)$ for all $t \in (\delta_0, t_0)$. By the continuity of
  $\bar{u}(t)$, the claim~\eqref{E_:Step2} follows. \medskip

  \noindent\textbf{Step 3.~} Finally, let us prove that $\bar{u}(t)$ is
  nonincreasing on $[\delta_0, t_0]$. Indeed, for any $t', t'' \in (\delta_0,
  t_0)$ with $t' < t''$, repeating the above arguments with $t_0$ being replaced
  by $t''$ and $\delta_0$ being replaced by $t'$, we have $\bar{u}(t') \ge
  \bar{u}(t'')$. It then follows that $\bar{u}(t)$ is nonincreasing on
  $[\delta_0, t_0]$. This completes the proof of Lemma~\ref{L:nonincreasing}.
  \bigskip

  (2) Assume that $a = b = 0$. We prove that $\bar{u}(t)$ is non-increasing on
  $(0,T_{\max})$. Note that, if $u_0(x)$ is a constant function, namely, $u_0(x)
  \equiv \bar{u}_0 \coloneqq \sup_{x\in \overline{\Omega}} u_0(x)$, then
  \[
    \left(u(t,x;u_0), v(t,x;u_0)\right) \equiv \left(\bar{u}_0,
    \frac{\nu}{\mu} \bar{u}_0^\gamma\right)
  \]
  and hence $\bar{u}(t)$ is non-increasing on $(0,T_{\max})$. Without loss of
  generality, we then assume that $u_0(x) \not \equiv \bar{u}_0$.

  Note that for any $t_0\in(0,T_{\max})$, we have either (i) $\bar{v}(t_0) =
  \frac{\nu}{\mu} (\bar{u}(t_0))^\gamma$ or (ii) $\bar{v}(t_0) <
  \frac{\nu}{\mu}(\bar{u}(t_0))^\gamma$. In the case of (i), we have
  \begin{align*}
    0 & =\Delta v(t_0,x)-\mu v(t_0,x)+\nu u^\gamma (t_0,x)         \\
      & \le \Delta v(t_0,x)-\mu v(t_0,x)+\nu (\bar{u}(t_0))^\gamma \\
      & =\Delta \big(v(t_0,x)-\bar{v}(t_0)\big)-\mu \big(v(t_0,x)-\bar{v}(t_0)\big).
  \end{align*}
  By Hopf's Lemma and the strong maximum principle for elliptic equations (See,
  e.g., \cite[Sec. 6.4.2]{evans:10:partial}), $v(t_0,x) \equiv \bar{v}(t_0)$ and
  then $u(t_0,x) \equiv \bar{u}(t_0)$. Note that, by~\eqref{E:constant-mass},
  \begin{align*}
    \int_\Omega u(t,x)dx=\int_\Omega u_0(x)=\bar{u}(t_0)|\Omega|
  \end{align*}
  for any $t \in (0,T_{\max})$. This implies that
  \begin{equation}\label{E:sup-eq1}
    \bar{u}(t) \ge  \bar{u}(t_0) \quad \forall\, t\in (0,t_0).
  \end{equation}

  In the case of (ii), let $x(t_0) \in \overline{\Omega}$ be such that
  $u(t_0,x(t_0)) = \bar{u}(t_0)$. Then
  \begin{align*}
    v(t_0,x(t_0))-\frac{\nu}{\mu} u^\gamma(t_0,x(t_0)) < 0.
  \end{align*}
  By the continuity of $u(t,x)$ and $v(t,x)$, there are $\epsilon > 0$ and $0 <
  \delta_0 < t_0$ such that
  \begin{align*}
    v(t,x)-\frac{\nu}{\mu} u^\gamma(t,x)<0\quad \text{for all $t\in [t_0-\delta_0, t_0]$ and $x\in \overline{\Omega} \cap B(x(t_0),\epsilon)$.}
  \end{align*}
  Recalling~\eqref{E:local-pf-eq7} with $a = b = 0$, we have
  \[
    \p_t u
    = \Delta u - m\chi_0 \frac{u^{m-1}}{(1+v)^\beta}\nabla u\cdot\nabla v
      + \chi_0\beta \frac{u^m}{(1+v)^{\beta+1}}|\nabla v|^2
      - \chi_0\frac{u^m}{(1+v)^\beta}(\mu v-\nu u^\gamma).
  \]
  On the above cylinder we have $\mu v-\nu u^\gamma<0$ and $\chi_0\le0$, so the
  last term is strictly negative, while the $|\nabla v|^2$-term is nonpositive.
  Hence, on this cylinder,
  \[
    \p_t u < \Delta u
    - m\chi_0 \frac{u^{m-1}}{(1+v)^\beta}\nabla u\cdot\nabla v,
  \]
  and, arguing as in Step~1 of the proof of part~(1) with $t^* = t_0$ and
  $x(t^*) = x(t_0)$, we obtain
  \begin{equation} \label{E:sup-eq2}
    \bar{u}(t) \ge \bar{u}(t_0) \quad \forall\, t\in [t_0-\delta_0,t_0].
  \end{equation}
  It follows from~\eqref{E:sup-eq1} and~\eqref{E:sup-eq2} that $\bar{u}(t)$ is
  non-increasing on $(0,T_{\max})$. This completes the proof of
  Lemma~\ref{L:nonincreasing} (2).
\end{proof}

\subsection{Proof of Theorem \ref{T:negative-sensitivity}}

In this subsection, we  prove Theorem~\ref{T:negative-sensitivity}.

\begin{proof}[Proof of Theorem~\ref{T:negative-sensitivity}]
  (1) Assume that $a, b > 0$. One of the following alternatives must hold: (i)
  $\bar{u}(t) \le u^*$ for all $t \in (0, T_{\max})$; (ii) $\bar{u}(t) > u^*$
  for all $t \in (0, T_{\max})$; or (iii) there exist $t_1, t_2 \in (0,
  T_{\max})$ such that $\bar{u}(t_1) > u^*$ and $\bar{u}(t_2) \le u^*$. If case
  (i) holds, then it is immediate that~\eqref{E:bounded-max-eq1} is satisfied.
  If case (ii) holds, then by Lemma~\ref{L:nonincreasing}, $\bar{u}(t)$ is
  nonincreasing on $(0,T_{\max})$. Hence, $\bar{u}(t) \le \Norm{u_0}_\infty$
  for all $t \in (0,T_{\max})$ and then~\eqref{E:bounded-max-eq1} is also
  satisfied. If case (iii) holds, then Lemma~\ref{L:nonincreasing} implies that
  one has to have $t_1 < t^* \le t_2$, where $t^* \coloneqq \sup\{t\in
  (t_1,t_2)\,|\, \bar{u}(t)>u^*\}$. Then by Lemma~\ref{L:nonincreasing} again,
  we see that
  \begin{equation*}%\label{E:new-persistence-eq2}
    \bar{u}(t) \le \|u_0\|_\infty\quad \forall\, t\in (0,t^*]\quad {\rm and}\quad
    \bar{u}(t) \le u^*\quad \forall\, t\in (t^*,T_{\max}),
  \end{equation*}
  which implies that~\eqref{E:bounded-max-eq1} also holds. This proves
  Theorem~\ref{T:negative-sensitivity} (1). \medskip

  \noindent (2) When $a = b = 0$, the desired result follows from
  Lemma~\ref{L:nonincreasing} (2). This completes the proof of
  Theorem~\ref{T:negative-sensitivity}.
\end{proof}

\begin{remark}\label{R:chi-negative}
  As mentioned in Remark~\ref{R:global-existence1}, one may try to
  use~\eqref{E:u-lp-eq1} or~\eqref{E:u-lp-eq3} to get the boundedness of
  $\int_\Omega u^p(t,x) dx$ for $p \gg 1$, and then get the
  $L^\infty$-boundedness for $u$. However, it is difficult to prove the
  boundedness of $\int_\Omega u^p(t,x) dx$ for $p \gg 1$ for any $\chi_0 < 0$
  when $\gamma > 1$.

  To see this, in the case that $\chi_0 < 0$, \eqref{E:u-lp-eq3} together with
  Young's inequality implies
  \begin{align}\label{E:u-lp-alt-eq3}
      \frac{1}{p}\frac{d}{dt}\int_\Omega u^p
    & \le \chi_0\nu\frac{p-1}{p+m-1}\left(1-\frac{p+m-1}{p+m-1+\gamma}\right) \int_\Omega \frac{u^{p+m-1+\gamma}}{(1+v)^\beta}\nonumber \\
    & \quad +{C}\int_\Omega \frac{ v^{\frac{p+m-1+\gamma }{\gamma}} }{(1+v)^\beta} +a\int_\Omega u^p-b\int_\Omega u^{p+\alpha}  \quad \forall\, t\in (0,T)
  \end{align}
  for some positive constant $C$. An application of Proposition~\ref{P:Main3}
  (with $p \gg 1$) above shows that for any $\epsilon > 0$, there is $C_\epsilon
  > 0$ such that
  \begin{align}\label{E:u-lp-alt-eq4}
    \int_\Omega \frac{ v^{\frac{p+m-1}{\gamma}+1}}{(1+v)^\beta}
    \le  \epsilon \int_\Omega \frac{ u^{p+m-1+\gamma}}{(1+v)^\beta}
    + C_\epsilon \left(\int_\Omega \frac{v}{(1+v)^{\frac{\beta \gamma}{p+m-1+\gamma}}}\right)^{\frac{p+m-1+\gamma}{\gamma}}.
  \end{align}
  By~\eqref{E:u-lp-alt-eq3} and~\eqref{E:u-lp-alt-eq4}, choosing $0 < \epsilon <
  \chi_0\nu \frac{p-1}{p+m-1} \left(1-\frac{p+m-1}{p+m-1+\gamma}\right)$, we
  have
  \begin{align}\label{E:u-lp-alt-eq5}
    \frac{1}{p}\frac{d}{dt}\int_\Omega u^p
    \le C_\epsilon' \Big(\int_\Omega v\Big)^{\frac{p+m-1+\gamma}{\gamma}}
    +a\int_\Omega u^p-b\int_\Omega u^{p+\alpha}\quad \forall\, t\in (0,T).
  \end{align}
  Therefore, it is essential to prove the boundedness of $\int_\Omega v$.
  However, by the second equation in~\eqref{E:main-PE}, it holds that
  \begin{align*}
    \mu\int_\Omega v = \nu\int_\Omega u^\gamma.
  \end{align*}
  If $\gamma > 1$, it is not easy to see whether $\int_\Omega v$ is bounded and
  hence it is not easy to get the boundedness of $\int_\Omega u^p$ directly. In
  the above, we provide a direct proof of the boundedness of
  $\Norm{u(t,\cdot;u_0)}_\infty$.
\end{remark}

\section{Boundedness and global existence with weak nonlinear cross diffusion}
\label{S:weak-cross-diffusion}

In this section, we study the boundedness and global existence of classical
solutions of~\eqref{E:main-PE} with weak nonlinear cross diffusion in the sense
that $0 < m \le 1$ and $\beta \ge 1$, and prove
Theorem~\ref{T:weak-cross-diffusion}.

\subsection{A technical lemma}

In this subsection, we present a lemma to be used in the proof of
Theorem~\ref{T:weak-cross-diffusion}.

\begin{lemma}\label{L:New-lm}
  Suppose that $u_0$ satisfies~\eqref{E:initial-cond-PE}, and $(u, v)$ is the
  classical solution of~\eqref{E:main-PE} with initial condition $u(0,x;u_0) =
  u_0(x)$.  Then, for any $\varepsilon>0$ and $p>1$, there is $C_{\varepsilon,p}>0$
    such that
    \[
      \int_\Omega u^p
      \le \varepsilon \int_\Omega u^{p-2}|\nabla u|^2
      +C_{\varepsilon, p} \left(\int_\Omega u\right)^{p} \quad \text{for all $0<t<T_{\max}$.}
    \]
\end{lemma}

\begin{proof}
  Set $w = u^{p/2}$. Then~\eqref{E:Norm-Gradient-w} holds. First, by Ehrling's
  lemma (see the reference before~\eqref{E:new-v-estimate-eq3})
  and~\eqref{E:Norm-Gradient-w}, for any $0 < \varepsilon < 1$, there is
  $C_{\varepsilon,p}' > 0$ such that
  \[
    \int_\Omega u^p
    = \Norm{w}_2^2
    \le \frac{\varepsilon}{1+\varepsilon} \frac{4}{p^2} \Norm{\nabla w}_2^2+C_{\varepsilon,p}' \Norm{w}_1^2
    = \frac{\varepsilon}{1+\varepsilon} \int_\Omega u^{p-2}|\nabla u|^2 + C_{\varepsilon,p}' \left(\int_\Omega u^{p/2}\right)^2.
  \]
  Next, by the $L^p$-$L^q$ interpolation theorem (see, e.g., Theorem~6.10
  of~\cite{folland:99:real}),
  \[
    \Norm{u}_{p/2}\le \Norm{u}_1^\lambda \Norm{u}_p^{1-\lambda},  \quad \text{with $\lambda = \frac{1}{p-1}$.}
  \]
  It then follows that
  \[
    \left(\int_\Omega u^{p/2}\right)^2
    = \Norm{u}_{p/2}^p
    \le \Norm{u}_1^{\lambda p} \Norm{u}_p^{(1-\lambda)p}
    = \left(\int_\Omega u\right)^{\frac{p}{p-1}} \left(\int_\Omega u^p\right)^{\frac{p-2}{p-1}}.
  \]
  Now, by Young's inequality, we see that
  \[
    C'_{\varepsilon,p} \left(\int_\Omega u\right)^{\frac{p}{p-1}}
    \left(\int_\Omega u^p\right)^{\frac{p-2}{p-1}}
    \le \frac{\varepsilon}{1+\varepsilon} \int_\Omega u^{p} + \underbrace{\frac{(p-2)^{p-2}(1+\varepsilon)^{p-2}C_{\varepsilon,p}'^{p-1}}{\varepsilon^{p-2}(p-1)^{p-1}}}_{ \displaystyle \eqqcolon C_{\varepsilon,p}''} \left(\int_\Omega u\right)^p,
  \]
  Using this last estimate in the inequality above we get
  \begin{align*}
    \int_\Omega u^p
    \le \frac{\varepsilon}{1+\varepsilon}\int_\Omega u^{p-2}|\nabla u|^2
    + \frac{\varepsilon}{1+\varepsilon}\int_\Omega u^p + C_{\varepsilon,p}''\left(\int_\Omega u\right)^p.
  \end{align*}
  It then follows that
  \[
    \int_\Omega u^p
    \le \varepsilon\int_\Omega u^{p-2}|\nabla u|^2
        + C_{\varepsilon, p} \left(\int_\Omega u\right)^p, \quad
        \text{with $C_{\varepsilon, p} \coloneqq C''_{\varepsilon,p} (1+\varepsilon)$.}
  \]
  This proves Lemma~\ref{L:New-lm}.
\end{proof}

\subsection{Proof of~Theorem~\ref{T:weak-cross-diffusion}~(1)}
\label{SS:global-exist_4}

In this subsection, we prove Theorem~\ref{T:weak-cross-diffusion} (1).

\begin{proof}[Proof of Theorem~\ref{T:weak-cross-diffusion} (1)]
  For all $p > 1$, by Young's inequality with any $n > -1$, there holds
  \begin{align*}%\label{E:new-eq1}
    & (p-1)\chi_0\int_\Omega \frac{u^{p+m-2}}{(1+v)^\beta}\nabla u\!\cdot\!\nabla v
    & \le \frac{(p-1)}{n+1} \int_\Omega u^{p-2}|\nabla u|^2 +
          \frac{(p-1){(n+1)} \chi_0^2}{4} \int_\Omega u^{p+2m-2}\frac{|\nabla v|^2}{(1+v)^{2\beta}}.
  \end{align*}
  On the other hand, applying \eqref{E:u-lp-eq2} with $p$ being replaced by
  $p+m-1$, we have
  \begin{align*}
     \int_\Omega u^{p+2m-2}\frac{|\nabla v|^2}{(1+v)^{2\beta}}
     & \le \int_\Omega \frac{u^{p+2m-2}}{(1+v)^{\beta+1}}|\nabla v|^2  \tag{since $\beta \ge 1$}                                                                   \\
     & \le \frac{p+2m-2}{\beta} \int_\Omega \frac{u^{p+2m-3}}{(1+v)^{\beta }}\nabla u\cdot\nabla v + \frac{\mu}{\beta} \int_\Omega \frac{u^{p+2m-2}v}{(1+v)^\beta} \\
     & \le \frac{p+2m-2}{\beta} \int_\Omega \frac{u^{p+2m-3}}{(1+v)^{\beta }}\nabla u\cdot\nabla v + \frac{\mu\Theta_{\beta-1}}{\beta} \int_\Omega u^{p+2m-2},
  \end{align*}
  where the last inequality is due to Lemma~\ref{L:v-entropy} and the condition
  $\beta > 1$ is critical in this step. Set
  \begin{align*}
    C_0 \coloneqq \frac{(p-1) (n+1) \chi_0^2}{4}.
  \end{align*}
  Hence, by Young's inequality, we have
  \begin{align}\label{E:ineqC1}
    \begin{aligned}
    C_0  \int_\Omega u^{p+2m-2}\frac{|\nabla v|^2}{(1+v)^{2\beta}}
    \le & \:\frac{p-1}{n+1}\int_\Omega u^{p-2}|\nabla u|^2
          + C_1 \int_\Omega u^{p+4(m-1)} \frac{|\nabla v|^2}{(1+v)^{2\beta}} \\
        & + C_0\frac{\mu \Theta_{\beta-1}}{\beta} \int_\Omega u^{p+2m-2},
    \end{aligned}
  \end{align}
  where
  \begin{align*}
    C_1 \coloneqq \frac{(p-1)\,(n+1)^3\,\chi_0^4\,(p+2m-2)^2}{64\,\beta^2}.
  \end{align*}

  Now, using the fact that $m < 1$, we obtain that
  \begin{align}\label{E:new-eq2}
    C_0 \int_\Omega u^{p+2m-2}\frac{|\nabla v|^2}{(1+v)^{2\beta}}
    \le \frac{p-1}{n+1}\int_\Omega u^{p-2}|\nabla u|^2 + C_1 \int_\Omega u^{p+4(m-1)}\frac{|\nabla v|^2}{(1+v)^{2\beta}} +\int_\Omega u^p + C_1',
  \end{align}
  where
  \begin{align*}
    C_1' \coloneqq \frac{2(1-m)}{p}
      \left(\frac{(p-1)\,(n+1)\,\chi_0^2\,\mu\,\Theta_{\beta-1}}{4\beta}\right)^{\!\frac{p}{2(1-m)}} |\Omega|.
  \end{align*}
  For general $k = 1, \cdots, n-1$, we see that
  \begin{align*}
    C_{k-1}\int_{\Omega} u^{p+2^k(m-1)}\frac{|\nabla v|^2}{(1+v)^{2\beta}}
    & \le \frac{p-1}{n+1}\int_{\Omega} u^{p-2}|\nabla u|^2
    +C_k \int_{\Omega} u^{p+2^{k+1}(m-1)}\frac{|\nabla v|^2}{(1+v)^{2\beta}}+ \int_{\Omega} u^{p} +C_k'
  \end{align*}
  with
  \begin{align*}
    C_k  \coloneqq \frac{n+1}{4(p-1)}\left(\frac{p + 2^k(m-1)}{\beta}\right)^{\!2} C_{k-1}^2 \quad \text{and} \quad
    C_k' \coloneqq \frac{2^k(1-m)}{p} \left(\frac{C_{k-1}\mu \Theta_{\beta-1}}{\beta}\right)^{\!\frac{p}{2^k(1-m)}} |\Omega|.
  \end{align*}
  Next, repeating the above arguments we see that
  \begin{align*}
    (p-1)\chi_0
    & \int_\Omega \frac{u^{p+m-2}}{(1+v)^\beta}\nabla u\!\cdot\!\nabla v                                                                                        \\
    & \le \frac{(p-1)}{{n+1}} \int_\Omega u^{p-2}|\nabla u|^2 + C_0 \int_\Omega u^{p+2(m-1)}\frac{|\nabla v|^2}{(1+v)^{2\beta}}                                 \\
    & \le \frac{2(p-1)}{{n+1}} \int_\Omega u^{p-2}|\nabla u|^2 + C_1 \int_\Omega u^{p+2^2(m-1)}\frac{|\nabla v|^2}{(1+v)^{2\beta}} + \int_{\Omega} u^{p} + C_1' \\
    & \vdots                                                                                                                                                    \\
    & \le \frac{n(p-1)}{n+1} \int_\Omega u^{p-2}|\nabla u|^2 + C_{n-1} \int_\Omega u^{p+2^n(m-1)}\frac{|\nabla v|^2}{(1+v)^{2\beta}} + (n-1) \int_{\Omega} u^{p} + \sum_{k=1}^{n-1} C_k'.
  \end{align*}
  Set $C' = \sum_{k=1}^{n-1} C_k'$. Choose $n \ge 1$ and $p > 1$ such that
  \begin{equation}\label{E:n-choice}
    p \coloneqq 2^n(1-m)> \max\{\gamma N,\, N\}.
  \end{equation}
  It then follows from~\eqref{E:u-lp-eq1} and the above inequality that
  \begin{align}\label{E:new-eq3}
    \frac{1}{p}\frac{d}{dt}\int_\Omega u^p
    \le -\frac{p-1}{n+1}\int_\Omega u^{p-2}|\nabla u|^2+ C_{n-1} \int_{\Omega} \frac{|\nabla v|^2}{(1+v)^{2\beta}}
    +(a+n-1)\int_\Omega u^p + C',
  \end{align}
  where we have removed the integral $-b\int_\Omega u^{p+\alpha} \le 0$ since
  we allow $b = 0$. Since $p > \gamma N$, we may select any $q \in (1, N)$ such
  that $\gamma q < p$. By Young's inequality,
  \[
    \int_{\Omega} \frac{|\nabla v|^2}{(1+v)^{2\beta}}
    \le \int_{\Omega} \frac{|\nabla v|^{2q}}{(1+v)^{2\beta q}} + C'',
    \quad \text{with $C'' \coloneqq (q-1) q^{-q/(q-1)}|\Omega|$.}
  \]
  Since $\beta \ge 1$, Proposition~\ref{P:Main1} yields
  \[
    \int_{\Omega} \frac{|\nabla v|^{2q}}{(1+v)^{2\beta q}}
    \le \int_{\Omega} \frac{|\nabla v|^{2q}}{(1+v)^{(1+\beta) q}}
    \le \Theta_\beta^{q}\, M^*(N,q,\mu,\nu) \int_\Omega u^{\gamma q}.
  \]
  By Lemma~\ref{L:New-lm}(1) and Proposition~\ref{P:Main-0}, there is $C'''>0$
  such that
  \[
    (a+n)\int_\Omega u^p \le \frac{p-1}{n+1}\int_\Omega u^{p-2}|\nabla u|^2 + C'''.
  \]
  Consequently, plugging the above estimates into~\eqref{E:new-eq3} we obtain
  \begin{align}\label{E:new-eq4}
    \frac{1}{p}\frac{d}{dt}\int_\Omega u^p
    \le C_{n-1} \Theta_\beta^q \, M^* \int_\Omega u^{\gamma q} -\int_\Omega u^p
     +  \left(C'''+ C_{n-1} C'' + C'\right),
  \end{align}
  with $\gamma q < p$. Therefore, by Gronwall's inequality, we see that
  $\int_\Omega u^p(t,x;u_0)$ stays bounded on $[0, T_{\max}(u_0))$. Finally, by
  Proposition~\ref{P:Lp->GlobalExist} and our choice of $p$
  in~\eqref{E:n-choice}, namely, $p > (1 \vee \gamma) N$, we can conclude
  that~\eqref{E:bounded-finite} holds. This proves
  Theorem~\ref{T:weak-cross-diffusion}(1).
\end{proof}

\subsection{Proof of Theorem~\ref{T:weak-cross-diffusion} (2)}\label{SS:global-exist_7}

In this subsection, we prove Theorem~\ref{T:weak-cross-diffusion}(2).

\begin{proof}[Proof of Theorem~\ref{T:weak-cross-diffusion} (2)] We assume $m =
  1$, $\beta \ge 1$, and $\chi_0 < \frac{2(2\beta-1)}{\max\{2, \gamma N\}}$.
  Let $(u, v)$ be a positive classical solution on $(0, T_{\max}(u_0))$ and let
  $p > 1$. We have two observations: (1) Since our result holds for all $b \ge
  0$, we cannot rely on the logistic term to control the $L^p$-norm of $u$; (2)
  The case $\chi_0 \le 0$ has been covered by
  Theorem~\ref{T:negative-sensitivity}. Therefore, without loss of generality,
  we assume that
  \begin{align*}
    b = 0 \quad \text{and} \quad \chi_0 > 0.
  \end{align*}
  Applying~\eqref{E:u-lp-eq1} with $m = 1$ and omitting the integral $b\int
  u^{p+\alpha}$, we see that
  \begin{equation}\label{E_:lp-id}
    \frac{1}{p}\frac{d}{dt}\int_\Omega u^p
    \le -(p-1)\int_\Omega u^{p-2}|\nabla u|^2
      + (p-1)\chi_0\int_\Omega \frac{u^{p-1}}{(1+v)^\beta}\nabla u\cdot\nabla v
      + a\int_\Omega u^p.
  \end{equation}
  We proceed with the proof in two steps. As the argument develops, we will
  determine the conditions on $\chi_0$ and $\beta$. \medskip

  \noindent\textbf{Step 1.~} In this step, we prove that if $\chi_0 \le 2 \beta
  -1$ and $\beta \ge 1$, then 
  \begin{equation}\label{E:claim-eq1}
    \frac{1}{p}\frac{d}{dt}\int_\Omega u^p
    \le - \int_\Omega u^p + C \left(\int_\Omega u\right)^p,
    \quad \text{for all $1 < p < (2\beta-1)/\chi_0$.}
  \end{equation}
  Then Gronwall's inequality immediately implies that
  \begin{equation}
    \label{E:new-cross-eq0}
    \sup_{t\in[0,T_{\max}(u_0))}\int_\Omega u^p(t,\cdot)\;<\;\infty
    \quad \text{for all $1 < p < (2\beta-1)/\chi_0$.}
  \end{equation}

  To prove~\eqref{E:claim-eq1}, by Young's inequality, the cross term
  in~\eqref{E_:lp-id} can be estimated as follows:
  \begin{align}\label{E_:young-cross}
    (p-1)\chi_0\int_\Omega \frac{u^{p-1}}{(1+v)^\beta}\,\nabla u\!\cdot\!\nabla v
    & \le \frac{p-1}{2} \int_\Omega u^{p-2}|\nabla u|^2
    + \frac{(p-1)\chi_0^2}{2} \int_\Omega \frac{u^{p}|\nabla v|^2}{(1+v)^{2\beta}}.
  \end{align}
  On the other hand, by removing the last term and then rearranging terms
  in~\eqref{E:u-lp-eq2} with $\beta$ being replaced by $2\beta-1$, together with
  setting $m = 1$, we have
  \begin{align*}
    & \frac{(p-1)\chi_0(2\beta-1)}{p} \int_\Omega \frac{u^{p}}{(1+v)^{2\beta}}|\nabla v|^2                                                                   \\
    & \le (p-1)\chi_0\int_\Omega \frac{u^{p-1}}{(1+v)^{2\beta-1}}\nabla u\cdot\nabla v +\frac{(p-1)\chi_0\mu}{p}\int_\Omega \frac{u^{p}}{(1+v)^{2\beta-1}} v \\
    & \le \frac{p-1}{2}\int_\Omega u^{p-2}|\nabla u|^2
    +\frac{(p-1)\chi_0^2}{2}\int_\Omega \frac{u^p}{(1+v)^{4\beta-2}}|\nabla v|^2
    + \frac{(p-1) \Theta_{2(\beta-1)} \chi_0\mu }{p} \int_\Omega u^p.
  \end{align*}
  In the last equality, we have applied Lemma~\ref{L:v-entropy} for the second
  integral and~\eqref{E_:young-cross} with $\beta$ replaced by $2\beta - 1$ for
  the first. Using the facts that $(1+v)^{-(4\beta-2)} \le (1+v)^{-2\beta}$ and
  $\Theta_{2(\beta-1)} \le 1$(because $\beta \ge 1$), we see that
  \begin{align*}%\label{E:ineq-m=1}
    \left(\frac{(p-1)\chi_0{ (2\beta-1)}}{p}-\frac{(p-1)\chi_0^2}{2}\right) \int_\Omega
    \frac{u^{p}}{(1+v)^{2\beta}}|\nabla v|^2
    & \le \frac{p-1}{2} \int_\Omega u^{p-2}|\nabla u|^2  +\frac{(p-1)\chi_0\mu }{p} \int_\Omega u^{p}.
  \end{align*}
  Notice that
  \[
    \frac{(p-1)\chi_0^2}{2}
    < \frac{(p-1)\chi_0(2\beta-1)}{p}-\frac{(p-1)\chi_0^2}{2}
    \quad \Longleftrightarrow \quad
    \chi_0 < \frac{2 \beta-1}{p}.
  \]
  Since $p< (2\beta-1)/\chi_0$, we have $\chi_0 < (2\beta-1) / p$. Hence, there
  is $0<\varepsilon<1$
  such that
  \[
    \frac{(p-1)\chi_0^2}{2}
    \le \left(1-\frac{2\varepsilon}{p-1}\right) \left(  \frac{(p-1)\chi_0(2\beta-1)}{p}-\frac{(p-1)\chi_0^2}{2}\right).
  \]
  Hence,
  \begin{equation}\label{E_:absorb-ineq}
    \frac{(p-1)\chi_0^2}{2}\int_\Omega \frac{u^{p}|\nabla v|^2}{(1+v)^{2\beta}}
    \;\le\; \left(1-\frac{2\varepsilon}{p-1}\right) \frac{p-1}{2}\int_\Omega u^{p-2}|\nabla u|^2
    + \left(1-\frac{2\varepsilon}{p-1}\right)\frac{(p-1)\chi_0\mu }{p}\int_\Omega u^p.
  \end{equation}
  Insert~\eqref{E_:young-cross} into~\eqref{E_:lp-id} and then
  apply~\eqref{E_:absorb-ineq} to see that 
  \[
    \frac{1}{p}\frac{d}{dt}\int_\Omega u^p
    \le -\varepsilon \int_\Omega u^{p-2}|\nabla u|^2+ \left(1+a + \frac{(p-1)\chi_0\mu }{p}\right)\int_\Omega u^p-\int_\Omega u^p.
  \]
  By Lemma \ref{L:New-lm}, there is $C_{\varepsilon,p} > 0$ such that
  \[
    \frac{1}{p}\frac{d}{dt}\int_\Omega u^p
    \le -\int_\Omega u^p + C_{\varepsilon,p} \left(\int_\Omega u\right)^p,
  \] 
  which proves~\eqref{E:claim-eq1}. \bigskip

  \noindent\textbf{Step 2.~} In this step, we apply the bootstrap argument given
  in Lemma~\ref{L:bootstrap} and Corollary~\ref{C:bootstrap}. Since $\chi_0 <
  \frac{2(2\beta-1)}{\max\{2, \gamma N\}}$, we have
  \[
      \frac{2\beta-1}{\chi_0}
    > \frac{\max\{2,\gamma N\}}{2}
    = \max\left\{1,\frac{\gamma N}{2}\right\}.
  \]
  Choose $p_0$ such that
  \[
    \max\left\{1,\frac{\gamma N}{2}\right\} < p_0 < \frac{2\beta-1}{\chi_0}.
  \]
  Then~\eqref{E:new-cross-eq0} implies that
  \[
    \sup_{t\in[0,T_{\max}(u_0))}\int_\Omega u^{p_0}(t,\cdot)\,dx < \infty.
  \]
  Next, fix $p>1$ and $\epsilon>0$. By Young's inequality, the cross term
  in~\eqref{E_:lp-id} can be estimated as follows:
  \begin{align}\label{E:new-cross-eq1}
    \chi_0\int_\Omega \frac{u^{p-1}}{(1+v)^\beta}\,\nabla u\!\cdot\!\nabla v
      & \le \epsilon  \int_\Omega u^{p-2}|\nabla u|^2
      + \frac{\chi_0^2}{4 \epsilon} \int_\Omega \frac{u^{p}|\nabla v|^2}{(1+v)^{2\beta}}.
  \end{align}
  By Young's inequality with conjugate exponents $\frac{p+\gamma}{p}$ and
  $\frac{p+\gamma}{\gamma}$, the last term in the above can be estimated as
  follows:
  \begin{equation}\label{E:new-cross-eq2}
    \int_\Omega \frac{u^p|\nabla v|^2}{(1+v)^{2\beta}}
    \le \frac{1}{2}\int_\Omega u^{p+\gamma}
       +\frac{\gamma}{p+\gamma}\left(\frac{2p}{p+\gamma}\right)^{p/\gamma}
         \int_\Omega \frac{|\nabla v|^{\frac{2(p+\gamma)}{\gamma}}}{(1+v)^{\frac{2\beta(p+\gamma)}{\gamma}}}.
  \end{equation}
  Let $q \coloneqq \frac{p+\gamma}{\gamma} > 1$. Applying
  Proposition~\ref{P:Main1} with $p = q$ and $\beta$ replaced by $2\beta-1$
  yields
  \[
    \int_\Omega \frac{|\nabla v|^{2q}}{(1+v)^{2\beta q}}
    = \int_\Omega \frac{|\nabla v|^{2q}}{(1+v)^{(1+(2\beta-1))q}}
    \le \Theta_{2\beta-1}^{q}\, M^*(N,q,\mu,\nu) \int_\Omega u^{\gamma q}.
  \]
  Therefore, by~\eqref{E:new-cross-eq1} and~\eqref{E:new-cross-eq2}, we have
  \begin{align}\label{E:new-cross-eq3}
    \chi_0\int_\Omega \frac{u^{p-1}}{(1+v)^\beta}\,\nabla u\!\cdot\!\nabla v
      & \le \epsilon  \int_\Omega u^{p-2}|\nabla u|^2
      + \underbrace{\frac{\chi_0^2}{4\epsilon}\left(\frac{1}{2}
          + \Theta_{2\beta-1}^{q}\, M^*(N,q,\mu,\nu)\frac{\gamma}{p+\gamma}\left(\frac{2p}{p+\gamma}\right)^{p/\gamma}\right)}_{\displaystyle \eqqcolon C_{\epsilon,p}}
        \int_\Omega u^{p+\gamma}.
  \end{align}
  The inequality~\eqref{E:new-cross-eq3} verifies~\eqref{E:C-bootstrap} in
  Corollary~\ref{C:bootstrap} with $\rho = \gamma$ (and $m = 1$). Together with
  the above $L^{p_0}$ bound, Corollary~\ref{C:bootstrap} implies that for every
  $p > 1$,
  \[
    \limsup_{t\to T_{\max}-}\int_\Omega u^p<\infty.
  \]
  In particular, by choosing $p > \max\{N,\gamma N\}$,
  Proposition~\ref{P:Lp->GlobalExist} yields~\eqref{E:bounded-finite}. Since $m
  = 1\ge 1$, it follows from~\eqref{E:local-infty} in
  Proposition~\ref{P:local-existence} that $T_{\max}(u_0) = \infty$. This
  completes the proof of Theorem~\ref{T:weak-cross-diffusion}(2).
\end{proof}

\section{Boundedness and global existence with relatively strong logistic source}
\label{S:strong-logistic-source}

In this section, we study the boundedness and global existence of classical
solutions of~\eqref{E:main-PE} with relative strong logistic source and prove
Theorem~\ref{T:strong-logistic-source}. Again, throughout this section, we
assume that $u_0$ satisfies~\eqref{E:initial-cond-PE}, and $(u, v)$ is the
classical solution of~\eqref{E:main-PE} with initial condition $u(0,x;u_0) =
u_0(x)$.

\subsection{Proof of Theorem~\ref{T:strong-logistic-source} (1) under condition~(i)}
\label{SS:global-exist_2}

In this subsection, we prove Theorem~\ref{T:strong-logistic-source} (1) under
condition~(i). We prove it by showing the boundedness of $\int_\Omega u^p(t,
x; u_0) dx$ on $[0, T_{\max}(u_0))$ for any $p > 1$. Proposition~\ref{P:Main1}
plays a key role in the proof of the boundedness of $\int_\Omega u^p(t, x; u_0)
dx$.

\begin{proof}[Proof of Theorem~\ref{T:strong-logistic-source} (1) under condition
  (i)] First, by Proposition~\ref{P:Lp->GlobalExist}, it suffices to prove the
  $L^p$-boundedness of $u(t,x;u_0)$ for some $p > \max\{N, mN, \gamma N\}$. By
  Theorem~\ref{T:negative-sensitivity}, we only need to consider $\chi_0 > 0$. In the
  following, we assume that $\chi_0 > 0$. By~\eqref{E:u-lp-eq3}, for all $p >
  1$, we have
  \begin{align*}
   \frac{1}{p}\frac{d}{dt}\int_\Omega u^p
   \le & -(p-1)\int_\Omega u^{p-2}|\nabla u|^2  + \frac{(p-1)\chi_0\beta}{p+m-1}\int_\Omega \frac{u^{p+m-1}}{(1+v)^{\beta+1}}|\nabla v|^2 \\
       & \quad +\frac{(p-1)\chi_0\nu}{p+m-1}\int_\Omega \frac{u^{p+m+\gamma -1}}{(1+v)^\beta} +a\int_\Omega u^p-b\int_\Omega u^{p+\alpha}\quad \forall\, t\in (0,T_{\max}(u_0)).
  \end{align*}
  By the assumption $\alpha > m + \gamma - 1$ and Young's inequality, for any
  $\epsilon > 0$, there exists a constant $C_\epsilon \ge 1$ such that
  \begin{gather*}
    \frac{(p-1)\chi_0\nu}{p+m-1}\int_\Omega \frac{u^{p+m+\gamma -1}}{(1+v)^\beta} \le \epsilon \int_\Omega u^{p+\alpha}+C_\epsilon \quad \text{and} \\
    \frac{(p-1)\chi_0\beta}{p+m-1}\int_\Omega \frac{u^{p+m-1}}{(1+v)^{\beta+1}}|\nabla v|^2 \le \epsilon \int_\Omega u^{p+\alpha} + C_\epsilon\int_\Omega \frac{|\nabla v|^{\frac{2(p+\alpha)}{\alpha+1-m}}}{(1+v)^{(\beta+1)\frac{p+\alpha}{\alpha+1-m}}}.
  \end{gather*}
  By Proposition~\ref{P:Main1}, there is a constant $M > 0$ such that
  \begin{align*}
    \int_\Omega \frac{|\nabla v|^{\frac{2(p+\alpha)}{\alpha+1-m}}}{(1+v)^{(\beta+1) \frac{p+\alpha}{\alpha+1-m}}}
    \le \int_\Omega \frac{|\nabla v|^{\frac{2(p+\alpha)}{\alpha+1-m}}}{(1+v)^{\frac{p+\alpha}{\alpha+1-m}}}
    \le M \int_\Omega u^{\gamma \frac{p+\alpha}{\alpha+1-m}}.
  \end{align*}
  By $\alpha > m + \gamma - 1$ again, we have
  $\gamma\frac{p+\alpha}{\alpha+1-m}<p+\alpha$. We can apply Young's inequality
  again to obtain to see that for some constant $C_\epsilon' > 0$,
  \begin{align*}
    \int_\Omega u^{\gamma \frac{p+\alpha}{\alpha+1-m}} \le \frac{\epsilon}{C_\epsilon M} \int_\Omega u^{p+\alpha} + C_\epsilon'.
  \end{align*}
  Therefore, there is a constant $C_\epsilon''>0$ such that
  \begin{align*}
    \frac{1}{p} \cdot \frac{d}{dt} \int_{\Omega} u^p\le a\int_\Omega u^p-(b-3\epsilon) \int_\Omega u^{p+\alpha} + C_\epsilon''.
  \end{align*}
   By Lemma~\ref{L:global-existence}, $\int_\Omega u^p(t,x;u_0) dx$ stays
   bounded on $[0, T_{\max}(u_0))$ for all $p > 1$, and then by
   Proposition~\ref{P:Lp->GlobalExist}, $\limsup_{t\to T_{\max}(u_0)-}
   \Norm{u(t,\cdot;u_0)}_\infty < \infty$. This proves
   Theorem~\ref{T:strong-logistic-source} (1) assuming the condition~(i) holds.
\end{proof}

\subsection{Proof of Theorem~\ref{T:strong-logistic-source} (1) under condition~(ii)}
\label{SS:global-exist_3}

In this subsection, we prove Theorem~\ref{T:strong-logistic-source} (1) under
condition~(ii). We will also prove it via establishing the boundedness of
$\int_\Omega u^p$ on $[0, T_{\max}(u_0))$ for all $p > 1$ analogously to the
proof of Theorem~\ref{T:strong-logistic-source} (1) under condition~(ii) above.

\begin{proof}[Proof of Theorem~\ref{T:strong-logistic-source} (1) under condition
  (ii)]
By Theorem~\ref{T:negative-sensitivity} (1), we
  only need to consider the case that $\chi_0 > 0$. Fix an arbitrary $p > 1$.
  The condition $\alpha > 2m + \gamma - 2$ implies that there is a constant
  $\epsilon \in (0, \alpha)$ such that
  \begin{equation}\label{E:param-1}
    \alpha+2-2m-\epsilon > \gamma \quad \text{and} \quad
    p+2m -2 < p+\alpha-\epsilon.
  \end{equation}
  Applying Young's inequality twice, we see that there is a constant $C > 0$
  such that
  \begin{align*}
    (p-1)\chi_0\int_\Omega \frac{u^{p+m-2}}{(1+v)^\beta}\nabla u\cdot\nabla v
    \le (p-1)\int_\Omega u^{p-2}|\nabla u|^2+\frac{(p-1)\chi_0^2}{4}\int_\Omega u^{p+2m -2}\frac{|\nabla v|^2}{(1+v)^{2\beta}} & \\
    \le (p-1)\int_\Omega u^{p-2}|\nabla u|^2+C\int_\Omega u^{p+\alpha-\epsilon}+C\int_\Omega \frac{|\nabla v|^{2 \frac{p+\alpha -\epsilon}{\alpha+2-2m-\epsilon}}}{(1+v)^{2\beta \frac{p+\alpha-\epsilon}{\alpha +2-2m-\epsilon}}}, &
  \end{align*}
  where the second inequality is due to the second inequality
  in~\eqref{E:param-1}. By the condition $\beta \ge 1/2$, set $\beta' =
  2\beta-1 \ge 0$. Applying Proposition~\ref{P:Main1},
  inequality~\eqref{E:Main-P1}, with exponent
  $\frac{p+\alpha-\epsilon}{\alpha+2-2m-\epsilon}$ and $\beta'$, there is a
  constant $M > 0$ such that
  \begin{align*}
    \int_\Omega \frac{|\nabla v|^{2 \frac{p+\alpha -\epsilon}{\alpha+2-2m-\epsilon}}}{(1+v)^{2\beta \frac{p+\alpha-\epsilon}{\alpha +2-2m-\epsilon}}}
    &\le \Theta_{\beta'}^{\frac{p+\alpha-\epsilon}{\alpha+2-2m-\epsilon}}\, M\int_\Omega u^{\gamma \frac{p+\alpha -\epsilon}{\alpha+2-2m-\epsilon}}
    \le M \int_\Omega u^{p+\alpha-\epsilon} + M |\Omega|,
  \end{align*}
  where the last inequality follows from the first inequality
  in~\eqref{E:param-1} and from the observation that $0 \le a \le b$ implies
  $x^a \le 1 + x^b$ for all $x \ge 0$. It then follows from~\eqref{E:u-lp-eq1}
  that
  \begin{align*}
    \frac{1}{p}\frac{d}{dt}\int_\Omega u^p\le C(1+M)\int_\Omega u^{p+\alpha-\epsilon} +a\int_\Omega u^p-b \int_\Omega u^{p+\alpha}+ C M |\Omega|.
  \end{align*}
  By Lemma~\ref{L:global-existence}, $\int_\Omega u^p(t,x;u_0)dx$ stays bounded
  on $[0, T_{\max}(u_0))$ for all $p > 1$, and then by
  Proposition~\ref{P:Lp->GlobalExist}, $\limsup_{t \to T_{\max}(u_0)-}
  \Norm{u(t, \cdot; u_0)}_\infty < \infty$.
\end{proof}

\subsection{Proof of~Theorem~\ref{T:strong-logistic-source}~(1) under condition~(iii)}
\label{SS:global-exist_5}

In this subsection, we prove Theorem~\ref{T:strong-logistic-source} (1) under the
condition~(iii). Our approach is to first establish the boundedness of
$\int_\Omega u^{p_0}$ for some $p_0 > \frac{N\alpha}{2}$, and then, using the
Gagliardo--Nirenberg inequality, to show the boundedness of $\int_\Omega u^p$
for all $p > 1$.

\begin{proof}[Proof of Theorem~\ref{T:strong-logistic-source} (1) under
  condition~(iii)]
  Again, by Theorem~\ref{T:negative-sensitivity}, we only need to consider the
  case that $\chi_0 > 0$. Fix an arbitrary $p \ge 1$. We will proceed in two
  steps. \medskip

  \noindent\textbf{Step 1:~} In this step, we prove that there is $p_0 > \max\{1,
  \frac{N\alpha}{2}\}$ such that $\int_\Omega u^{p_0}(t, x; u_0) dx$ stays
  bounded on $[0, T_{\max}(u_0))$; see~\eqref{E:proof-thm1-4-eq3} below. Recall
  that $\alpha = m + \gamma - 1$ and $\beta \ge 0$. By $\alpha = m + \gamma - 1$
  and $\gamma > 0$, we have $p + m - 1 < p + \alpha$. By
  Lemma~\ref{L:v-entropy}, we see that
  \[
    \beta\int_\Omega u^{p+m-1} \frac{|\nabla v|^2}{(1+v)^{\beta+1}}
    \le \Psi_\beta\, \int_\Omega u^{p+m-1}\frac{|\nabla v|^2}{v}.
  \]
  Let $s\coloneqq \dfrac{p+\alpha}{\alpha-m+1}=\dfrac{p+\alpha}{\gamma}>1$ and
  $r\coloneqq \dfrac{s}{s-1}=\dfrac{p+\alpha}{p+m-1}$. By the H\"older
  inequality and~\eqref{E:Main-P0},
  \begin{align*}
    \int_\Omega u^{p+m-1}\frac{|\nabla v|^2}{v}
    \le & \left(\int_\Omega u^{p+\alpha}\right)^{\!\frac{s-1}{s}} \left(\int_\Omega \frac{|\nabla v|^{2s}}{v^{s}}\right)^{\!\frac{1}{s}} \\
    \le & M^*(N,s,\mu,\nu)^{1/s}\int_\Omega u^{\gamma s}
    = M^*(N,s,\mu,\nu)^{1/s}\int_\Omega u^{p+\alpha}.
  \end{align*}
  Therefore, we have shown that
  \begin{align}\label{E:thm1-4_Step1-1}
    \beta \int_\Omega u^{p+m-1}\,\frac{|\nabla v|^2}{(1+v)^{\beta+1}}
    \le \Psi_\beta\, M^*(N,s,\mu,\nu)^{1/s}\int_\Omega u^{p+\alpha}.
  \end{align}
  Apply~\eqref{E:u-lp-eq2} with $\alpha = m + \gamma - 1$ to see that
  \begin{align*}
    \int_{\Omega} \frac{u^{p+m -2}}{(1+v)^\beta}\nabla u \cdot \nabla v
    & = -\frac{\mu }{p+m-1}\int_\Omega \frac{u^{p+m-1 } v}{(1+v)^\beta} + \frac{\nu}{p+m-1}\int_\Omega \frac{u^{p+\alpha}}{(1+v)^\beta} \\
    & \quad + \frac{\beta}{p+m-1}\int_\Omega u^{p+m-1}\frac{|\nabla v|^2}{(1+v)^{\beta+1}}                                              \\
    & \le \frac{\nu}{p+m-1}\int_\Omega u^{p+\alpha} + \frac{\beta}{p+m-1}\int_\Omega u^{p+m-1}\frac{|\nabla v|^2}{(1+v)^{\beta+1}}      \\
    & \le \frac{1}{p+m-1}\left(\nu + \Psi_\beta\, M^*(N,s,\mu,\nu)^{1/s} \right)\int_\Omega u^{p+\alpha},
  \end{align*}
  where the last inequality is due to~\eqref{E:thm1-4_Step1-1}. Now for all
  $\beta > 0$ and $q \ge 1$, set
  \begin{equation}\label{E:F}
    F(q)
    \coloneqq
    \chi_0\, \frac{(q-1)}{(q-1) + m}
    \left(\nu + \Psi_\beta\, M^*\!\left(N,\,\frac{q+\alpha}{\alpha-m+1},\,\mu,\,\nu\right)^{\!\frac{\alpha-m+1}{\,q+\alpha\,}}\right).
  \end{equation}
  It then follows from~\eqref{E:u-lp-eq1} that
  \begin{align}\label{E:proof-thm1-4-eq1}
    \frac{1}{p} \cdot \frac{d}{dt} \int_{\Omega} u^p
    & \le -(p-1)\int_{\Omega} u^{p-2} |\nabla u|^2 + a \int_{\Omega} u^p - \left(b- F(p) \right) \int_{\Omega} u^{p+\alpha}.
  \end{align}
  This, together Lemma~\ref{L:global-existence}, implies that
  \begin{equation}\label{E:proof-thm1-4-eq2}
    \limsup_{t\to T_{\max}(u_0)}\int_\Omega u^p(t,\cdot;u_0)<\infty
    \quad \text{provided that}\,\,  b > F(p),
  \end{equation}
  where we note that $\tfrac{p+\alpha}{\alpha-m+1}>1$ for all $p \ge 1$ and $m >
  0$.

  Let $K$ be as in \eqref{E:K}. Note that condition~\eqref{E:cond-chi-eq1} is
  equivalent to
  \[
    b> \chi_0 \frac{\left(\frac{N\alpha}{2}-1\right)_+}{\left(\frac{N\alpha}{2}-1\right)_+ + m} \left(\nu + K \Psi_\beta \right).
  \]
  Assume~\eqref{E:cond-chi-eq1} holds. Then, by the definition of $K$
  in~\eqref{E:K}, there is $p_0>q_* = \max\left\{1, \frac{N\alpha}{2}\right\}$
  such that $b > F(p_0)$. Then by~\eqref{E:proof-thm1-4-eq2} applied to $p =
  p_0$, we have
  \begin{equation}\label{E:proof-thm1-4-eq3}
    \limsup_{t\to T_{\max}(u_0)}\int_\Omega u^{p_0}(t,\cdot;u_0) < \infty.
  \end{equation}
  \medskip

	  \noindent\textbf{Step 2:~} In this step, we prove
	  that~\eqref{E:bounded-finite} holds.

	  Let $p>1$ and $\epsilon>0$ be arbitrary.
	  By the arguments in step 1, there is $C_p>0$ such that
	  \[
	    \chi_0\int_\Omega \frac{u^{p+m-2}}{(1+v)^\beta}\nabla u\cdot\nabla v\le C_p\int_\Omega u^{p+\alpha}\quad \forall\, t\in (0,T_{\max}).
	  \]
	  Hence,~\eqref{E:C-bootstrap} holds with $\rho=\alpha$ and $C_{\epsilon,p}=C_p$.
	  It then follows from Corollary~\ref{C:bootstrap} that, for any $p>1$,
	  \[
	    \limsup_{t\to T_{\max}-}\int_\Omega u^p<\infty.
	  \]
  Then by Proposition \ref{P:Lp->GlobalExist}, \eqref{E:bounded-finite} holds.
\end{proof}

\subsection{Proof of Theorem~\ref{T:strong-logistic-source} (1) under
condition~(iv)}\label{SS:global-exist_6}

In this subsection, we prove Theorem~\ref{T:strong-logistic-source} (1) in the borderline
case when $\alpha = 2m + \gamma - 2$, and $\beta \ge 1/2$.

\begin{proof}[Proof of Theorem~\ref{T:weak-cross-diffusion} (1) under
  condition~(iv)] We divide the proof into two steps. \smallskip

  \noindent\textbf{Step 1.} In this step, we prove that there is
  $p_0>\max\{1,\frac{\alpha N}{2}\}$ such that $ \displaystyle \limsup_{t\to
  T_{\max}-}\int_\Omega u^{p_0}<\infty$. \smallskip

  Starting from~\eqref{E:u-lp-eq1}, apply Young's inequality to the cross term
  to get
  \[
    (p-1)\chi_0\int_\Omega \frac{u^{p+m-2}}{(1+v)^\beta}\nabla u\!\cdot\!\nabla v
    \le (p-1) \int_\Omega u^{p-2}|\nabla u|^2 +
        \frac{(p-1)\chi_0^2}{4} \int_\Omega u^{p+2m-2}\frac{|\nabla v|^2}{(1+v)^{2\beta}}.
  \]
  For the last integral, since $\alpha+2-2m = \gamma$, we can apply the H\"older
  inequality with
  \[
    r = \frac{p+\alpha}{p+2m-2} > 1, \qquad
    s = \frac{p+\alpha}{\alpha-2m+2} = \frac{p+\alpha}{\gamma}, \qquad
    \frac{1}{r}+\frac{1}{s} = 1
  \]
  to obtain
  \begin{align*}
    \int_\Omega u^{p+2m-2}\frac{|\nabla v|^2}{(1+v)^{2\beta}}
    \le & \left(\int_\Omega u^{p + \alpha}\right)^{1/r} \left(\int_\Omega \frac{|\nabla v|^{2s}}{(1+v)^{2\beta s}}\right)^{1/s}                          \\
    \le & \left(\int_\Omega u^{p + \alpha}\right)^{1/r} \Theta_{\beta'}\,\left(M^*(N,s,\mu,\nu)\right)^{1/s} \left(\int_\Omega u^{\gamma s}\right)^{1/s} \\
     =  & \Theta_{\beta'}\,\left(M^*(N,s,\mu,\nu)\right)^{1/s} \int_\Omega u^{p+\alpha},
  \end{align*}
  where we have applied Proposition~\ref{P:Main1} with $\beta' = 2\beta-1 \ge 0$
  (since $\beta \ge 1/2$) and use the fact that $\gamma s = p+\alpha$. Using
  this in~\eqref{E:u-lp-eq1} yields
  \[
    \frac{1}{p}\frac{d}{dt}\int_\Omega u^p
    \le a \int_\Omega u^p - \left[\,b - \frac{(p-1)\chi_0^2}{4}\Theta_{\beta'}\,\left(M^*(N,s,\mu,\nu)\right)^{1/s}\right]\int_\Omega u^{p+\alpha}.
  \]
  Here $\beta' = 2\beta-1$ and $s=\dfrac{p+\alpha}{\gamma}$. Define, for $q\ge 1$,
  \begin{equation}\label{E:F_vi}
   \tilde  F(q)
    \coloneqq \frac{(q-1)_+\,\chi_0^2}{4}\,\Theta_{2\beta-1}\,\left[\,M^*\!\left(N,\, \frac{q+\alpha}{\gamma},\, \mu,\, \nu\right)\,\right]^{\!\frac{\gamma}{q+\alpha}}.
  \end{equation}
  Then
  \begin{align}\label{E:proof-thm1-6-eq1}
    \frac{1}{p} \; \frac{d}{dt} \int_{\Omega} u^p
    \le a \int_{\Omega} u^p - \big(b- \tilde{F}(p)\big) \int_{\Omega} u^{p+\alpha}.
  \end{align}
  By Lemma~\ref{L:global-existence}, this implies
  \begin{equation}\label{E:proof-thm1-6-eq2}
    \limsup_{t\to T_{\max}(u_0)} \int_\Omega u^p(t,\cdot;u_0)<\infty
    \quad \text{provided that}\,\, b > \tilde{F}(p),
  \end{equation}
  where we note that $\tfrac{p+\alpha}{\gamma}>1$ for all $p \ge 1$.

  Let $K$ be as in~\eqref{E:K}. Under the condition~\eqref{E:cond-chi-eq2},
  there exists $p_0>q_* = \max\left\{1, \frac{N\alpha}{2}\right\}$ such that $b
  > \tilde{F}(p_0)$. Hence, by~\eqref{E:proof-thm1-6-eq2} with $p = p_0$,
  \begin{equation}\label{E:proof-thm1-6-eq3}
    \limsup_{t\to T_{\max}(u_0)}\int_\Omega u^{p_0}(t,\cdot;u_0) < \infty.
  \end{equation}
  \medskip

  \noindent\textbf{Step 2.} In this step, we prove that \eqref{E:bounded-finite}
  holds. \smallskip

  Let $p>1$ and $\epsilon>0$ be arbitrary. By Young's inequality and the
  arguments in step~1, there is $C_{\epsilon,p}>0$ such that
  \[
    \chi_0\int_\Omega \frac{u^{p+m-2}}{(1+v)^\beta}\nabla u\cdot\nabla v\le 
    \epsilon\int_\Omega u^{p-2}|\nabla u|^2+C_{\epsilon,p} \int_\Omega u^{p+\alpha}\quad \forall\, t\in (0,T_{\max}).
  \]
  It then follows from Corollary~\ref{C:bootstrap} that, for any $p>1$,
  \[
    \limsup_{t\to T_{\max}-}\int_\Omega u^p<\infty.
  \]
  Then by Proposition~\ref{P:Lp->GlobalExist}, \eqref{E:bounded-finite} holds.
\end{proof}

\subsection{Proof of Theorem \ref{T:strong-logistic-source} (2)}\label{SS:global-exist_part2}

In this subsection, we prove Theorem~\ref{T:strong-logistic-source} (2).

\begin{proof}[Proof of Theorem~\ref{T:strong-logistic-source} (2)]
  It follows directly from~\eqref{E:local-infty} and Theorem~\ref{T:strong-logistic-source} (1).
\end{proof}

\section*{Acknowledgments}
\addcontentsline{toc}{section}{Acknowledgments}
Both L.~C. and I.~R. were partially supported by NSF grants DMS-2246850 and
DMS-2443823. L.~C. was also partially supported by a Collaboration Grant for
Mathematicians (\#959981) from the Simons Foundation.

% \section*{References}
% \addcontentsline{toc}{section}{References}
% % \bibliographystyle{alpha}
% \bibliographystyle{amsrefs}
% \bibliography{All}
\begingroup
\let\section\subsection  % prevents a second "References" header
\addcontentsline{toc}{section}{References}
\bibliographystyle{amsrefs}
\bibliography{All}
\endgroup
\end{document}